\documentclass{amsart}
\pdfoutput=1

\allowdisplaybreaks
\usepackage{graphicx}

\usepackage{amsmath,amssymb}

\usepackage{MnSymbol}

\usepackage{enumerate}

\usepackage{array}
\newcolumntype{L}[1]{>{\raggedright\let\newline\\\arraybackslash\hspace{0pt}}m{#1}}
\newcolumntype{C}[1]{>{\centering\let\newline\\\arraybackslash\hspace{0pt}}m{#1}}
\newcolumntype{R}[1]{>{\raggedleft\let\newline\\\arraybackslash\hspace{0pt}}m{#1}}

\usepackage{mathtools}

\usepackage{pdflscape}

\usepackage{multirow}

\usepackage[refpage,noprefix]{nomencl}
\usepackage{multicol}

\renewcommand*{\nompreamble}{\begin{multicols}{4}}
\renewcommand*{\nompostamble}{\end{multicols}}
\makenomenclature
\makeatletter
\def\@@@nomenclature[#1]#2#3{%
 \def\@tempa{#2}\def\@tempb{#3}%
 \protected@write\@nomenclaturefile{}%
 {\string\nomenclatureentry{#1\nom@verb\@tempa @[{\nom@verb\@tempa}]%
 |nompageref{\begingroup\protect\nomeqref{\theequation}}}%
 {\thepage}}%
 \endgroup
 \@esphack}
\makeatother

\usepackage{tikz}
\usetikzlibrary{decorations.markings}
\usetikzlibrary{arrows}
\usepackage{tkz-berge}
\usepackage{tkz-euclide}
\usetkzobj{all} 

\usepackage{color}
\definecolor{darkgreen}{rgb}{0,0.7,0}

\usepackage{hyperref}
\hypersetup{colorlinks=true,citecolor=blue,linkcolor=blue}
\usepackage[capitalize]{cleveref}
\crefformat{equation}{(#2#1#3)}

\numberwithin{equation}{section}
\newtheorem{theorem}{Theorem}[section]
\newtheorem{proposition}[theorem]{Proposition}
\newtheorem{lemma}[theorem]{Lemma}
\newtheorem{corollary}[theorem]{Corollary}
\newtheorem{conjecture}[theorem]{Conjecture}

\numberwithin{figure}{section}
\numberwithin{table}{section}

\theoremstyle{definition}
\newtheorem{definition}[theorem]{Definition}
\newtheorem{remark}[theorem]{Remark}
\newtheorem{example}[theorem]{Example}

\DeclareMathOperator{\Span}{Span}

\DeclareMathOperator{\Fan}{Fan}
\DeclareMathOperator{\Cone}{Cone}

\newcommand{\set}[1]{{\lbrace #1 \rbrace}}

\newcommand{\br}[1]{{\left\langle #1 \right\rangle}}
\newcommand{\ck}{^{\vee}}
\newcommand{\integers}{\mathbb{Z}}

\newcommand{\cm}[3]{(#1\Vert#2)_{#3}}

\newcommand{\cmrarrow}{{\rightarrow}}
\newcommand{\cmlarrow}{{\leftarrow}}
\newcommand{\cmcircarrow}{{\circlearrowright}}
\newcommand{\cmr}[3]{(#1\cmrarrow#2)_{#3}}

\newcommand{\cml}[3]{(#1\cmlarrow#2)_{#3}}

\newcommand{\cmcirc}[3]{(#1\,\cmcircarrow\,#2)_{#3}}

\newcommand{\dist}{\operatorname{dist}}

\renewcommand{\d}{{\mathbf d}}

\newcommand{\g}{{\mathbf g}}
\newcommand{\aff}{\mathrm{aff}}
\newcommand{\fin}{\mathrm{fin}}
\newcommand{\re}{\mathrm{re}}
\newcommand{\reals}{\mathbb{R}}

\renewcommand{\th}{^\text{th}}

\newcommand{\A}{{\mathcal A}}
\newcommand{\DF}{{\mathcal {DF}}}
\newcommand{\cl}{\operatorname{cl}}
\newcommand{\F}{{\mathcal F}}
\newcommand{\adj}[2]{\operatorname{adj}_{#1}(#2)}

\newcommand{\eigenspace}[1]{U^{#1}}
\newcommand{\RSChar}{\Phi}
\newcommand{\RS}{\RSChar}
\newcommand{\RSre}{\RS^\re}
\newcommand{\RSpos}{\RS^+}
\newcommand{\RSneg}{\RS^-}
\newcommand{\RSfin}{\RS_\fin}
\newcommand{\RSfinpos}{\RSfin^+}

\newcommand{\SimplesChar}{\Pi}
\newcommand{\Simples}{\SimplesChar}

\newcommand{\RSTChar}{\Upsilon}
\newcommand{\RST}[1]{\RSTChar^{#1}}
\newcommand{\RSTfin}[1]{\RST{#1}_\fin}
\newcommand{\SimplesTChar}{\Xi}
\newcommand{\SimplesT}[1]{\SimplesTChar^{#1}}
\newcommand{\simpleT}{\beta}
\newcommand{\Supp}{\operatorname{Supp}_\SimplesChar}
\newcommand{\SuppT}{\operatorname{Supp}_\SimplesTChar}
\newcommand{\TravInfChar}{\Psi}
\newcommand{\TravInf}[1]{\TravInfChar^{#1}}
\newcommand{\proj}{\to}
\newcommand{\TravProj}[1]{\overrightarrow{\TravInfChar}^{#1}}
\newcommand{\inj}{\leftarrow}
\newcommand{\TravInj}[1]{\overleftarrow{\TravInfChar}^{#1}}
\newcommand{\TravRegChar}{\Omega}
\newcommand{\TravReg}[1]{\TravRegChar^{#1}}
\newcommand{\AP}[1]{\RS_{#1}}
\newcommand{\APre}[1]{\AP{#1}^\re}
\newcommand{\APTChar}{\Lambda}
\newcommand{\APT}[1]{\APTChar_{#1}}      
\newcommand{\APTre}[1]{\APT{#1}^\re}     

\newcommand{\dynkinradius}{.06cm}
\newcommand{\dynkinstep}{.5cm}
\newcommand{\dynkinlinesep}{.08cm}
\newcommand{\dynkinaffinedot}[4]{\fill[fill=red] (\dynkinstep*#1,\dynkinstep*#2) circle (\dynkinradius); \node at (\dynkinstep*#1,\dynkinstep*#2) [font=\tiny,#4] {$#3$};}
\newcommand{\dynkindot}[4]{\fill (\dynkinstep*#1,\dynkinstep*#2) circle (\dynkinradius); \node at (\dynkinstep*#1,\dynkinstep*#2) [font=\tiny,#4] {$#3$};}
\tikzset{
  line/.style={thin},
  dotline/.style={dotted},
  doubleline/.style={thin,double distance=\dynkinlinesep,postaction={decorate,decoration={markings,mark=at position 0.7 with {\arrow[line width=0.06cm]{angle 60}}}}},
  tripleline/.style={thin,double distance=\dynkinlinesep*1.5,postaction={decorate,decoration={markings,mark=at position 0.7 with {\arrow[line width=0.06cm]{angle 60}}}},postaction={draw}}
}

\newcommand{\dynkinline}[4]{\draw[line] (\dynkinstep*#1,\dynkinstep*#2) -- (\dynkinstep*#3,\dynkinstep*#4);}
\newcommand{\dynkindotline}[4]{\draw[dotline] (\dynkinstep*#1,\dynkinstep*#2) -- (\dynkinstep*#3,\dynkinstep*#4);}
\newcommand{\dynkindoubleline}[4]{\draw[doubleline] (\dynkinstep*#1,\dynkinstep*#2) -- (\dynkinstep*#3,\dynkinstep*#4);}
\newcommand{\dynkintripleline}[4]{\draw[tripleline] (\dynkinstep*#1,\dynkinstep*#2) -- (\dynkinstep*#3,\dynkinstep*#4);}
\newenvironment{dynkin}{\begin{tikzpicture}[baseline]}{\end{tikzpicture}}

\newcommand{\newword}[1]{\textbf{\emph{#1}}}

\author{Nathan Reading}
\author{Salvatore Stella}
\title{An affine almost positive roots model}
\address[N. Reading]{Department of Mathematics, North Carolina State University, Raleigh, NC, USA}
\address[S. Stella]{Department of Mathematics, University of Leicester, University Road, Leicester, LE1 7RH, UK}
\thanks{Nathan Reading was supported in part by NSF grants DMS-1101568 and DMS-1500949.\\ \indent Salvatore Stella was partially supported by NCSU and INdAM and ISF grant 1144/16.}

\begin{document}

\begin{abstract}
We generalize the almost positive roots model for cluster algebras from finite type to a uniform finite/affine type model.
We define a subset $\AP{c}$ of the root system and a compatibility degree on $\AP{c}$, given by a formula that is new even in finite type.
The clusters (maximal pairwise compatible sets of roots) define a complete fan $\Fan_c(\RS)$.
Equivalently, every vector has a unique cluster expansion.
We give a piecewise linear isomorphism from the subfan of $\Fan_c(\RS)$ induced by real roots to the $\g$-vector fan of the associated cluster algebra.
We show that $\AP{c}$ is the set of denominator vectors of the associated acyclic cluster algebra and conjecture that the compatibility degree also describes denominator vectors for non-acyclic initial seeds.
We extend results on exchangeability of roots to the affine case.
\end{abstract}

\maketitle

\setcounter{tocdepth}{1}
\tableofcontents

\vspace{-30pt}

\section{Introduction}\label{intro}  
One of the first major achievements in the structural theory of cluster algebras was the classification, by Fomin and Zelevinsky \cite{FoZe03a,FoZe03}, of cluster algebras of finite type (cluster algebras with finitely many seeds).  
The classification parallels the Cartan-Killing classification, and an acyclic cluster algebra of finite type is specified by a Cartan matrix of finite type, an orientation of its Dynkin diagram (which together specify an exchange matrix) and a choice of ``coefficients.''
The cluster variables in the cluster algebra are in bijection with the almost positive roots (roots that either are positive or are the negatives of simples) in the corresponding root system.
The bijection sends a cluster variable to its denominator vector (or \newword{$\d$-vector}), leading to a purely combinatorial model based on almost positive roots, generalizing the associahedron or Stasheff polytope.
The almost positive roots also model \cite{Ste13} the $\g$-vectors of cluster variables.

After Cartan matrices of finite type, the next simplest Cartan matrices are those of affine type.
An affine Cartan matrix together with an \emph{acyclic} orientation of its Dynkin diagram and a choice of coefficients specifies a cluster algebra of \newword{affine type}.
Except in rank $2$, these are precisely the cluster algebras whose ``growth rate'' is linear, in the sense of \cite{FeShTu12}.  
(Every $2\times2$ Cartan matrix gives rise to a cluster algebra that is finite or has linear growth.)
 Another intrinsic characterization of affine type is a consequence of \cite[Theorem~3.5]{Seven}:  
An $n\times n$ exchange matrix with $n\ge 3$ is of finite or affine type if and only if it is of finite mutation type and is mutation equivalent to an acyclic exchange matrix.

In this paper, we extend the almost positive roots model to cluster algebras of affine type.
In particular, we extend many of the main results in \cite{associahedra,FoZe03a,FoZe03} to affine type.
We define an \newword{affine generalized associahedron fan}, extending the normal fan of generalized associahedra from finite type to a uniform finite/affine type construction. 
We do not, at this time, have an affine analog of the generalized associahedra as polytopes.

From a certain point of view, one might say that the almost positive roots model has already been extended to affine type and beyond in the representation theory literature.
In that setting, the denominator vectors are dimension vectors of indecomposable rigid representation and the compatibility degree is the dimension of $\mathrm{Ext}^1$ between the representations \cite{Buan06,Buan05c,Caldero05}.
There is also an approach to denominator vectors in the surfaces model \cite[Section~6]{Fomin08}.
However, neither of these successful models is an almost positive roots model.
We provide such a model by first defining explicitly a set of roots $\AP{c}$ in terms of the action of a Coxeter element on the root system, and then defining a compatibility degree in terms of multiplication and addition of simple-root (or coroot) coordinates or elementary counting with supports.
In the representation theoretic and surface models, by contrast, the definitions are in terms of modules and $\textrm{Ext}^1$ or in terms of tagged arcs and intersection numbers.
Our model agrees with the surfaces model but only agrees with the representation theory model up to the notion of compatibility.
Details on the relationship between our model and these other models are found in remarks at the end of this introduction.
We now give details on our model and results.

In affine type, taking $\tau_c$ to be the usual deformation of $c$, the set $\AP{c}$ is the union of the $\tau_c$-orbits of the positive roots that do not have full support, together with the imaginary root $\delta$.
(See \cref{def:Phic,apSchur} for several other characterizations of $\AP{c}$.)

We give a new formula for the usual finite-type compatibility degree $\cm\alpha\beta c$ \cite{FoZe03,MRZ} and extend the formula to affine type.  Two roots in $\AP{c}$ are $c$-compatible if their compatibility degree is zero.
We give the affine versions of the key properties already known in finite type, including cardinalities of clusters (maximal sets of pairwise $c$-compatible roots in $\AP{c}$).
We also prove the existence of unique cluster expansions (\cref{clus exp thm}), or in other words, we show (\cref{fan thm}) that the nonnegative spans of clusters are the maximal cones of a complete fan $\Fan_c(\RS)$. 
Finally, we give the affine version of the characterization of exchangeability in terms of compatibility degree.

We now describe the connection to cluster algebras.
Suppose $B$ is an exchange matrix arising from a Cartan matrix $A$ of affine type and an acyclic orientation of the associated Dynkin diagram.
The orientation of the Dynkin diagram is encoded in~$B$ as a choice of signs of off-diagonal entries.
We write $\A_\bullet(B)$ for the principal-coefficients cluster algebra determined by $B$.
Let $c$ be the Coxeter element obtained by multiplying the simple reflections $S$ in an order such that $s_i$ precedes $s_j$ if $b_{ij}>0$.
The notation $\Fan_c^\re(\RS)$ denotes the subfan of $\Fan_c(\RS)$ consisting of cones spanned by clusters not containing~$\delta$, and $\nu_c$ is a piecewise linear map defined in \cref{g d sec}.

\begin{theorem}\label{nu thm}
Suppose that $B$ is an acyclic exchange matrix whose associated Cartan matrix is of affine type.
Let $\RS$ be the associated root system and $c$ the associated Coxeter element.  
\begin{enumerate}[\rm(1)]
\item The piecewise-linear map $\nu_c$ induces an isomorphism from $\Fan_c^\re(\RS)$ to the $\g$-vector fan of $\A_\bullet(B)$.  
\item The cluster complex of $\A_\bullet(B)$ is isomorphic to the simplicial complex underlying $\Fan_c^\re(\RS)$.
\end{enumerate}
\end{theorem}

The finite-type version of the theorem follows immediately from \cite[Theorem~8.1]{sortable} and \cite[Theorem~5.39]{framework}.
We prove \cref{nu thm} using results (proved here and in \cite{framework,afframe}) that use only the combinatorics of root systems and Coxeter groups.
In contrast, the following result requires a theorem (\cite[Proposition~9]{Rupel}, quoted here as \cref{nu d g}) that is proved using representation theory.
Given a seed $\Sigma$ (essentially a choice of $B$ and coefficients), write $\A_\Sigma$ for the cluster algebra determined by~$\Sigma$.

\begin{theorem}\label{denom thm}
Suppose $\Sigma$ is a seed with exchange matrix $B$ that is acyclic and whose associated Cartan matrix is of affine type.
Let $\RS$ be the associated root system and $c$ the associated Coxeter element.  
The map from cluster variables to $\d$-vectors is a bijection to $\AP{c}\setminus\{\delta\}$.
The nonnegative linear spans of $\d$-vectors of clusters in $\A_\Sigma$ are the maximal cones of a fan, which coincides with $\Fan_c^\re(\RS)$. 
\end{theorem}

Ceballos and Pilaud \cite[Corollary~3.3]{CP} showed that the almost positive roots model describes denominator vectors in all finite types (without requiring acyclicity).
We conjecture that the same is true in affine type.
We write $\d(x)$ or $\d_\Sigma(x)$ for the denominator vector of a cluster variable $x$ with respect to an initial seed $\Sigma$.

\begin{conjecture}\label{c p conj} 
Suppose $\Sigma$ is an acyclic seed of affine type with exchange matrix~$B$, associated root system $\RS$ and associated Coxeter element~$c$.  
Index the cluster variables of $\A_\Sigma$ as $x(\beta)$ for $\beta\in\AP{c}$ according to the bijection in \cref{denom thm}.
Given any seed $\Sigma'$ mutation-equivalent to $\Sigma$ and with the cluster in $\Sigma'$ indexed as~$x(\beta_1),\ldots,x(\beta_n)$ and given any $\beta\in\AP c$, we have
\[\d_{\Sigma'}(x(\beta))=\bigl(\cm{\beta_1}\beta c,\ldots,\cm{\beta_n}\beta c\bigr).\]
\end{conjecture}

After playing a crucial role in the classification of cluster algebras of finite type (as described above), $\d$-vectors have fallen somewhat out of favor, because $\g$-vectors appear to have nicer properties.
However, \cite[Corollary~3.3]{CP} and \cref{c p conj} exhibit a nice property of $\d$-vectors that appears to have no analog for $\g$-vectors.
In \cref{g d sec}, we discuss evidence for \cref{c p conj}, including the case of surfaces, many cases where $\Sigma'$ is also acyclic, and some computational evidence.

\begin{remark}[Applications]\label{app remark}  
The connections to cluster algebras mentioned above concern $\Fan_c^\re(\RS)$, which is only part of the complete fan $\Fan_c(\RS)$, but we expect the entire fan $\Fan_c(\RS)$ to be important to the theory of cluster algebras.
As part of work in progress, we have constructed cluster scattering diagrams (in the sense of \cite{GHKK}) of affine type, and proved that the fan defined by these cluster scattering diagrams coincides with $\nu_c\Fan_c(\RS)$.
Furthermore, we have proved that the mutation fan (in the sense of \cite{universal}) coincides with $\nu_c\Fan_c(\RS)$.
As a consequence, we expect to prove a conjectured description of affine universal geometric cluster algebras \cite[Conjecture~10.15]{universal}.
We also intend to study, with Jon McCammond, the surprising similarities between $\Fan_c(\RS)$ and a lattice constructed by McCammond and Sulway \cite{McSul} to prove longstanding conjectures about Euclidean Artin groups.
(Compare especially \cite[Table~1]{McSul} and \cite[Proposition~7.6]{McSul} with \cref{tab:type-by-type,cyclo}.)
\end{remark}

We now make some remarks on the connection between the current paper and the existing literature.
In particular, \cref{denom rem,compat ext rem,decomp rem} describe the relationship between our results and some related representation-theoretic results.

\begin{remark}[Related work on the affine case]\label{affine rem}
\cref{nu thm,denom thm} build affine $\g$-vector fans and $\d$-vector fans in terms of compatibility of roots in $\AP{c}$.  
Affine $\g$-vector fans were previously constructed in \cite{afframe} as doubled Cambrian fans.
In particular, cones were constructed by specifying their normal vectors.
The construction here gives $\g$-vectors directly (rather than by inverting a matrix) and provides a direct test for compatibility (whether two cluster variables are in a common cluster) and exchangeability (whether they are in adjacent clusters).
Affine $\d$-vectors were also constructed in \cite{Buan08b}, as we will discuss further in \cref{compat ext rem}.
Affine cluster algebras of rank $2$ are treated in \cite{CalZel,MusPro,SZ,Zel}.  
The affine types (except for a finite number of exceptional cases) can be approached through the surfaces model \cite[Section~6]{Fomin08} or the orbifolds model \cite{FeShTu12}.
Finally, \cite{RupSteWil} characterizes cluster variables of affine type as generalized minors.
\end{remark}

\begin{remark}[Schur roots] \label{denom rem} 
  Caldero and Keller \cite[Theorem~4]{Caldero05} showed, for any acyclic quiver, that the cluster character
  map \cite{CalderoChapoton} is a bijection from the set of indecomposable objects without self-extensions in the cluster category to the set of cluster variables.
They also showed \cite[Theorem~3]{Caldero05} that the dimension vector of the object is the $\d$-vector of the corresponding cluster variable.  
The dimension vectors of these indecomposable objects are the real \newword{Schur roots}.
Thus the first assertion in our \cref{denom thm} combines with the Caldero-Keller results to identify the roots $\AP{c}$ as follows:
\begin{corollary}\label{apS}  
  If $B$ is acyclic and skew-symmetric, with associated root system $\RS$ and Coxeter element $c$, then the positive real 
  roots in $\AP{c}$ are precisely the Schur roots for the corresponding quiver.
\end{corollary}
Our \cref{denom thm} uses \cref{nu d g}, which is proved in \cite{Rupel} using the same circle of ideas that appear in \cite{Caldero05} (generalized to the skew-symmetrizable setting).
It seems likely that one could use ideas from \cite{Caldero05,Rupel} to prove \cref{apS} directly, and then cite \cite[Theorem~3]{Caldero05} (or its generalization \cite[Proposition~5]{Rupel}) to prove \cref{denom thm}.
In any case, in light of \cref{apS}, it seems reasonable to call $\AP{c}$ the ``almost positive Schur roots''.
Comparison with \cite{IPT,Scherotzke} makes it clear that $\APT{c}=\APre{c}\cap\eigenspace{c}$ (\cref{def:Phic}) is the set of dimension vectors of the regular representations.
As usual, the deformed Coxeter element $\tau_c$ corresponds to the Auslander-Reiten translation (or its inverse, depending on conventions).
\end{remark}

\begin{remark}[Compatibility degree, $\textrm{Ext}^1$, and $\textrm{Hom}$]\label{compat ext rem}
One representation-theoretic notion of compatibility degree is the dimension of $\textrm{Ext}^1$ between the corresponding modules in the cluster category.
Comparing \cite[Theorems~3~and~4]{Caldero05} to our \cref{nu thm,denom thm} and comparing \cite[Lemma~3.2]{Scherotzke} or \cite[Lemma~5.3]{IPT} with our  \cref{delta c compat}, we see that our compatibility degree \emph{on distinct roots} is zero if and only if the $\mathrm{Ext}^1$ compatibility degree is zero.
In other words, our notion of $c$-compatibility (\cref{compat def}) agrees with the $\mathrm{Ext}^1$ notion of compatibility.
However, beyond that, the two notions of compatibility degree need not agree.
By \cite[Theorem~7.5]{Buan06}, the $\mathrm{Ext}^1$ compatibility degree of real roots $\alpha$ and $\beta$ is $1$ in both directions if and only if $\alpha$ and $\beta$ are \emph{real} $c$-exchangeable in the sense of \cref{def exchangeable}.
\cref{exchangeable} shows that our $c$-compatibility degree on real roots is 1 in both directions if and only if the two roots are $c$-exchangeable (as opposed to \emph{real} $c$-exchangeable).  
The notions of $c$-exchangeability and real $c$-exchangeability are distinct, as made clear in \cref{exchangeable}.

In \cite[Theorem~A]{Buan08b}, denominator vectors in affine type are described in terms of the dimensions of certain spaces of homomorphisms.
(This is not called a ``compatibility degree'' in \cite{Buan08b}, but perhaps should be.)
If \cref{c p conj} is true, then our $c$-compatibility degree constitutes a root-theoretic formulas for these dimensions.
\end{remark}

\begin{remark}[Cluster expansions and canonical decompositions]\label{decomp rem}
Since compatibility is the same in our setting and the representation-theoretic setting, our cluster expansions (\cref{clus exp def}) are analogous to the canonical decompositions (or generic decompositions) of \cite{Kac82}.
Thus, parts of \cref{cluster properties} correspond to results proved in \cite[Section~3]{Scherotzke} and \cite[Section~6]{IPT}.
The canonical decomposition fan constructed in \cite{IPT} coincides with $\Fan_c(\RS)$ within the span of the positive roots.  
\end{remark}

We conclude this introduction with a remark about some potentially simplifying choices that we did not make.

\begin{remark}
In constructing the $\d$-vector fan and $\g$-vector fan, we could have saved some complications by ignoring the imaginary root $\delta$.
However, the complete fan that we obtain by including $\delta$ is crucial to the future work mentioned in \cref{app remark}.
Similarly, if our goals were entirely combinatorial, we could have restricted our attention to the standard affine root systems (the affinizations of finite root systems). 
However, the full range of affine root systems is needed in order to model all cluster algebras of affine type.
Because we didn't make these simplifications, some results cannot be stated uniformly.
Rather, there are exceptions in one affine type:  Type $A_{2k}^{(2)}$ in Kac's notational system \cite[\S4.8]{Kac90}.
This is the unique type with three distinct root lengths (except when $k=1$, when there are two lengths with ratio $2$).
\end{remark}

\section{Background}\label{sec:background}

\subsection{Root systems and Coxeter groups}\label{sec:root_systems}
We assume familiarity with the most basic theory of Cartan matrices, root systems and Coxeter groups, and in this background section the goal is to establish terminology and notation and to recall some some of the less basic aspects of the theory.

Let $n$ be a positive integer.
A Cartan matrix is a symmetrizable integer matrix $A=[a_{ij}]_{1\le i,j\le n}$, with symmetrizing constants $d_i$ (so that $d_i a_{ij}=d_j a_{ji}$ for all $i,j$).

Let $V$ \nomenclature[v]{$V$}{ambient vector space} be a real vector space with basis $\Simples=\set{\alpha_1,\ldots,\alpha_n}$.\nomenclature[zzq]{$\Simples$}{simple roots}
The $\alpha_i$ are the \newword{simple roots}.\nomenclature[zza1]{$\alpha_i$}{simple root}
The \newword{simple co-roots} are $\alpha_i\ck= d_i^{-1} \alpha_i$ \nomenclature[zza2c]{$\alpha_i\ck$}{simple coroot} and $\Simples\ck$ is the set of simple co-roots. \nomenclature[zzqc]{$\Simples\ck$}{simple coroots}
Given $v\in V$ and $\alpha_i\in\Simples$, we write $[v:\alpha_i]$ \nomenclature[zzzz]{$[v:\alpha_i]$}{$i$-th coefficient of $v$ in the basis $\Simples$} for the coefficient of $\alpha_i$ when $v$ is expanded in the basis of simple roots.
Let $V^*$ \nomenclature[v_zzzzd]{$V^*$}{dual space to $V$} be the dual space to $V$ and let $\br{\,\cdot\,,\,\cdot\,}$ \nomenclature[zzzz]{$\br{\,\cdot\,,\,\cdot\,}$}{canonical pairing between $V^*$ and $V$} be the canonical pairing between $V^*$ and $V$. 
The \newword{fundamental weights} $\set{\rho_i:i=1,\ldots,n}$ \nomenclature[zzri]{$\rho_i$}{fundamental weight} are the basis of $V^*$ that is dual to the basis $\Simples\ck$ of co-roots.

The Cartan matrix $A$ encodes a symmetric bilinear form $K$ on $V$ defined by ${K(\alpha\ck_i, \alpha_j)=a_{ij}}$. 
\nomenclature[k]{$K$}{symmetric bilinear form on $V$}
For each $i=1,\ldots,n$, the \newword{simple reflection} $s_i$ \nomenclature[si]{$s_i$}{simple reflection} is defined on the basis $\Simples$ by $s_i(\alpha_j)=\alpha_j-K(\alpha\ck_i, \alpha_j) \alpha_i$.
Equivalently, $s_i(\alpha\ck_j)=\alpha\ck_j-K(\alpha\ck_j, \alpha_i) \alpha\ck_i$.
The \newword{Weyl group} $W$ \nomenclature[w]{$W$}{Weyl group} is the group generated by $S=\set{s_i:i=1,\ldots,n}$. \nomenclature[s]{$S$}{simple reflections} 
Each $s_i$ acts as a reflection with respect to $K$, and thus the action of $W$ preserves $K$.
As usual, the action of $W$ on $V^*$, dual to its action on $V$, is given by $\br{w\phi,v}=\br{\phi,w^{-1}v}$.

\begin{remark}\label{where things live}
Our convention places both roots and co-roots in the space $V$ and places weights and co-weights in the dual space $V^*$.
The standard Lie-theoretic setup places roots and weights in $V^*$ and co-roots and co-weights in $V$.
The approach used here matches the approach implied in \cite[Chapter~4]{Bj-Br} and the approach in \cite{framework,cyclic,typefree,afframe}.
Furthermore, while the scattering diagram construction in \cite{GHKK} is not phrased in terms of roots and weights, it is naturally rephrased in those terms, and this rephrasing also follows the convention of the present paper.
\end{remark}

The \newword{real roots} are the vectors $w\alpha_i\in V$ for $w\in W$ and $i=1,\ldots,n$, and the \newword{real co-roots} are the vectors $w\alpha_i\ck$.
The \newword{real root system} $\RSre$ \nomenclature[zzv3re]{$\RSre$}{real root system} is the set of all real roots.
The \newword{root system} is a larger set $\RS$, \nomenclature[zzv1]{$\RS$}{root system} containing $\RSre$ (strictly when $\RSre$ is infinite).
The set $\RS\setminus\RSre$ is the set of \newword{imaginary roots}.
We do not need the full generality of imaginary roots, so we do not define them here. 

The root system $\RS$ is a subset of the \newword{root lattice} (the lattice spanned by $\Simples$), and is the disjoint union of \newword{positive roots} $\RSpos=\set{\beta\in\RS:[\beta:\alpha_i]\ge0\text{ for }i=1,\ldots n}$ \nomenclature[zzv1zza]{$\RSpos$}{positive roots} and \newword{negative roots} $\RSneg=-\RSpos$.
For each real root $\beta$ there is a co-root $\beta\ck=\frac{2}{K(\beta,\beta)}\beta$.
There is a bijection $\beta\mapsto t_\beta$ between real positive roots and reflections in $W$ given by $t_\beta(v)=v-K(\beta\ck, v) \beta$\nomenclature[sitbeta]{$t_\beta$}{reflection with respect to $\beta$} for $v\in V$.

A \newword{Coxeter element} $c$ \nomenclature[c]{$c$}{Coxeter element $s_1\cdots s_n$} is the product of any permutation of~$S$.
Mostly, we fix~$c$ and assume that $\Simples$ has been indexed so that $c=s_1\cdots s_n$.
But sometimes we let~$c$ vary, usually without referring directly to any numbering of the simple reflections.
For $s\in S$, we say $s$ is \newword{initial in $c$} if $c$ has a reduced word whose first letter is $s$ and $s$ is \newword{final in $c$} if $c$ has a reduced word whose last letter is $s$.
A \newword{source-sink move} is the operation of replacing $c$ by the Coxeter element $scs$ for $s$ initial or final in~$c$.

%

Given a reduced word $s_1\cdots s_n$ for $c$, the \newword{Euler form} $E_c$ \nomenclature[ec]{$E_c$}{Euler form on $V$} is defined on the bases of simple roots and co-roots by 
\begin{equation}\label{Ec def}
E_c(\alpha\ck_i,\alpha_j)=\begin{cases}
a_{ij}&\text{if } i>j,\\
1&\text{if }i=j,\text{ or}\\
0&\text{if } i<j.
\end{cases}
\end{equation}
Since $a_{ij}=K(\alpha\ck_i,\alpha_j)$, for any $\alpha,\beta\in V$ we have
\begin{equation}\label{K Ec}
K(\alpha, \beta)=E_c(\alpha, \beta)+E_{c}(\beta, \alpha).
\end{equation}
The form $E_c$ depends on $c$ but is independent of the choice of reduced word for~$c$.
Some facts about it will be useful; the first is \cite[Lemma~3.3]{typefree}.

\begin{lemma}\label{Ec invariant}
If~$s$ is initial or final in $c$, then $E_c(\alpha,\beta)=E_{scs}(s\alpha,s\beta)$ for all $\alpha$ and~$\beta$ in $V$.
\end{lemma}

Applying \cref{Ec invariant} repeatedly with $s$ running backwards through a reduced word for $c$, we obtain the following fact:

\begin{lemma}\label{Ec c}
$E_c(\alpha,\beta)=E_c(c\alpha,c\beta)$ for all $\alpha,\beta\in V$.
\end{lemma}

We also check two more useful facts about $E_c$.
\begin{lemma}\label{Ec 1}
$E_c(\beta\ck,\beta)=1$ for all $\beta\in\RSre$.
\end{lemma}
\begin{proof}
Since $\beta\ck$ is a positive scaling of $\beta$ and $E_c$ is bilinear, we have $E_c(\beta,\beta\ck)=E_c(\beta\ck,\beta)$.
Thus \cref{K Ec} says that $E_c(\beta\ck,\beta)=\frac12K(\beta\ck,\beta)=1$ for $\beta\in\RSre$. 
\end{proof}

\begin{lemma}\label{Ec inv T}
$E_c(\alpha,\beta)=E_{c^{-1}}(\beta,\alpha)$ for all $\alpha,\beta\in V$.
\end{lemma}
\begin{proof}
For any simple co-root $\alpha_i\ck$ and simple root $\alpha_j$, we have
\[E_{c^{-1}}(\alpha_j,\alpha\ck_i)=\frac{K(\alpha_j,\alpha_j)}{K(\alpha_i,\alpha_i)}E_{c^{-1}}(\alpha_j\ck,\alpha_i)=\begin{cases}
\frac{K(\alpha_j,\alpha_j)}{K(\alpha_i,\alpha_i)}a_{ji}&\text{if } i>j,\\
1&\text{if }i=j,\text{ or}\\
0&\text{if } i<j,
\end{cases}
\]
which agrees with \cref{Ec def}.
\end{proof}

We now review, and slightly modify for the present purposes, a result of Howlett \cite[Theorem~2.1]{Howlett}.
The proof given here is from \cite{Howlett}.

We abuse the notation $E_c$ by allowing it to stand not only for the bilinear form as above, but also for the matrix giving that form in the basis of simple roots on the right and simple co-roots on the left. 
Thus $E_c$ is the $n\times n$ matrix whose $ij$-entry is~$a_{ij}$ if $i>j$, is $1$ if $i=j$, and is $0$ if $i<j$.
Similarly, $E_{c^{-1}}$ is the $n\times n$ matrix whose $ij$-entry is $0$ if $i>j$, is $1$ if $i=j$, and is $a_{ij}$ if $i<j$.

\begin{theorem}\label{Howlett's}
Given an arbitrary symmetrizable Cartan matrix $A$ and a Coxeter element $c=s_1\cdots s_n$, the matrix for $c$ in the basis of simple roots is $-E_{c^{-1}}^{-1}E_c$.
\end{theorem}
\begin{proof}
Throughout the proof, the symbol $1$ denotes an identity matrix of appropriate size and the symbol $0$ denotes a matrix of zeros of appropriate shape.
These sizes and shapes are clear from context.

For each $i$ from $1$ to $n$, let $\ell_i$ and $u_i$ be the row vectors such that row $i$ in $A$ is $\begin{bsmallmatrix}\ell_i&2&u_i\end{bsmallmatrix}$.
In particular, $\ell_i$ has $i-1$ entries and $u_i$ has $n-i$ entries.
Thus the $k\th$ row of $E_c$ is $\begin{bsmallmatrix}\ell_k&1&0\end{bsmallmatrix}$.
Similarly, the $k\th$ row of $E_{c^{-1}}$ is $\begin{bsmallmatrix}0&1&u_k\end{bsmallmatrix}$.
The matrix for the simple reflection $s_k$ in the basis of simple roots is 
$\begin{bsmallmatrix}
1&0&0\\
-\ell_k&-1&-u_k\\
0&0&1
\end{bsmallmatrix}$.

For each $i$ from $1$ to $n$, define $L_i$ to be the $i\times i$ matrix whose $k\th$ row is $\begin{bsmallmatrix}\ell_k&1&0\end{bsmallmatrix}$.
Define $U_{n-i}$ to be the $(n-i)\times(n-i)$ matrix whose $k\th$ row is $\begin{bsmallmatrix}0&1&u_{i+k}\end{bsmallmatrix}$
We now prove by induction on $i$ that for $i=0,1,\ldots,n$ the matrix $E_{c^{-1}}s_1\cdots s_i$ equals $\begin{bsmallmatrix}-L_i&0\\0&U_{n-i}\end{bsmallmatrix}$.
The base case, where $i=0$, says that $E_{c^{-1}}=U_n$, which is true by construction.
If~$i>0$ then by induction $E_{c^{-1}}s_1\cdots s_i$ is 
\[
\begin{bsmallmatrix}-L_{i-1}&0\\0&U_{n-i+1}\end{bsmallmatrix}s_i
=
\begin{bsmallmatrix}-L_{i-1}&0&0\\0&1&u_i\\0&0&U_{n-i}\end{bsmallmatrix}
\begin{bsmallmatrix}1&0&0\\-\ell_i&-1&-u_i\\0&0&1\end{bsmallmatrix}
=
\begin{bsmallmatrix}-L_{i-1}&0&0\\-\ell_i&-1&0\\0&0&U_{n-i}\end{bsmallmatrix}
=\begin{bsmallmatrix}-L_i&0\\0&U_{n-i}\end{bsmallmatrix}.
\]
This completes the inductive proof.
In particular, for $i=n$ we have $E_{c^{-1}}c=-L_n$, but by construction $L_n=E_c$, so $E_{c^{-1}}c=-E_c$ as desired.
\end{proof}

As a consequence of \cref{Howlett's}, we obtain two more facts about the form $E_c$. 
\begin{lemma}
  \label{Ec Ecinv} $E_c(\alpha,\beta)=-E_{c^{-1}}(\alpha,c\beta)$ for all $\alpha,\beta\in V$.
\end{lemma}
\begin{proof}
\cref{Howlett's} says that the matrix describing the form $E_c$ in the basis of simple roots on the right and the basis of simple co-roots on the left is $-E_{c^{-1}}c$, where $E_{c^{-1}}$ is the matrix describing the form $E_{c^{-1}}$ in the same bases and $c$ is the matrix describing the action of $c$ in the basis of simple roots (on both sides).
\end{proof}

\begin{lemma}\label{Ec K=0}
For $\alpha,\beta\in V$, if $K(\alpha,\beta)=0$, then $E_c(\alpha,\beta)=E_c(\alpha,c^{-1}\beta)$.
\end{lemma}
\begin{proof}
Since $K(\alpha,\beta)=0$, we have $E_c(\alpha,\beta)=-E_c(\beta,\alpha)$ by \cref{K Ec}.
The latter equals $-E_{c^{-1}}(\alpha,\beta)$ by \cref{Ec inv T}, which equals $E_c(\alpha,c^{-1}\beta)$ by \cref{Ec Ecinv}.
\end{proof}

A \newword{(standard) parabolic subgroup} $W'$ of a Coxeter group $W$ is a subgroup generated by the reflections in some subset $S'\subset S$.
The associated \newword{(standard) parabolic root subsystem} is the set $\RS'=\RS\cap\Span(\set{\alpha_i\in\RS:s_i \in S'})$.
This is a root system in its own right with simple roots $\set{\alpha_i:s_i\in S'}$.
Given a Coxeter element $c$ of $W$ and $S'\subseteq S$, the \newword{restriction} of~$c$ to $W'$ is the Coxeter element $c'$ of $W'$ obtained by taking a reduced word for $c$ and deleting all the letters in $S\setminus S'$.

\subsection{Affine type}\label{sec:aff}  
When the symmetric bilinear form $K$ is positive definite, $W$ is finite and $A$ is said to be of \newword{finite type}.
Otherwise, $W$ is infinite.

When $K$ is positive semidefinite and not positive definite and, for all $\Pi'\subsetneq\Pi$, the restriction of $K$ to $\Span\set{\alpha_i\in \Pi'}$ is positive definite, $A$ is of \newword{affine type}, and $W$ and $\RS$ are called \newword{affine}.
Both $W$ and $\RS$ are irreducible when they are affine. 
Details on affine root systems are found, for example, in \cite{Kac90,Macdonald}.
We break with the common practice of taking an affine root system to have rank $n+1$ and, instead, we continue to index simple roots by $\set{1,\dots,n}$.

When $A$ is of affine type, there exists an index $\aff\in\set{1,\ldots,n}$ \nomenclature[aff]{$\aff$}{index of the affine simple root} such that, writing $W_\fin$ \nomenclature[wfin]{$W_\fin$}{finite Weyl group associated to an affine Cartan matrix} for the parabolic subgroup of $W$ generated by $S_\fin=S\setminus\set{s_\aff}$,\nomenclature[s_fin]{$S_\fin$}{simple reflections associated to $W_\fin$} the group $W$ is isomorphic to a semidirect product of $W_\fin$ with the lattice generated by the simple co-roots $\set{\alpha_i\ck:i\neq \aff}$.
There may not be a unique choice of $\aff$, but the choices are equivalent up to diagram automorphisms, and we fix a choice.
We call $\alpha_\aff$ \nomenclature[zza3aff]{$\alpha_\aff$}{affine simple root} the \newword{affine simple root} and $s_\aff$ \nomenclature[saff]{$s_\aff$}{affine simple reflection} the \newword{affine simple reflection}.

We write $\Simples_\fin$ \nomenclature[zzqfin]{$\Simples_\fin$}{finite simple roots associated to an affine Cartan matrix} for $\Simples\setminus\set{\alpha_\aff}$ and $\RSfin$ \nomenclature[zzv5fin]{$\RSfin$}{finite root system associated to an affine Cartan matrix} for the corresponding parabolic root subsystem of $\RS$.
This is an indecomposable finite root system.
We write $V_\fin$ \nomenclature[vfin]{$V_\fin$}{subspace spanned by $\Simples_\fin$} for the subspace of $V$ spanned by $\Simples_\fin$.
We have $V=V_\fin\oplus\Span\alpha_\aff$. 

There is a standard construction (see for example~\cite[Proposition~2.1]{Macdonald}) that builds an affine root system from an indecomposible finite root system. 
We call a root system arising in this way a \newword{standard} affine root system.
The standard affine root systems are represented on Table Aff~1 of \cite[Chapter~4]{Kac90}.
(See also \cref{tab:type-by-type} of the present paper.)
Some affine root systems are not standard, but every affine root system $\RS$ is a \newword{rescaling} of a unique standard affine root system~$\RS'$.
That is, each root of $\RS$ is a positive scaling of a root in $\RS'$ and the Cartan matrices of $\RS$ and $\RS'$ define the same symmetric bilinear form $K$ on the space $V$.
The rescaling factors are necessarily constant on $W$-orbits of roots.

When $\RS$ is an affine root system, the form $K$ has a one-dimensional kernel, which contains a one-dimensional sublattice of the root lattice.
The nonzero elements of this sublattice are the imaginary roots.
Let $\delta$ \nomenclature[zzd]{$\delta$}{positive imaginary root closest to the origin} be the positive imaginary root closest to the origin.
Because $\delta$ is in the kernel of $K$, it is fixed by $W$.
Every real root in $\RS$ is a positive scaling of $\beta+k\delta$ for some $\beta\in\RSfin$ and $k\in\integers$.
If $\RS$ is a standard affine root system, then every root is $\beta+k\delta$ (with no scaling needed).

The imaginary root $\delta$ has strictly positive simple-root coordinates, so in particular its $\alpha_\aff$-coordinate $[\delta:\alpha_\aff]$ is positive.
In fact, in almost every root system of affine type, $[\delta:\alpha_\aff]=1$.  
The exception is type $A^{(2)}_{2k}$, where $[\delta:\alpha_\aff]=2$. 
The vector $\theta=\delta-[\delta:\alpha_\aff]\alpha_\aff$ \nomenclature[zzh]{$\theta$}{highest root or highest short root in $\RSfin$} is a positive root in $\RSfin$.
In the standard affine root systems $\theta$ is the highest root of $\RSfin$, but in other affine root systems, it is either the highest root or the highest short root \cite[Proposition~6.4]{Kac90}.  

The co-roots $\set{\beta\ck:\beta\in\RSre}$ are the real roots of the \newword{dual root system} $\RS\ck$ \nomenclature[zzv2zzc]{$\RS\ck$}{root system dual to $\RS$} (the root system associated to the transpose $A^T$ of the Cartan matrix~$A$.)
The Cartan matrices $A$ and $A^T$ define the same symmetric bilinear form $K$, so if $\RS$ is affine, then $\RS\ck$ is also affine.
Let $\delta\ck$ \nomenclature[zzdc]{$\delta\ck$}{positive imaginary co-root closest to the origin} be the positive imaginary root in $\RS\ck$ that is closest to zero.
This is a positive scaling of $\delta$.
Indeed, one can calculate that 
\begin{equation}\label{delta eq}
\delta\ck=\begin{cases}
\displaystyle\frac2{K(\alpha_\aff,\alpha_\aff)}\delta=\frac2{K(\theta,\theta)}\delta 	&\text{if $\RS$ is not of type $A^{(2)}_{2k}$, or}\\[12pt]
\displaystyle\frac1{K(\alpha_\aff,\alpha_\aff)}\delta=\frac4{K(\theta,\theta)}\delta		&\text{if $\RS$ is of type $A^{(2)}_{2k}$.}
\end{cases}
\end{equation}
The non-uniformity in \eqref{delta eq} arises for two reasons:
First, $\alpha_\aff=\delta-\theta$ in every case except type $A^{(2)}_{2k}$, where $\alpha_\aff=\frac12(\delta-\theta)$; and second, the index $\aff\ck$ for $\RS\ck$ can be taken to coincide with the index $\aff$ for $\RS$ in every type except type~$A^{(2)}_{2k}$.

\subsection{Coxeter elements in affine type}  \label{sec:cox aff}
The following proposition is known and not difficult in type $A^{(1)}_n$. It follows from \cite[Theorem~1.2]{BGP} in the other types.

\begin{proposition}\label{aff conj} 
If $W$ is an affine Weyl group not of type $A_{n-1}^{(1)}$, then any two Coxeter elements of $W$ are conjugate in $W$.
If $W$ is of type $A_{n-1}^{(1)}$, then there is one conjugacy class for each $k\in\set{1,\dots, n-1}$, represented by the Coxeter element $s_1\cdots s_n$ with $(s_1s_{k+1})^3=(s_ks_n)^3=1$ and $(s_is_{i+1})^3=1$ for $i\neq k,n$.
The conjugations can be carried out by a sequence of source-sink moves.
\end{proposition}

The following is part of \cite[Proposition~3.1]{afforb}.   
For an example related to this proposition, see \cite[Example~1.3]{afforb}.
\begin{proposition}\label{eigen}
Let $\RS$ be an affine root system and let $c$ be a Coxeter element.
\begin{enumerate}[\qquad \upshape (1)]
  \item \label{eigen 1}
    $c$ has eigenvalue $1$ with algebraic multiplicity $2$ and geometric multiplicity~$1$.
    The imaginary root $\delta$ is a $1$-eigenvector of $c$.
  \item \label{eigen 1 gen}
    There exists a unique generalized $1$-eigenvector $\gamma_c$ \nomenclature[zzcc]{$\gamma_c$}{$1$-eigenvector of $c$} contained in the subspace $V_\fin$ of $V$.
    (This means that $(c-1)\gamma_c=\delta$.)
  \item \label{eigen U}  
    $c$ has finite order on the hyperplane $\eigenspace{c}=\set{v\in V:K(\gamma_c,v)=0}$. \nomenclature[uc]{$\eigenspace{c}$}{subspace generated by eigenvectors of $c$}
\end{enumerate}
\end{proposition}
Denote by $\phi_c$ \nomenclature[zzv8c]{$\phi_c$}{$K(\gamma_c,\,\cdot\,)$} the element of $V^*$ defined by $\br{\phi_c,v}=K(\gamma_c,v)$ for all $v\in V$.
In view of \cref{eigen}(\ref{eigen U}) it is useful to know $\phi_c$ up to scaling.
The following lemma is a concatenation of \cite[Lemma~3.5]{afforb} and \cite[Lemma~3.6]{afforb}.

\begin{lemma}\label{phic sign}
  $\phi_c$ is a negative scalar times $\sum_{1\le i<j\le n}\bigl([\delta:\alpha_j]a_{ij}\rho_i-[\delta:\alpha_i]a_{ji}\rho_j\bigr)$.
\end{lemma}

We now summarize some results from \cite{afforb} concerning the $c$-orbits of roots in an affine root system.
Let $\RST{c}$ be $\RS\cap \eigenspace{c}$\nomenclature[zzuc]{$\RST{c}$}{$\RS\cap \eigenspace{c}$} and define $\RSTfin{c}$ to be $\RSfin\cap \eigenspace{c}$.\nomenclature[zzufinc]{$\RSTfin{c}$}{$\RSfin\cap \eigenspace{c}$}
Let $\SimplesT{c}_\fin$ \nomenclature[zzocfin]{$\SimplesT{c}_\fin$}{simple roots for $\RSTfin{c}$} be the unique simple system of $\RSTfin{c}$ such that $\SimplesT{c}\subset\RSpos$.
The following is a rephrasing of \cite[Proposition~4.4]{afforb}.

\begin{proposition}\label{Up details}
  For any affine root system $\RS$ and a Coxeter element $c$
  \begin{enumerate}[\qquad \upshape (1)]
  \item 
    \label{Upsilon full rank}
    $\RSTfin{c}$ is a finite root system of rank $n-2$.
  \item 
    \label{type A}
    The irreducible components  of $\RSTfin{c}$ are all of finite type $A$.  
  \item 
    \label{canon order}
    $\SimplesT{c}_\fin=\set{\beta_i}_{i=1}^{n-2}$ can be ordered so that either $c\beta_i = \beta_{i+1}$ or $c\beta_i\not\in\SimplesT{c}_\fin$.
  \item 
    \label{Omega indep}
    The order in \eqref{canon order} may not be unique but $\TravReg{c}=\set{\beta_1,t_{\beta_1}\beta_2,\ldots,t_{\beta_1}\cdots t_{\beta_{n-2}}\beta_{n-2}}$ is the same for any choice.
\nomenclature[zzzc1]{$\TravReg{c}$}{transversal of the finite $c$-orbits in $\AP{c}$}
  \end{enumerate}
\end{proposition}

The set $\RST{c}$ is not a root system in the usual sense, but in accordance with \cite[Theorem~2.7]{typefree} (which is a theorem of \cite{Deodhar} and \cite{Dyer}), it inherits from~$\RS$ a canonical system $\SimplesT{c}$ \nomenclature[zzoc]{$\SimplesT{c}$}{simple roots for $\RST{c}$}of simple roots. 
Specifically, $\SimplesT{c}$ is the unique minimal subset of $\RST{c}\cap\RSpos$ containing $\RST{c}\cap\RSpos$ in its nonnegative span.   
See also \cref{SuppT delta} and \cite[Remark~1.7]{afforb}.  

For a reduced word $c=s_1\dots s_n$, let $\psi^\proj_{c;j}=s_1\cdots s_{j-1}\alpha_j$ \nomenclature[zzyzzzzc1]{$\psi^\proj_{c;j}$}{$s_1\cdots s_{j-1}\alpha_j\in\TravProj{c}$}
and $\psi^\inj_{c;j}=s_n\cdots s_{j+1}\alpha_j$.  \nomenclature[zzyzzzzc2]{$\psi^\inj_{c;j}$}{$s_n\cdots s_{j+1}\alpha_j\in\TravInj{c}$}
Define the sets  $\TravProj{c}=\set{\psi^\proj_{c;j}:j=1,\ldots,n}$ and $\TravInj{c}=\set{\psi^\inj_{c;j}:j=1,\ldots,n}$, and let $\TravInf{c}=\TravProj{c}\cap\TravInj{c}$. 
\nomenclature[zzyzzzzc0]{$\TravInf{c}$}{$\TravProj{c}\cup\TravInj{c}$, transversal of infinite $c$-orbits in $\RS$}%
\nomenclature[zzyzzzzc01]{$\TravProj{c}$}{}%
\nomenclature[zzyzzzzc02]{$\TravInj{c}$}{}%
Since different reduced words for the same Coxeter element are related by a sequence of commutations of adjacent commuting letters, the sets $\TravProj{c}$, $\TravInj{c}$, and $\TravInf{c}$ depend only on $c$, not on the chosen reduced word.

Let $\TravReg{c}$ be as defined in \cref{Up details}.
For any $\beta\in\TravReg{c}$, there exists a smallest positive integer $\kappa(\beta)$ such that $\kappa(\beta)\delta-\beta$ is a root.
The following is \cite[Theorem~1.2]{afforb}.  (A similar result is \cite[Proposition~1.9]{Dlab76}, but the result from \cite{afforb} has additional details that are crucial for our purposes.  
See \cite[Remark~1.9]{afforb}.)

\begin{theorem}\label{aff c-orbits}
  Suppose $\RS$ is an affine root system and $c$ is a Coxeter element in the associated Weyl group $W$.   
\begin{enumerate}[\qquad \upshape (1)]
    \item	\label{infinite transversal}
      There are exactly $2n$ infinite $c$-orbits in $\RS$. 
      The set $\TravInf{c}$ is a transversal of these orbits.
    \item \label{criterion}
      The $c$-orbit of a root $\beta\in\RS$ is finite if and only if $\beta\in \eigenspace{c}$.
    \item \label{im orb}
      Every imaginary root is fixed by $c$.
    \item \label{finite transversal}
      For $\RS$ of rank $2$, there are no finite $c$-orbits of real roots.
      For larger rank, there are infinitely many finite $c$-orbits of real roots and the set $\set{\beta + m\cdot\kappa(\beta)\delta:\beta\in\TravReg{c},\,m\in\integers}$ is a transversal of them.
    \item	\label{pos neg}
      Each finite $c$-orbit contains either only positive roots or only negative roots. 
      In particular, the $c$-orbit of a real root $\beta + m\cdot\kappa(\beta)\delta$ for $\beta\in\TravReg{c}$ consists of positive roots if and only if $m\ge0$.    \item \label{finite transversal finite}
      A finite $c$-orbit intersects $\RSfinpos$ if and only if it intersects $\TravReg{c}$.
  \end{enumerate}		
\end{theorem}

\cref{tab:type-by-type} shows the roots $\beta_i$ (indexed as in \cref{Up details}) for standard affine root systems and a particular choice of~$c$.  
(This table also appears as \cite[Table~1]{afforb}.)
The types are named as in \cite[\S4.8]{Kac90}, except that, as mentioned earlier, we use $n$ as the rank of the root system in every case.
The Coxeter element $c=s_1\cdots s_n$ is described by the labeling of nodes in the second column.
In every case, $\alpha_\aff=\alpha_n$.
\begin{table}[bp]
  \centering
  \begin{tabular}{|c|c|c|l|}

    \hline

    \hspace{-0.25ex}Type\hspace{-0.25ex} & Diagram of $\RS$ & \hspace{-0.55ex}Diagram of $\RSTfin{c}$\hspace{-0.55ex} & \multicolumn{1}{c|}{Simple roots of $\RSTfin{c}$}\\

    \hline&&&\\[-11pt]
  \hline&&& \\[-9pt]

      $A^{(1)}_1$

      &

      \raisebox{4pt}{\begin{dynkin}
        \draw[thin,double distance=\dynkinlinesep,postaction={decorate,decoration={markings,mark=at position 0.15 with {\arrowreversed[line width=0.06cm,xshift=-1pt]{angle 60}},mark=at position 0.85 with {\arrow[line width=0.06cm,xshift=1pt]{angle 60}}}}](0,0) -- (1.6*\dynkinstep,0);
     \dynkindot{0}{0}{1}{below}
     \dynkinaffinedot{1.6}{0}{2}{below}
    \end{dynkin}}

      &

      &

      \\

    \hline

      $\begin{array}{c}A^{(1)}_{n-1}\\[-2pt] _{(n\ge3)}\\[-2pt] _{k\neq n}\end{array}$

      &

      \begin{dynkin}
     \dynkinline{1}{0.7}{0.5}{.7}
     \dynkinline{2}{0.7}{2.5}{.7}
     \dynkinline{1}{-0.7}{0.5}{-0.7}
     \dynkinline{2}{-0.7}{2.5}{-0.7}
     \dynkinline{0}{0}{0.5}{.7}
     \dynkinline{0}{0}{0.5}{-.7}
     \dynkindotline{1}{.7}{2}{.7}
     \dynkindotline{1}{-.7}{2}{-.7}
     \dynkinline{2.5}{.7}{3}{0}
     \dynkinline{2.5}{-.7}{3}{0}
     \dynkindot{0}{0}{1}{left}
     \dynkindot{0.5}{.7}{2}{above}
     \dynkindot{0.5}{-.7}{k+1}{below}
     \dynkindot{2.5}{.7}{k}{above}
     \dynkindot{2.5}{-.7}{n-1}{below}
     \dynkinaffinedot{3}{0}{n}{right}
    \end{dynkin}

    &

      \begin{dynkin}
    \dynkinline{1}{0.7}{2.4}{0.7}
    \dynkinline{3.6}{0.7}{4}{0.7}
    \dynkindotline{2.4}{0.7}{3.6}{0.7}
    \dynkindot{1}{0.7}{1}{below}
    \dynkindot{2}{0.7}{2}{below}
    \dynkindot{4}{0.7}{k-1}{below}

    \dynkinline{1}{-0.3}{2.4}{-0.3}
    \dynkinline{3.6}{-0.3}{4}{-0.3}
    \dynkindotline{2.4}{-0.3}{3.6}{-0.3}
    \dynkindot{1}{-0.3}{k}{below}
    \dynkindot{2}{-0.3}{k+1}{below}
    \dynkindot{4}{-0.3}{n-2}{below}
   \end{dynkin}

      &

      {\small$
      \begin{array}{l}
        \beta_j=\alpha_{j+1}
      \end{array}
      $}

      \\

    \hline

      $\begin{array}{c}B^{(1)}_{n-1}\\[-2pt] _{(n\ge4)}\end{array}$

      &

    \begin{dynkin}
     \dynkinline{0}{0}{0.5}{0}
    \dynkinline{1.5}{0}{2}{0}
     \dynkinline{2.5}{.7}{2}{0}
     \dynkinline{2.5}{-.7}{2}{0}
     \dynkindotline{0.5}{0}{1.5}{0}
     \dynkindoubleline{0}{0}{-1.25}{0}
     \dynkindot{2.5}{.7}{n-1}{right}
     \dynkinaffinedot{2.5}{-.7}{n}{right}
     \dynkindot{2}{0}{n-2}{right}
     \dynkindot{0}{0}{2}{below}
     \dynkindot{-1.25}{0}{1}{below}
    \end{dynkin}

      &

   \begin{dynkin}
    \dynkinline{1}{0.7}{2.4}{0.7}
    \dynkinline{3.6}{0.7}{4}{0.7}
    \dynkindotline{2.4}{0.7}{3.6}{0.7}
    \dynkindot{1}{0.7}{1}{below}
    \dynkindot{2}{0.7}{2}{below}
    \dynkindot{4}{0.7}{n-3}{below}
    \dynkindot{2.5}{-0.3}{n-2}{below}
   \end{dynkin}

      &

      {\small$
    \begin{array}{l}
          \beta_{n-2}=\sum_{i=1}^{n-1} \alpha_i\\
     \beta_j=\alpha_{j+1}
    \end{array}
   $}

      \\

    \hline

      $\begin{array}{c}C^{(1)}_{n-1}\\[-2pt] _{(n\ge3)}\end{array}$

      &

    \begin{dynkin}
    \dynkinline{1}{0}{1.5}{0}
   \dynkinline{2.5}{0}{3}{0}
     \dynkindoubleline{-0.25}{0}{1}{0}
     \dynkindotline{1.5}{0}{2.5}{0}
     \dynkindoubleline{4.25}{0}{3}{0}
     \dynkindot{-0.25}{0}{1}{below}
     \dynkindot{1}{0}{2}{below}
     \dynkindot{3}{0}{n-1}{below}
     \dynkinaffinedot{4.25}{0}{n}{below}
    \end{dynkin}

    &

   \begin{dynkin}
    \dynkinline{1}{0.2}{2.4}{0.2}
    \dynkinline{3.6}{0.2}{4}{0.2}
    \dynkindotline{2.4}{0.2}{3.6}{0.2}
    \dynkindot{1}{0.2}{1}{below}
    \dynkindot{2}{0.2}{2}{below}
    \dynkindot{4}{0.2}{n-2}{below}
   \end{dynkin}

      &

      {\small$
      \begin{array}{l}
        \beta_j=\alpha_{j+1}
      \end{array}
      $  }

      \\

    \hline&&&\\[-9pt]

      $\begin{array}{c}D^{(1)}_{n-1}\\[-2pt] _{(n\ge5)}\end{array}$

      &

      \begin{dynkin}
   \dynkinline{0.5}{0}{1}{0}
   \dynkinline{2.5}{0}{2}{0}
     \dynkinline{0}{.7}{0.5}{0}
     \dynkinline{0}{-.7}{0.5}{0}
     \dynkindotline{1}{0}{2}{0}
     \dynkinline{2.5}{0}{3}{.7}
     \dynkinline{2.5}{0}{3}{-.7}
     \dynkindot{0}{.7}{1}{left}
     \dynkindot{0}{-.7}{2}{left}
     \dynkindot{0.5}{0}{3}{left}
     \dynkindot{2.5}{0}{n-2}{right}
     \dynkindot{3}{.7}{n-1}{right}
     \dynkinaffinedot{3}{-.7}{n}{right}
    \end{dynkin}

      &

   \begin{dynkin}
    \dynkinline{1}{1.3}{2.4}{1.3}
    \dynkinline{3.6}{1.3}{4}{1.3}
    \dynkindotline{2.4}{1.3}{3.4}{1.3}
    \dynkindot{1}{1.3}{1}{below}
    \dynkindot{2}{1.3}{2}{below}
    \dynkindot{4}{1.3}{n-4}{below}
    \dynkindot{2.5}{0.3}{n-3}{below}
    \dynkindot{2.5}{-0.7}{n-2}{below}
   \end{dynkin}

      &

      {\small$
    \begin{array}{l}
          \beta_{n-3}=\alpha_1+\sum_{i=3}^{n-1} \alpha_i \\
     \beta_{n-2}=\alpha_2+\sum_{i=3}^{n-1} \alpha_i \\
     \beta_j=\alpha_{j+2}
    \end{array}
   $}

      \\[18pt]

    \hline

    $E^{(1)}_6$

      &
  \raisebox{-12pt}{
    \begin{dynkin}
     \dynkinline{0}{0}{4}{0}
     \dynkinline{2}{0}{2}{2}
     \dynkindot{0}{0}{3}{below}
     \dynkindot{1}{0}{4}{below}
     \dynkindot{2}{0}{5}{below}
     \dynkindot{3}{0}{6}{below}
     \dynkinaffinedot{4}{0}{7}{below}
     \dynkindot{2}{1}{2}{right}
     \dynkindot{2}{2}{1}{right}
    \end{dynkin}
  }
      &

   \begin{dynkin}
    \dynkinline{1}{1.3}{2}{1.3}
    \dynkinline{1}{0.3}{2}{0.3}
    \dynkindot{1}{1.3}{1}{below}
    \dynkindot{2}{1.3}{2}{below}
    \dynkindot{1}{0.3}{3}{below}
    \dynkindot{2}{0.3}{4}{below}
    \dynkindot{1.5}{-0.7}{5}{below}
   \end{dynkin}

   &

      {\small$
    \begin{array}{l}
     \beta_1= \alpha_{4} + \alpha_{5} \\
     \beta_2= \alpha_{1} + \alpha_{2} + \alpha_{5} + \alpha_{6} \\
     \beta_3= \alpha_{2} + \alpha_{5} \\
     \beta_4= \alpha_{3} + \alpha_{4} + \alpha_{5} + \alpha_{6} \\
     \beta_5= \alpha_{2} + \alpha_{4} + \alpha_{5} + \alpha_{6}
    \end{array}

   $}

    \\

    \hline

    $E^{(1)}_7$

      &

      \begin{dynkin}
     \dynkinline{0}{0}{6}{0}
     \dynkinline{3}{0}{3}{1}
     \dynkindot{0}{0}{2}{below}
     \dynkindot{1}{0}{3}{below}
     \dynkindot{2}{0}{4}{below}
     \dynkindot{3}{0}{5}{below}
     \dynkindot{4}{0}{6}{below}
     \dynkindot{5}{0}{7}{below}
     \dynkinaffinedot{6}{0}{8}{below}
     \dynkindot{3}{1}{1}{right}
    \end{dynkin}

      &

   \begin{dynkin}
       \dynkinline{1}{1.3}{3}{1.3}
        \dynkindot{1}{1.3}{1}{below}
        \dynkindot{2}{1.3}{2}{below}
        \dynkindot{3}{1.3}{3}{below}
        \dynkindot{1.5}{0.3}{4}{below}
        \dynkindot{2.5}{0.3}{5}{below}
        \dynkinline{1.5}{0.3}{2.5}{0.3}
        \dynkindot{2}{-0.7}{6}{below}
   \end{dynkin}

      &

      {\small$
    \begin{array}{l}
     \beta_1=\alpha_{4} + \alpha_{5}\\
     \beta_2=\alpha_{1} + \alpha_{5} + \alpha_{6}\\
          \beta_3=\sum_{i=2}^7 \alpha_{i} \\
          \beta_4=\sum_{i=3}^6 \alpha_{i} \\
          \beta_5=\alpha_{1} + \sum_{i=4}^7 \alpha_{i} \\
          \beta_6=\alpha_{1} + \alpha_{5} + \sum_{i=3}^7 \alpha_{i}
    \end{array}
   $}

    \\

    \hline

      $E^{(1)}_8$

      &

      \!\!\!\begin{dynkin}
     \dynkinline{1}{0}{8}{0}
     \dynkinline{3}{0}{3}{1}
     \dynkindot{1}{0}{2}{below}
     \dynkindot{2}{0}{3}{below}
     \dynkindot{3}{0}{4}{below}
     \dynkindot{4}{0}{5}{below}
     \dynkindot{5}{0}{6}{below}
     \dynkindot{3}{1}{1}{right}
     \dynkindot{6}{0}{7}{below}
     \dynkindot{7}{0}{8}{below}
     \dynkinaffinedot{8}{0}{9}{below}
    \end{dynkin}\!\!\!

      &

   \begin{dynkin}
     \dynkinline{1}{1.3}{4}{1.3}
        \dynkindot{1}{1.3}{1}{below}
        \dynkindot{2}{1.3}{2}{below}
        \dynkindot{3}{1.3}{3}{below}
        \dynkindot{4}{1.3}{4}{below}
        \dynkindot{2}{0.3}{5}{below}
        \dynkindot{3}{0.3}{6}{below}
        \dynkinline{2}{0.3}{3}{0.3}
        \dynkindot{2.5}{-0.7}{7}{below}
   \end{dynkin}

   &

  {\small $
    \begin{array}{l}
     \beta_1=\alpha_{3} + \alpha_{4} + \alpha_{5}\\
          \beta_2=\alpha_{1} + \sum_{i=4}^6 \alpha_{i} \\
          \beta_3=\sum_{i=2}^7 \alpha_{i} \\
          \beta_4=\alpha_{1} + \sum_{i=3}^8 \alpha_{i} \\
          \beta_5=\alpha_{1} + \alpha_{4} + \sum_{i=3}^7 \alpha_{i} \\
          \beta_6=\alpha_{4} + \alpha_{5} + \sum_{i=1}^8 \alpha_{i} \\
          \beta_7=\alpha_{4} + \sum_{i=3}^6 \alpha_{i} + \sum_{i=1}^8 \alpha_{i}\!\!\\
    \end{array}
   $}

    \\

    \hline

      $F^{(1)}_4$

      &

      \begin{dynkin}
     \dynkinline{4}{0}{2}{0}
     \dynkindoubleline{2}{0}{0.75}{0}
     \dynkinline{0.75}{0}{-0.250}{0}
     \dynkinaffinedot{4}{0}{5}{below}
     \dynkindot{3}{0}{4}{below}
     \dynkindot{2}{0}{3}{below}
     \dynkindot{0.75}{0}{2}{below}
     \dynkindot{-0.25}{0}{1}{below}
    \end{dynkin}

      &

   \begin{dynkin}
    \dynkinline{1}{0.7}{2}{0.7}
    \dynkindot{1}{0.7}{1}{below}
    \dynkindot{2}{0.7}{2}{below}
    \dynkindot{1.5}{-0.3}{3}{below}
   \end{dynkin}

      &

      {\small$
    \begin{array}{l}
     \beta_1=\alpha_2+\alpha_3 \\
     \beta_2=\alpha_1+\alpha_2+\alpha_3+\alpha_4 \\
     \beta_3=2\alpha_2+\alpha_3+\alpha_4
    \end{array}
   $}

      \\

    \hline

    $G^{(1)}_2$

      &

      \begin{dynkin}
     \dynkinline{2}{0}{1}{0}
     \dynkintripleline{1}{0}{-0.25}{0}
     \dynkinaffinedot{2}{0}{3}{below}
     \dynkindot{1}{0}{2}{below}
     \dynkindot{-0.25}{0}{1}{below}
    \end{dynkin}

    &

   \begin{dynkin}
    \dynkindot{1}{0.3}{1}{below}
   \end{dynkin}

   &

      {\small$
      \begin{array}{l}
        \beta_1=\alpha_1+\alpha_2
      \end{array}
      $ }

      \\

    \hline
  \end{tabular}

 \bigskip

  \caption{Standard affine root systems and their finite orbits}
  \label{tab:type-by-type}
\end{table}

The following is \cite[Proposition~4.3]{afforb}.  
\begin{proposition}\label{useful}
The orbits of roots in $\TravProj{c}$ are separated from the orbits of roots in $\TravInj{c}$ by the hyperplane $\eigenspace{c}$.
Specifically, $K(\gamma_c,\beta)>0$ for $\beta\in c^{m}\TravProj{c}$ and $m\in\mathbb{Z}$, while $K(\gamma_c,\beta)<0$ for $\beta\in c^{m}\TravInj{c}$ and $m\in\mathbb{Z}$.
\end{proposition}

\cref{Up details}\eqref{type A} implies that the irreducible components of~$\RST{c}$ are of affine type $A$.
The following is \cite[Proposition~4.6]{afforb}.

\begin{proposition}\label{c on tUp components}
The action of $c$ on each component of $\RST{c}$ is to rotate the Dynkin diagram of the component, taking each node to an adjacent node (or when the component has rank $2$, to transpose the two nodes of the Dynkin diagram).
\end{proposition}

We now prove some special facts about the form $E_c$ in affine type.

\begin{proposition}\label{E delta in tubes}
Suppose $\RS$ is of affine type and $\beta$ is a root contained in $\eigenspace{c}$.
Then $E_c(\beta\ck,\delta)=E_c(\delta\ck,\beta)=0$.
\end{proposition}
\begin{proof}
Since $E_c(\delta\ck,\delta)=E_c(\delta,\delta\ck)$ and also $K(\delta\ck,\delta)=0$, \cref{K Ec} implies $E_c(\delta\ck,\delta)=E_c(\delta,\delta\ck)=0$.
We check that $E_c(\delta\ck,\beta)=0$ whenever $\beta$ is a simple root of $\RST{c}$.  We apply \cref{Ec c} repeatedly and recall that $\delta$ is fixed by the action of $W$ to conclude that $E_c(\delta\ck,\beta')$ is constant for roots $\beta'$ in the $c$-orbit of $\beta$.
The sum over the $c$-orbit of $\beta$ is $\kappa(\beta)\delta$, so the sum of the terms $E_c(\delta\ck,\beta')$ over roots $\beta'$ in the $c$-orbit of $\beta$ is $\kappa(\beta)E_c(\delta\ck,\delta)=0$.
Thus each of these terms is zero, and in particular $E_c(\delta\ck,\beta)=0$.
By linearity $E_c(\delta\ck,\beta)=0$ for any $\beta\in \eigenspace{c}$, and since $K(\delta\ck,\beta)=0$ for any root~$\beta$, \cref{K Ec} implies that $E_c(\beta,\delta\ck)=0$ as well, and thus $E_c(\beta\ck,\delta)=0$ by linearity.
\end{proof}

\begin{proposition}\label{E in tubes}
  Suppose $\RS$ is of affine type and $\beta,\beta'$ are simple roots of the root subsystem $\RST{c}$ of $\RS$.
  Then 
  \begin{align}\label{Ec tube}
    E_c(\beta\ck,\beta')&=
    \left\lbrace\begin{array}{rl}
    1&\text{if }\beta'=\beta,\\
    -1&\text{if }\beta'=c^{-1}\beta,\text{ or}\\
    0&\text{otherwise.}
  \end{array}\right.
\end{align}
\end{proposition}
\begin{proof}
\cref{Ec 1} says that $E_c(\beta\ck,\beta)=1$.
\cref{Ec Ecinv} says that $E_c(\beta\ck,c^{-1}\beta)=-E_{c^{-1}}(\beta\ck,\beta)=-1$.
If $\beta$ and $\beta'$ are in different components of $\RST{c}$, then ${K(\beta\ck,\beta')=0}$ and thus \cref{Ec K=0} says that $E_c(\beta\ck,\beta')=E_c(\beta\ck,c^{-1}\beta')$.
But then $c^{-1}\beta'$ is in the same component as $\beta'$, so we can continue to show that $E_c(\beta\ck,c^{-1}\beta')=E_c(\beta\ck,c^{-2}\beta')$, and so forth until we conclude that $E_c(\beta\ck,\beta')=E_c(\beta\ck,c^k\beta')$ is constant as $k$ varies.
The sum over the $c$-orbit of $\beta'$ is $\kappa(\beta)\delta$, and since $E_c(\beta\ck,\delta)=0$ by \cref{E delta in tubes}, we see that $E_c(\beta\ck,\beta')=0$.

It remains to show that $E_c(\beta\ck,\beta')=0$ when $\beta'$ is in the $c$-orbit of $\beta$ but $\beta'\not\in\set{\beta,c^{-1}\beta}$.
The $c$-orbit of $\beta$ has finite size $k\ge2$.  
If $k=2$, then we are done, so assume $k>2$.
We argue by induction on $i=2,\ldots,k-1$ that $E_c(\beta\ck,c^{-i}\beta)=0$.
For the base case $i=2$, by replacing $c$ by $c^{-1}$ in the statement $E_c(\beta\ck,c^{-1}\beta)=-1$ which we already proved, we obtain $E_{c^{-1}}(\beta\ck,c\beta)=-1$.
Then \cref{Ec c} says that $E_{c^{-1}}(c^{-1}\beta\ck,\beta)=-1$.
Since $K(c^{-1}\beta\ck,\beta)=-1$, \cref{K Ec} implies that $E_{c^{-1}}(\beta,c^{-1}\beta\ck)=0$.
By bilinearity, $E_{c^{-1}}(\beta\ck,c^{-1}\beta)=0$ as well, so $E_c(\beta\ck,c^{-2}\beta)=0$ by \cref{Ec Ecinv}.
If $2<i\le k-1$ then by induction $E_c(\beta\ck,c^{-i+1}\beta)=0$.
Since $K(\beta\ck,c^{-i+1}\beta)=0$, we appeal to \cref{Ec K=0} to conclude that $E_c(\beta\ck,c^{-i}\beta)=E_c(\beta\ck,c^{-i+1}\beta)=0$.
\end{proof}

\section{The roots \texorpdfstring{$\AP{c}$}{Phi c}}\label{apS sec} 
\begin{definition}\label{def:Phic}
Suppose $\RS$ is an affine root system and $c$ is a Coxeter element.
We write $\APTre{c}=\set{c^k\beta:k\in\integers,\beta\in\TravReg{c}}$.\nomenclature[zzlcre]{$\APTre{c}$}{$\set{c^k\beta:k\in\integers,\beta\in\TravReg{c}}=\APre{c}\cap\RSpos\cap \eigenspace{c}$}
This is finite by \cref{aff c-orbits}.
We write $\APT{c}$ for $\APTre{c}\cup\set\delta$. \nomenclature[zzlc]{$\APT{c}$}{$\set{c^k\beta:k\in\integers,\beta\in\TravReg{c}}\cup\set\delta=\AP{c}\cap\RSpos\cap \eigenspace{c}$}  
We define  \nomenclature[zzv6c]{$\AP{c}$}{almost positive Schur roots}  \nomenclature[zzv7c]{$\APre{c}$}{almost positive real Schur roots}
\begin{align*}
\AP{c}&=-\Simples\cup(\RSpos\setminus \eigenspace{c})\cup\APT{c},\\
\APre{c}&=-\Simples\cup(\RSpos\setminus \eigenspace{c})\cup\APTre{c} = \AP{c}\setminus\set\delta.
\end{align*}
\end{definition}
Since $\TravReg{c}=\TravReg{c^{-1}}$ and $\eigenspace{c}=\eigenspace{c^{-1}}$,
we have the following proposition:
\begin{proposition}\label{Phic inv}
$\AP{c}=\AP{c^{-1}}$.
\end{proposition}
For any simple reflection $s\in S$ denote by $\alpha_s\in \Simples$ the simple root associated to $s$ and define a map $\sigma_s$ \nomenclature[zzss]{$\sigma_s$}{deformation of the simple reflection $s$}
on $-\Simples\cup\RSpos$ by 
		\[
			\sigma_s(\alpha)=
			\begin{cases}
				\alpha & \text{if }\alpha\in-\Simples\setminus\set{-\alpha_s}\\
				s(\alpha) & \text{otherwise}.
			\end{cases}
		\]

\begin{proposition}\label{sigma prop}
The map $\sigma_s$ is an involution on $-\Simples\cup\RSpos$.
If $s$ is initial or final in $c$, then $\sigma_s$ restricts to a bijection from $\AP{c}$ to $\AP{scs}$.
\end{proposition}

The proof of \cref{sigma prop} uses a few lemmas.
The first is \mbox{\cite[Corollary 3.3]{afforb}}.
\begin{lemma}\label{scs U}
If $s$ is initial or final in $c$, then $\eigenspace{scs}=s\eigenspace{c}$.
\end{lemma}

\begin{lemma}\label{tUp init fin}
If $s$ is initial or final in $c$, then $\alpha_s\not\in\RST{c}$.
\end{lemma}
\begin{proof}
Since $s$ is initial or final in $c$, either $c\alpha_s$ or $c^{-1}\alpha_s$ is a negative root, and thus \cref{aff c-orbits} lets us rule out the possibility that $\alpha_s\in\RST{c}$.
\end{proof}

\begin{lemma}\label{tUp simples}
If $s$ is initial or final in $c$, then the simple roots of $\RST{c}$ and of $\RST{scs}$ are related by $\SimplesT{scs}=s\,\SimplesT{c}$, or equivalently, $\SimplesT{scs}=\sigma_s\,\SimplesT{c}$.
\end{lemma}
\begin{proof}
\cref{scs U} says that $\eigenspace{scs}=s\eigenspace{c}$.
Since $\RST{c}$ is the set of roots in $\RS$ that are contained in $\eigenspace{c}$, and similarly for $\RST{scs}$, and since $\pm\alpha_s$ are the only roots that change sign under the action of $s$, and $\alpha_s\not\in\RST{c}$ by \cref{tUp init fin}, the simple roots $\SimplesT{scs}$ are obtained from the simple roots $\SimplesT{c}$ by the action of $s$.
Since all of these roots are positive, the action of $s$ on them corresponds to the action of $\sigma_s$.
\end{proof}

\begin{definition}\label{def Supp}
The \newword{support} $\Supp(\beta)$ \nomenclature[supp]{$\Supp(\beta)$}{support of $\beta$ as a root in $\RS$} of a root $\beta$ is the set of simple roots that contribute with non-zero coefficient to the simple root expansion of $\beta$.
The support $\Supp(R)$ of a set $R$ of roots is $\bigcup_{\beta\in R}\Supp(\beta)$.
A support is \newword{full} if it is the entire set $\Simples$ of simple roots.
\end{definition}

\begin{definition}\label{def SuppU}
Given a real root $\beta\in\RST{c}$, its \newword{tube support} $\SuppT(\beta)$ \nomenclature[suppzzo]{$\SuppT(\beta)$}{support of $\beta$ as a root in $\RST{c}$} is the support of $\beta$ as a root in $\RST{c}$ (the set of roots that appear with nonzero coefficients in the expansion of $\beta$ as a linear combination of roots in $\SimplesT{c}$).
The tube support of a set $R\subseteq\RST{c}$ of real roots is $\SuppT(R)=\bigcup_{\beta\in R}\SuppT(\beta)$.
A tube support~is \newword{component-full} if it contains the full set of simple roots in some component of~$\RST{c}$.  
\end{definition}

\begin{remark}\label{SuppT delta}
The simple roots $\SimplesT{c}$ of $\RST{c}$ may fail to be a basis for their linear span, precisely to the extent that $\RST{c}$ is reducible.
However, every real root in $\RST{c}$ is in some irreducible component of $\RST{c}$ and thus has a unique expansion as a linear combination of simple roots.
We do not consider $\SuppT(\delta)$, which is not well-defined when $\RST{c}$ is reducible.
For a set $R$ of real roots in $\RST{c}$, the tube support $\SuppT(R)$ may contain simple roots in multiple components of $\RST{c}$.
It is component-full if it contains the entire set of simple roots in at least one of them.
\end{remark}

\begin{lemma}\label{pos real in tubes} 
The set $\APTre{c}$ is the set of positive real roots of $\RST{c}$ whose tube support is not component-full.
\end{lemma}
\begin{proof}
This is immediate from \cref{Up details}\eqref{type A} and \cref{c on tUp components}.
\end{proof}

\begin{lemma}\label{sigma pos real in tubes}
If $s$ is initial or final in $c$, then $\APTre{scs}=\sigma_s(\APTre{c})$.
\end{lemma}
\begin{proof}
This is an immediate consequence of \cref{tUp simples,pos real in tubes}.
\end{proof}

\begin{proof}[Proof of \cref{sigma prop}]
The first assertion is immediate because $\alpha_s$ is the unique positive root whose sign is changed by $s$.
By the symmetry of swapping $c$ and $scs$, it is enough to show that $\sigma_s(\AP{c})=\AP{scs}$.
The map $\sigma_s$ fixes $-\Simples\setminus\set{-\alpha_s}$, swaps~$\pm\alpha_s$, and sends all positive roots, aside from $\alpha_s$, to positive roots.
It also fixes $\delta$.
\cref{scs U,pos real in tubes,sigma pos real in tubes} complete the proof.
\end{proof}

Define $\tau_c=\sigma_1\cdots\sigma_n$  \nomenclature[zztc]{$\tau_c$}{deformation $\sigma_1\cdots\sigma_n$ of the Coxeter element $c$}
where $\sigma_i$ \nomenclature[zzsi]{$\sigma_i$}{$\sigma_{s_i}$} is an abbreviation for $\sigma_{s_i}$.   
\begin{proposition}\label{tau prop}
Suppose $\RS$ is an affine root system, $c=s_1\cdots s_n$ is a Coxeter element in the associated Weyl group $W$, and $\tau_c$ is $\sigma_1\cdots\sigma_n$ as above.
\begin{enumerate}[\qquad \upshape (1)]
\item \label{tauc perm}
The map $\tau_c$ restricts to a permutation of $\AP{c}$.
\item \label{tauc def}
For $\alpha\in\AP{c}$, 
\[\qquad\tau_c(\alpha)=
		\begin{cases}
			\psi^\proj_{c;i}\!	& \text{if } \alpha=-\alpha_i\\
			-\alpha_i \!\!& \text{if } \alpha=\psi^\inj_{c;i}\\
			c\alpha		&	\text{otherwise}
		\end{cases}
\qquad\text{and}\qquad 
\tau_c^{-1}(\alpha)=
		\begin{cases}
			\psi^\inj_{c;i}	& \text{if } \alpha=-\alpha_i\\
			-\alpha_i & \text{if } \alpha=\psi^\proj_{c;i}\\
			c^{-1}\alpha\!\!		&	\text{otherwise.}
		\end{cases}\]
\item\label{orbit count}
  There are $n$ infinite $\tau_c$-orbits  and $n-2$ finite $\tau_c$-orbits in $\APre{c}$.
\item \label{tauc traversal inf}
The set $-\Simples$ is a transversal of the infinite $\tau_c$-orbits in $\AP{c}$.
\item \label{tauc fin}
A root $\alpha\in\AP{c}$ is in a finite $\tau_c$-orbit if and only if $\alpha\in\RSpos\cap \eigenspace{c}$.
\item \label{tauc traversal fin}
$\TravReg{c}$ is a transversal of the finite $\tau_c$-orbits in $\APre{c}$.
\item\label{delta orbit}
$\set\delta$ is a $\tau_c$-orbit.
\item \label{tidily}
For all $i$, ${K(\gamma_c,\tau_c^m(-\alpha_i))>0}$ if ${m>0}$ and $K(\gamma_c,\tau_c^m(-\alpha_i))<0$ if $m<0$.
\end{enumerate}
\end{proposition}
\begin{proof}
The first assertion is obtained by applying \cref{sigma prop} $n$ times.
The expressions in \eqref{tauc def} are readily verified by inspection.
Assertions \eqref{orbit count}--\eqref{delta orbit} follow from \eqref{tauc def} and \cref{aff c-orbits}.
\cref{useful} and \eqref{tauc def} imply \eqref{tidily}.
\end{proof}

The properties of $\tau_c$-orbits given in \cref{tau prop} allow us to give several additional characterizations of $\AP{c}$.  

\begin{proposition}\label{apSchur}
Each of the following expressions specifies the set $\APre{c}$. 
(The symbol $\cupdot$ is disjoint union.)

\begin{enumerate}[\quad \upshape (1)]
\item \label{tauc easy}
$\set{\tau_c^k\beta:\beta\in(-\Simples\cup\TravReg{c}),k\in\integers}$.
\item \label{3cup}
$\set{c^{-m}\psi^\inj_{c;j}:m\ge0,\,1\le j\le n}\cupdot-\Simples\cupdot\set{c^m\psi^\proj_{c;j}:m\ge0,\,1\le j\le n}\cupdot\APTre{c}$.
\item \label{un inf fin}
The union of all finite $\tau_c$-orbits of roots $\beta\in\RSfinpos$ and all infinite $\tau_c$-orbits.
\item \label{support c}
$\set{c^k\beta\in\RS:\Supp(\beta)\subsetneq\Simples,k\in\integers}\cap(\RSpos\cup-\Simples)$.  \item \label{support}
$\set{\tau_c^k\beta\in\RSpos:\Supp(\beta)\subsetneq\Simples,k\in\integers}$.
\end{enumerate}
\end{proposition}
\begin{proof}
Characterization~\eqref{tauc easy} is a direct restatement of \cref{tau prop}(\ref{tauc traversal inf},\ref{tauc traversal fin}).
Characterization \eqref{3cup} also follows easily from \cref{tau prop}.

Observe that any root $\beta$ in $\RSpos\setminus\AP{c}$ has $\Supp(\beta)=\Simples$; indeed by \cref{aff c-orbits}(\ref{finite transversal},\ref{pos neg})  
 all the real roots $\beta$ in $\RSpos\setminus\AP{c}$ are of the form $\beta = c^k\beta'+m\cdot\kappa(\beta')\delta$ for some positive root~$c^k\beta'\in \APTre{c}$ (with $\beta'\in\TravReg{c}$) and $m>0$.
In particular $\Supp(\beta)\supseteq\Supp(\delta)=\Simples$ so that $\beta\not\in\RSfin$.
Characterization~\eqref{un inf fin} then follows immediately from the inclusion $\TravReg{c}\subsetneq\RSfin$.

By the same observation, a root $\beta\in\AP{c}$ is in a finite $c$- (or equivalently $\tau_c$-) orbit if and only if it is in the orbit of some $\beta'\in\TravReg{c}\subset\RSfin$.
Such a $\beta'$ has $\Supp(\beta')\subsetneq\Simples$, so for Characterizations~\eqref{support c}~and~\eqref{support} it remains only to consider infinite orbits.

The only roots in $\TravInf{c}$ that can possibly have full support are $s_1\cdots s_{n-1}\alpha_n$ and $s_n\cdots s_2\alpha_1$.
We have $cs_1\cdots s_{n-1}\alpha_n=-\alpha_n$ and $c^{-1}s_n\cdots s_2\alpha_1=-\alpha_1$, 
establishing Characterization~\eqref{support c}.
We also have $\tau_c^{2}s_1\cdots s_{n-1}\alpha_n=\alpha_n$ and $\tau_c^{-2}s_n\cdots s_2\alpha_1=\alpha_1$, proving Characterization~\eqref{support}.
\end{proof}

\begin{remark}
\label{finite ap}  
The set $\AP{c}$ is an affine version of the set $\RS_{\ge-1}$ of almost positive roots in a root system $\RS$ of finite type.
We deviate from the standard convention and denote the latter also by $\AP{c}$ even though this set does not depend on the choice of Coxeter element.
This notation allows us treat the finite and affine cases together in several definitions and proofs.
\end{remark}

\begin{proposition}\label{AP restrict} 
If $\RS$ is of affine type and $\RS'$ is a parabolic root subsystem of $\RS$, then $\AP{c}\cap\RS'=\RS'_{c'}$ where $c'$ is the restriction of $c$ to the parabolic subgroup~$W'$.
\end{proposition}
\begin{proof} 
If $\RS'$ is a proper parabolic root subsystem of $\RS$, then $\RS'$ is finite because $\RS$ is affine.
Thus the assertion follows from \cref{apSchur}\eqref{support}.
\end{proof}

\section{The compatibility degree}\label{compat sec}  
In this section, we define the compatibility degree for ordered pairs roots in $\AP{c}$ and establish its crucial properties.  
We need all of the notation of the previous sections.
Reminders on notation are found in the Index to Notation on page~\pageref{index}.

\begin{definition}\label{def adj}  
Suppose $\alpha\in\RST{c}$ is a root in the subsystem $\RS\cap \eigenspace{c}$ and suppose $\simpleT_j\in\SimplesT{c}$ (i.e.\ $\simpleT_j$ is a simple root of the subsystem).
Then $\simpleT_j$ is \newword{adjacent} to $\alpha$ if $\simpleT_j$ is in $\SuppT(c\alpha)\cup\SuppT(c^{-1}\alpha)$ but not in $\SuppT(\alpha)$.   
Given $\alpha,\beta\in\RST{c}$, define $\adj\alpha\beta$ \nomenclature[adj]{$\adj\alpha\beta$}{number of roots in $\SuppT(\beta)$ that are adjacent to $\alpha$} to be the number of roots of $\SimplesT{c}$ that are adjacent to $\alpha$ and contained in $\SuppT(\beta)$.
This number is not symmetric in $\alpha$ and $\beta$.
Define also  
\[
  \cmcirc\alpha\beta c=
  \left\lbrace\begin{array}{rl}
    -1&\text{if }\alpha=\beta,\\
    0&\text{if }\SuppT(\alpha)\subsetneq\SuppT(\beta)\text{ or }\SuppT(\beta)\subsetneq\SuppT(\alpha),\\
    \adj\alpha\beta&\text{otherwise.}
  \end{array}\right.
\]
\nomenclature[zzzzz]{$\cmcirc\alpha\beta c$}{expression for the compatibility degree on $\APTre{c}$}
\end{definition}

\begin{remark}
Recall from \cref{Up details}\eqref{type A} and from \cref{c on tUp components} that the irreducible components of~$\RST{c}$ are of affine type $A$ and that $c$ acts on a component by rotating its Dynkin diagram.
The tube support $\SuppT(\alpha)$ of a root $\alpha\in\RST{c}$ is a connected subgraph of one of these components.
Thus $\simpleT_j$ is adjacent to $\alpha$ if and only if $\simpleT_j$ is not in $\SuppT(\alpha)$ but is connected to $\SuppT(\alpha)$ by an edge in the Dynkin diagram of $\RST{c}$.
\end{remark}

We continue to denote by $[\beta:\alpha_i]$ the $\alpha_i$-coordinate of $\beta$ in the basis of simple roots and set $[\beta\ck:\alpha_i\ck]$ to be the $\alpha_i\ck$-coordinate of $\beta\ck$ in the basis of simple co-roots.
We write $[\beta:\alpha_i]_+$\nomenclature[zzzz]{$[x]_+$}{$\max(x,0)$} for $\max([\beta:\alpha_i],0)$ and $[\alpha\ck:\alpha_i\ck]_+$ for $\max([\alpha\ck:\alpha_i\ck],0)$.
For roots $\alpha$ and $\beta$ in $\AP{c}$, define  
\begin{align}
\cmr\alpha\beta c&=-\sum_{i=1}^n[\alpha\ck:\alpha_i\ck][\beta:\alpha_i]-\sum_{1\le j<i\le n}a_{ij}[\alpha\ck:\alpha_i\ck]_+[\beta:\alpha_j]_+\\
\cml\alpha\beta c&=-\sum_{i=1}^n[\alpha\ck:\alpha_i\ck][\beta:\alpha_i]-\sum_{1\le i< j\le n}a_{ij}[\alpha\ck:\alpha_i\ck]_+[\beta:\alpha_j]_+.
\end{align}
\nomenclature[zzzz]{$\cmr\alpha\beta c$}{part of the definition of compatibility degree}
\nomenclature[zzzz]{$\cml\alpha\beta c$}{part of the definition of compatibility degree}

\begin{definition}\label{compat deg def}
  For $\alpha$ and $\beta$ in $\AP{c}$, the $c$-\newword{compatibility degree} of $\alpha$ with $\beta$ is   
    \begin{align*}
    \cm\alpha\beta c&=\left\lbrace\begin{array}{ll}
\max\left[\cmr\alpha\beta c,\cml\alpha\beta c\right]
&\begin{array}{l}\text{except when }\alpha,\beta\in\APTre{c}\\\qquad\text{ and }\SuppT(\alpha,\beta)\text{ is component-full.}\end{array}
\\[10pt]  
\adj\alpha\beta&\begin{array}{l}\text{in that case.}\end{array}
\end{array}\right.
  \end{align*}
  \nomenclature[zzzz]{$\cm\alpha\beta c$}{compatibility degree}
\end{definition}
\begin{example}
  To help the reader parse the definition we compute the $c$-compatibility degree $\cm{2\alpha_1+\alpha_2}{\alpha_2}{c}$ in type $D_3^{(2)}$ for $c=s_1s_2s_3$.
  (Compare this example with \cite[Example 1.3]{afforb} and \cref{fanex}.)
  First observe that $2\alpha_1+\alpha_2$ is in an infinite $\tau_c$-orbit so we are in the first case of \cref{compat deg def}.
  Since $(2\alpha_1+\alpha_2)^\vee=\alpha_1^\vee+\alpha_2^\vee$ we have:
  \begin{align*}
    \cmr{2\alpha_1+\alpha_2}{\alpha_2}{c} = -(1 \cdot 0 + 1 \cdot 1 + 0 \cdot 0) - ( -1 \cdot 1 \cdot 0 + 0 \cdot 0 \cdot 0 - 2 \cdot 0 \cdot 0) =& -1\\
    \cml{2\alpha_1+\alpha_2}{\alpha_2}{c} = -(1 \cdot 0 + 1 \cdot 1 + 0 \cdot 0) - ( -2 \cdot 1 \cdot 1 + 0 \cdot 1 \cdot 0 - 1 \cdot 1 \cdot 0) =& 1 
  \end{align*}
  and we get $\cm{2\alpha_1+\alpha_2}{\alpha_2}{c}=1$.
  For further examples involving pairs of roots in $\eigenspace{c}$ we refer the reader to the proof of \cref{compat prop,compat prop c fixed}.
\end{example}
In the next two theorems, we establish the following properties and show that they uniquely characterize the compatibility degree.
\begin{align}
\label{compat base}
\cm{-\alpha_i}{\beta}c&=[\beta:\alpha_i]\text{ for }\alpha_i\text{ simple and }\beta\in\AP{c}.\\
\label{compat cobase}
\cm{\beta}{{-\alpha_i}}c&=[\beta\ck:\alpha_i\ck]\text{ for }\alpha_i\text{ simple and }\beta\in\AP{c}.\\
\label{compat U}		
\cm{\alpha}{\beta}c&=\cmcirc\alpha\beta c\text{ if }\alpha,\beta\in\APTre{c}.\\
\label{compat delta U}
\cm{\delta}{\alpha}c&=\cm{\alpha}{\delta}c=0\text{ if }\alpha\in\APT{c}.\\
\label{compat sigma}
\cm\alpha\beta c&=\cm{\sigma_s\alpha}{\sigma_s\beta}{scs} \text{ for }\alpha,\beta\in\AP{c}\text{ and }s\text{ initial or final in }c.\\
\label{compat tau} 
\cm\alpha\beta c&=\cm{\tau_c\alpha}{\tau_c\beta}c.
\end{align} 

\begin{theorem}\label{compat prop}  Fix a finite or affine root system $\RS$.
The assignment $(c,\alpha,\beta)\mapsto\cm\alpha\beta c$ is the unique function satisfying {\rm\cref{compat base}--\cref{compat sigma}}.
\end{theorem}

\begin{theorem}\label{compat prop c fixed}  
Fix a finite or affine root system $\RS$ and a Coxeter element~$c$.
The assignment $(\alpha,\beta)\mapsto\cm\alpha\beta c$ is the unique function satisfying {\rm\cref{compat base}--\cref{compat delta U}} and {\rm\cref{compat tau}}.
\end{theorem}

We now give two results on computing the compatibility degree for special pairs of roots (positive roots in $\AP{c}$ or roots in $\APTre{c}$).
The first of these results is immediate from the definitions:
\begin{lemma}\label{compat Ec}
  If $\alpha$ and $\beta$ are positive roots, then 
  \[
    \cmr\alpha\beta c=-E_c(\alpha\ck,\beta) 
    \qquad\text{and}\qquad
    \cml\alpha\beta c=-E_{c^{-1}}(\alpha\ck,\beta).
  \]
\end{lemma}

\cref{compat Ec} combines with \cref{E in tubes} (and the fact that each component of $\RST{c}$ is of affine type A) to give the following description of $\cmr\alpha\beta c$ and $\cml\alpha\beta c$ when $\alpha$ and $\beta$ are in $\APTre{c}$.
(That is, $\alpha$ and $\beta$ are positive roots in $\AP{c}$, contained in finite $\tau_c$-orbits.)

\begin{proposition}\label{compute in tubes}
If $\alpha,\beta\in\APTre{c}$, then
\begin{enumerate}[\rm(1)]
\item $\cmr\alpha\beta c$ is the number of roots $\beta_i\in\SuppT(\alpha)$ with $c\beta_i\in\SuppT(\beta)$ minus the number of roots in $\SuppT(\alpha)\cap\SuppT(\beta)$.
\item $\cml\alpha\beta c$ is the number of roots $\beta_i\in\SuppT(\alpha)$ with $c^{-1}\beta_i\in\SuppT(\beta)$ minus the number of roots in $\SuppT(\alpha)\cap\SuppT(\beta)$.
\end{enumerate}
\end{proposition}

We now prove \cref{compat prop,compat prop c fixed}.

\begin{proof}[Proof of \cref{compat prop,compat prop c fixed}]
Property~\eqref{compat base} records the simple observation that $\cmr{(-\alpha_i)}\beta c$ and $\cml{(-\alpha_i)}\beta c$ both equal $[\beta:\alpha_i]$.
Property~\eqref{compat cobase} is similar.

To verify Property~\eqref{compat U}, we need to show that $\max\left[\cmr\alpha\beta c,\cml\alpha\beta c\right]$ equals $\cmcirc\alpha\beta c$ when $\alpha,\beta\in\APTre{c}$, except possibly when $\SuppT(\alpha,\beta)$ is component-full.
This comes down to checking several cases and computing $\cmr\alpha\beta c$ and $\cml\alpha\beta c$ as described in \cref{compute in tubes}.
If $\alpha$ and $\beta$ are in different components of $\RST{c}$, then $\cmr\alpha{\beta}c=\cml\alpha{\beta}c=0$, so their maximum is $0$ as desired.
The remaining cases are described in \cref{tubes table}, with representative pictures.
In each case, $\SuppT(\alpha)$ is outlined in red and $\SuppT(\beta)$ is outlined in dotted blue.
The quantity $\adj\alpha\beta$ is left out of the table when it is irrelevant.
\begin{table}
\centering
\begin{tabular}{|C{66pt}|c|c|c|c|c|c|}\hline
\textbf{\small Description}&\textbf{\small Pictures}&\small\!$\cmr\alpha\beta c$\!&\small\!$\cml\alpha\beta c$\!&\small\!$\adj\alpha\beta\!$&\small$\cmcirc\alpha\beta c$\!&\small$\cm\alpha\beta c$\\\hline
\small$\alpha=\beta$&\raisebox{-28pt}{\includegraphics{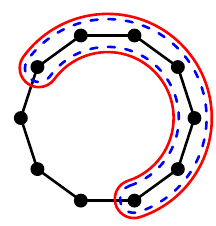}}&-1&-1&&-1&-1\\\hline
\small$\alpha$ and $\beta$ disjoint, not adjacent&\raisebox{-28pt}{\includegraphics{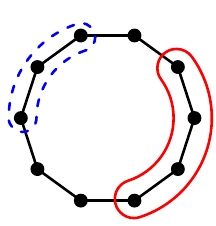}}&0&0&0&0&0\\\hline
\small
$\SuppT(\alpha)$ and $\SuppT(\beta)$ nested&\raisebox{-28pt}{\includegraphics{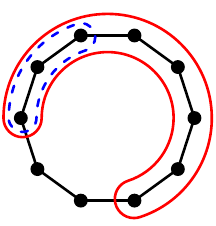}}&0&-1&&0&0\\\hline
$\alpha$ and $\beta$ overlapping or adjacent on one side&\raisebox{-28pt}{\includegraphics{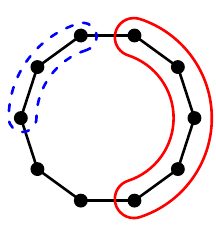}}&1&0&1&1&1\\\hline
\end{tabular}\\[10pt]
  \caption{Compatibility degree for positive real roots in $\eigenspace{c}$}
\label{tubes table}
\end{table}  
The action of $c$ is to rotate counterclockwise by one position.
Some of the cases have variations not pictured.
For example, in the case where $\SuppT(\alpha)$ and $\SuppT(\beta)$ are nested, their tube supports can have one or the other endpoint in common (or neither, but not both).
In these variations, the values of $\cmr\alpha\beta c$ and $\cml\alpha\beta c$ can vary, but their max does not vary and by definition $\cmcirc\alpha\beta c$ does not vary.
In the cases where $\SuppT(\alpha)$ and $\SuppT(\beta)$ are nested, $\cmr\alpha\beta c$ and $\cml\alpha\beta c$ may be $0$ or $-1$, but at least one of them is $0$.

Property~\eqref{compat delta U} holds by \cref{E delta in tubes,compat Ec}.

To show that \cref{compat sigma} holds, first, consider the case where one of $\alpha$ and $\beta$ (or both) is in $-\Simples\setminus\set{-\alpha_s}$. 
In particular, we know that the exceptional case of the definition of $\cm\alpha\beta c$ does not apply, so we have to prove that the maximum of $\cmr\alpha\beta c$ and $\cml\alpha\beta c$ equals the maximum of $\cmr{\sigma_s\alpha}{\sigma_s\beta}{scs}$ and $\cml{\sigma_s\alpha}{\sigma_s\beta}{scs}$.

If $\alpha=-\alpha_i$, then both $\cmr\alpha\beta c$ and $\cml\alpha\beta c$ equal $[\beta:\alpha_i]$.
But $\sigma_s\alpha=\alpha=-\alpha_i$ so both $\cmr{\sigma_s\alpha}{\sigma_s\beta}{scs}$ and $\cml{\sigma_s\alpha}{\sigma_s\beta}{scs}$ equal $[\sigma_s\beta:\alpha_i]$.
Since $\sigma_s\beta=\beta+a\alpha_s$ for some $a$ we have $[\sigma_s\beta:\alpha_i]=[\beta:\alpha_i]$.
The case where $\beta=-\alpha_i$ is similar.

We can now assume that neither $\alpha$ nor $\beta$ is in $-\Simples\setminus\set{-\alpha_s}$.
In particular $\sigma_s$ acts on $\alpha$ and $\beta$ as $s$, so \cref{Ec invariant,compat Ec} combine to prove \cref{compat sigma} except in the case where $\alpha,\beta\in\APTre{c}$ and $\SuppT(\alpha,\beta)$ is component-full.
\cref{scs U} says that $\eigenspace{scs}=s\eigenspace{c}$.
Furthermore, \cref{tUp simples} says that the simple roots of $\RST{scs}$ are obtained from the simple roots of $\RST{c}$ by the action of $s$.
Thus $s$ also takes $\SuppT(\alpha)$ and $\SuppT(\beta)$ (with respect to $c$) to $\SuppT(s\alpha)$ and $\SuppT(s\beta)$ (with respect to $scs$).
Therefore $\adj{s\alpha}{s\beta}$, with respect to $scs$, equals $\adj\alpha\beta$, with respect to $c$.

We have verified \cref{compat sigma} and now \cref{compat tau} follows by repeated applications of \cref{compat sigma}.
It remains only to verify the uniqueness statements in \cref{compat prop,compat prop c fixed}.
Together, Properties \eqref{compat U} and \eqref{compat delta U} specify the compatibility degree for all pairs of roots in~$\APT{c}$.
By \cref{tau prop}, we see that Properties \eqref{compat base}, \eqref{compat cobase}, and \eqref{compat tau} completely specify the compatibility degree on all other pairs of roots, so we have proved the uniqueness in \cref{compat prop c fixed}.
But if a function on pairs of roots and Coxeter elements satisfies \eqref{compat sigma}, then it satisfies \eqref{compat tau} for each Coxeter element $c$, and thus it is uniquely determined by Properties \eqref{compat base} and \eqref{compat cobase} for each $c$.
That is the uniqueness assertion of \cref{compat prop}.
\end{proof}

We now gather some properties of the compatibility degree.

The usual notion of compatibility degree for $\RS$ of finite type, defined in \cite{FoZe03,MRZ}, is the unique function satisfying \eqref{compat base} and \eqref{compat tau}.
See also \cite{sortable,shih, Ste13}.
(We have followed Ceballos and Pilaud \cite[Remark~2.10]{CP} in modifying the usual notion by taking $\cm\alpha\alpha c=-1$ rather than $0$ for $\alpha\in\AP{c}$.)

Thus the following is an immediate corollary of \cref{compat prop c fixed}.
\begin{corollary}\label{compat usual}
If $\RS$ is of finite type, then $\cm\alpha\beta c$ agrees with the usual compatibility degree on almost positive roots, modified to set $\cm\alpha\alpha c=-1$.
Thus, the usual compatibility degree is $\cm\alpha\beta c=\max\left[\cmr\alpha\beta c,\cml\alpha\beta c\right]$.
\end{corollary}

\begin{proposition}\label{c cinv prop}
  For all $\alpha,\beta\in\AP{c}$ we have $\cm\alpha\beta c=\cm\alpha\beta{c^{-1}} $.
\end{proposition}
\begin{proof}
We can write $\cm\alpha\beta{c^{-1}}$ because $\AP{c}=\AP{c^{-1}}$.
(See \cref{Phic inv}.) 
Replacing~$c$ with $c^{-1}$ swaps $\cmr\alpha\beta c$ with $\cml\alpha\beta c$.
Since $\APTre{c} =\APTre{c^{-1}}$, tube supports and the quantity $\adj\alpha\beta$ are the same with respect to $c$ and $c^{-1}$.
\end{proof}

In the following proposition, the quantity $\cm{\alpha\ck}{\beta\ck}c$ means the compatibility degree of $\alpha\ck$ and $\beta\ck$ in the root system $\RS\ck$.
Since $\RS$ and $\RS\ck$ define the same Weyl group, it makes sense to consider the same $c$ in both contexts.
Furthermore, it is apparent (for example by \cref{apSchur}) that $\RS\ck_c=\set{\beta\ck:\beta\in\AP{c}}$.

\begin{proposition}\label{compat commute} 
If $\alpha,\beta\in\AP{c}$, then $\cm\alpha\beta c=\cm{\beta\ck}{\alpha\ck}c$ except possibly when $\alpha,\beta\in\APTre{c}$ and $\SuppT(\alpha,\beta)$ is component-full, in which case $\set{\cm\alpha\beta c,\cm{\beta\ck}{\alpha\ck}c}\subseteq\set{1,2}$.
\end{proposition}
\begin{proof}
Except when $\alpha,\beta\in\APTre{c}$ and $\SuppT(\alpha,\beta)$ is component-full, this is immediate from the definitions of $\cmr\alpha\beta c$ and $\cml\alpha\beta c$.
When $\alpha,\beta\in\APTre{c}$ and $\SuppT(\alpha,\beta)$ is component-full, $\SuppT(\beta\ck,\alpha\ck)$ is also component-full.
However, it is possible that $\adj{\alpha}{\beta}=1$ while $\adj{\beta\ck}{\alpha\ck}=2$ or vice-versa.
\end{proof}

We now describe how the compatibility degree changes under rescaling.

\begin{proposition}\label{prop:scaling}  
Suppose $\RS$ and $\RS'$ are related by the rescaling $\beta\mapsto\beta'=\lambda_\beta\beta$.
Fix a Coxeter element~$c$ of their common Weyl group.
The rescaling map restricts to a bijection from $\AP{c}$ to $\RS'_c$.
For $\alpha,\beta\in\AP{c}$, the compatibility degree $\cm{\alpha'}{\beta'}c$, computed in $\RS'$ equals $f\cdot\cm\alpha\beta c$, where $\cm\alpha\beta c$ is the compatibility degree computed in $\RS$ and $f$ is some positive scalar. 
The scalar $f$ is $\frac{\lambda_\beta}{\lambda_\alpha}$, except possibly when $\alpha=\delta$.
\end{proposition}
\begin{proof}
\cref{apSchur} implies that $\beta\mapsto\beta'$ restricts to a bijection from $\AP{c}$ to~$\RS'_c$.
Scaling roots does not affect membership in $\eigenspace{c}$.
If $\alpha$ and $\beta$ are in $\RST{c}$ but in different components or if they are in the same component and have nested tube support, then $\cm\alpha\beta c$ and $\cm{\alpha'}{\beta'}c$ are both zero.
Now suppose $\alpha$ and $\beta$ are in the same component of $\RST{c}$ but do not have nested tube supports.
The two compatibility degrees are $\adj{\alpha'}{\beta'}$ and $\adj{\alpha}{\beta}$.
Rescaling does not affect tube supports, so $\adj{\alpha'}{\beta'}=\adj\alpha\beta$.
Since $\alpha$ and $\beta$ are of the same length and $\alpha'$ and $\beta'$ are of the same length, we have $\frac{\lambda_\beta}{\lambda_\alpha}=1$, and we are done in the case where $\alpha,\beta\in\APTre{c}$.

We next show that $\cmr\alpha\beta c$ and $\cml\alpha\beta c$ change as desired.
When $\alpha$ and $\beta$ are positive, by \cref{compat Ec} and the fact that $E_c$ is bilinear, this amounts to showing that $E_c$ is the same bilinear form whether it is defined in terms of $\RS$ or $\RS'$.
The real co-roots in $\RS$ and $\RS'$ are related by 
\begin{equation}\label{lambda ck}
(\beta')\ck=\frac{2}{K(\beta',\beta')}\beta'=\frac{2\lambda_\beta}{(\lambda_\beta)^2K(\beta,\beta)}\beta=\frac1{\lambda_\beta}\beta\ck.
\end{equation}
Using the definition \eqref{Ec def} in terms of $\RS$, we calculate $E_c((\alpha'_i)\ck,\alpha'_j)$ on simple roots/co-roots of $\RS'$ to be
\[
\frac{\lambda_j}{\lambda_i}E_c(\alpha_i\ck,\alpha_j)
=\left\lbrace\begin{array}{ll}
\frac{\lambda_j}{\lambda_i}K(\alpha_i\ck,\alpha_j)=K((\alpha'_i)\ck,\alpha'_j)&\text{if } i>j,\\
1&\text{if }i=j,\\
0&\text{if } i<j,
\end{array}\right.
\]
where $\lambda_i$ is shorthand for $\lambda_{\alpha_i}$.
This agrees with the definition of $E_c((\alpha'_i)\ck,\alpha'_j)$ in terms of $\RS'$.
In light of \eqref{lambda ck}, we see that $\cm{\alpha'}{\beta'}c=\frac{\lambda_\beta}{\lambda_\alpha}\cm\alpha\beta c$, except possibly when $\alpha=\delta$.
Furthermore, when $\alpha=\delta$, \eqref{delta eq} implies that $\cm{\alpha'}{\beta'}c=f\cdot\cm\alpha\beta c$ for some $f>0$, but there is no guarantee that $f=\frac{\lambda_\beta}{\lambda_\alpha}$.

If $\alpha=-\alpha_j$ for some~$j$, then \eqref{compat base} says that $\cm\alpha\beta c=[\beta:\alpha_j]$ and $\cm{\alpha'}{\beta'} c=[\beta':\alpha'_j]$.
The latter is equal to $\frac{\lambda_\beta}{\lambda_\alpha}[\beta:\alpha_j]$.
If $\beta=-\alpha_k$, then \eqref{compat cobase} says that $\cm\alpha\beta c=[\alpha\ck:\alpha_k\ck]$ and $\cm{\alpha'}{\beta '}c=[(\alpha')\ck:(\alpha_k')\ck]$.
If $\alpha$ is real, then \eqref{lambda ck} implies that $[(\alpha')\ck:(\alpha_k')\ck]=\frac{\lambda_\beta}{\lambda_\alpha}[\alpha\ck:\alpha_k\ck]$.
If $\alpha=\delta$, then we use \eqref{delta eq} again to conclude that $\cm{\alpha'}{\beta'}c=f\cdot\cm\alpha\beta c$ for some positive $f$.
\end{proof}

\begin{proposition}\label{compatibility almost symmetrizable}
For $\alpha,\beta\in\AP{c}$, the quantities $\cm\beta\alpha c$ and $\cm\alpha\beta c$ are related by a positive scaling depending on $\alpha$ and $\beta$.
When $\alpha,\beta\in\APre{c}$, except possibly when $\alpha,\beta\in\APTre{c}$ and $\SuppT(\alpha,\beta)$ is component-full,
\[\cm\beta\alpha c=\frac{K(\alpha,\alpha)}{K(\beta,\beta)}\cm{\alpha}{\beta}c.\]
\end{proposition}
\begin{proof}
Combine \cref{compat commute} with \cref{prop:scaling}.
\end{proof}
\begin{proposition}\label{compat ge -1}
For distinct roots $\alpha\neq\beta$ in $\APre{c}$, we have $\cm\alpha\beta c>-1$.
\end{proposition}
\begin{proof}
  This follows immediately from \cref{compat prop c fixed} and \cref{apSchur}.
\end{proof}

\begin{proposition}\label{compat restrict}
Let $\RS'$ be a parabolic root subsystem of $\RS$ and let $c'$ be the restriction of $c$ to the corresponding parabolic subgroup $W'$.
Then, for any $\alpha,\beta\in\RS'_{c'}$, we have $\cm\alpha\beta c=\cm\alpha\beta{c'}$, where the latter compatibility degree is computed in $\RS'$.
\end{proposition}
\begin{proof}
\cref{AP restrict} says that $\RS'_{c'}=\AP{c}\cap\RS'$.
If $\alpha$ and $\beta$ are in $\APTre{c}$ and $\SuppT(\alpha,\beta)$ is component-full, then $\alpha+\beta-\delta$ has nonnegative simple root coordinates.
Since $\delta$ already has strictly positive simple root coordinates, we see that $\alpha$ and $\beta$ are not both contained in the same proper parabolic root subsystem, so this case is impossible.
We thus restrict our attention to the case where $\cm\alpha\beta c=\max\left[\cmr\alpha\beta c,\cml\alpha\beta c\right]$.
By the definition of $\cmr\alpha\beta c$ and $\cml\alpha\beta c$, this maximum is $\max\left[\cmr\alpha\beta {c'},\cml\alpha\beta {c'}\right]$ which is $\cm\alpha\beta{c'}$ because $\RS'$ is finite.
\end{proof}
\begin{remark}
The finite cases of several propositions in this section are known (see for example \cite[Section~3]{MRZ}) but can also be proved by the same arguments given here for the affine case.
\end{remark}

\section{Compatibility and clusters}\label{clus sec}
\begin{definition}\label{compat def}
Distinct roots $\alpha$ and $\beta$ in $\AP{c}$ are \newword{$c$-compatible} if and only if ${\cm\alpha\beta c=0}$. 
\cref{compatibility almost symmetrizable} implies that $c$-compatibility is a symmetric relation.
\end{definition}

\begin{remark}\label{compat rem}
\cref{compat def} applies only to pairs of \emph{distinct} roots.
If we drop the requirement of distinctness, then by \eqref{compat delta U} and \cref{compat ge -1}, the only root that is ``$c$-compatible with itself'' is $\delta$.
We don't know whether the fact that $\cm\delta\delta c=0$ is significant, but we also see no reason to declare by fiat that $\cm\delta\delta c$ should be $-1$.
\end{remark}

\reversemarginpar
\begin{definition}\label{clus def}
The \newword{$c$-cluster complex} is the abstract simplicial whose vertex set is $\AP{c}$ and whose faces are the sets in which any two distinct roots are $c$-compatible.
  Maximal faces in the $c$-cluster complex are called \newword{$c$-clusters}.
  An \newword{imaginary $c$-cluster} is a cluster containing $\delta$ and \newword{real $c$-cluster} is a cluster not containing~$\delta$.  
\end{definition}
\cref{nu thm,denom thm}, which we will prove in \cref{g d sec}, state that the $c$-cluster complex (restricted to real $c$-clusters) is the cluster complex of the corresponding cluster algebra, and more specifically that real $c$-clusters are precisely the denominator vectors of clusters of cluster variables.
It is currently less clear what imaginary $c$-clusters mean in the context of cluster algebras, but as stated in the introduction, we expect that $\delta$ is the denominator vector of an important element of the cluster algebra.
An example of imaginary $c$-clusters can be seen in \cref{fanex}:
There are two imaginary $c$-clusters, each composed of $\delta$ (pictured at infinity in \cref{fig:fan}) and one of the two roots in $\APTre{c}$ (pictured in cyan in the figure).

We now establish the essential properties of compatibility and clusters.
The finite type versions of these properties can be found in \cite{FoZe03,MRZ} or can be proved using the arguments given here.

The following proposition is immediate from the characterization of compatibility degree in \cref{compat prop,compat prop c fixed,compat restrict}.
\begin{proposition}\label{prop:inductive_cluster}
For any root system $\RS$ of affine type and any Coxeter element~$c$,
\begin{enumerate}[\qquad \upshape (1)]
\item \label{cluster tau}
$\tau_c$ acts as a permutation of the set of $c$-clusters.
\item \label{cluster sigma}
If $s$ is an initial or final reflection in $c$ then $\sigma_s$ acts as a bijection from the set of $c$-clusters in $\AP{c}$ to the set of $scs$-clusters in $\AP{scs}$.
\item\label{cluster induct}
For each $i\in\set{1,\ldots n}$, write $\RS'$ for the parabolic root subsystem whose simple roots are $\Simples\setminus\set{\alpha_i}$ and write $c'$ for the restriction of $c$ to the corresponding parabolic subgroup of $W$.
Then $C\mapsto C\setminus\{-\alpha_i\}$ is a bijection from the set of $c$-clusters in $\AP{c}$ containing $-\alpha_i$ to the set of $c'$-clusters in~$\RS'_{c'}$.
\end{enumerate}
\end{proposition}

The following theorem reconciles \cref{clus def} with another reasonable way one might have defined ``real $c$-clusters.''

\begin{theorem}\label{clus def OK}
  A set $C\subseteq\AP{c}$ is a real $c$-cluster (a $c$-cluster not containing $\delta$) if and only if it is a maximal set of pairwise $c$-compatible roots in $\APre{c}$.
\end{theorem}

A real $c$-cluster is certainly a maximal set of pairwise $c$-compatible roots in $\APre{c}$, but \cref{clus def OK} is needed because it is conceivable that a maximal set $C$ of pairwise $c$-compatible roots in $\APre{c}$ can fail to be a $c$-cluster.
This would happen precisely if every root in $C$ were $c$-compatible with $\delta$.
To rule out this possibility, we need several preliminary results, one of which requires, for the first time in the paper, case-by-case checking using the classification of affine root systems.

\begin{proposition}\label{delta c compat}
  A root $\alpha\in\APre{c}$ is $c$-compatible with~$\delta$ if and only if $\alpha\in\APTre{c}$.
\end{proposition}
\begin{proof}
By \eqref{compat delta U}, roots in $\APTre{c}$ are $c$-compatible with~$\delta$.
Since $\Supp(\delta)=\Simples$, \cref{compat base} implies that $\cm{-\alpha_i}\delta c\neq0$ for all $i$.
Since $\tau_c$ fixes $\delta$, \cref{compat tau} completes the proof.
\end{proof}

\begin{definition}
For any affine root system $\RS$, we define $M_\RS$ to be 
\nomenclature[m]{$M_\RS$}{a useful quantity for technical arguments}
the maximal number of steps needed to go from any Coxeter element to a conjugate Coxeter element by a sequence of conjugations by initial simple reflections plus $n$ times the least common multiple of the ranks of the components of~$\RST{c}$.
\end{definition}
The precise quantity $M_\RS$ is of no great theoretical importance, but appears as a convenient bound in several lemmas where the point is that some bound exists.

\begin{lemma}\label{lemma in the tubes}  
  Given an affine root system $\RS$, a Coxeter element $c$ and a collection $R$ of real roots in $\APTre{c}$ such that $\SuppT(R)$ is not component-full, there exists a sequence $a_1,\ldots,a_\ell$ of elements of $S$ with $\ell\le M_\RS$ such that 
\begin{enumerate}[\rm(i)]
\item $a_i$ is initial in the Coxeter element $a_{i-1}\cdots a_1ca_1\cdots a_{i-1}$ for all $i=1\ldots \ell$, and 
\item $\Supp(a_\ell\cdots a_1R)$ is not full.
\end{enumerate}
\end{lemma}

\begin{proof}  
The conclusion of the lemma is not affected by rescaling the root system, so we may as well take $\RS$ to be a standard affine root system.
By an approporate choice of the initial letters of the sequence (making source-sink moves to change the Coxeter element), we may as well assume that $c$ as in \cref{tab:type-by-type}.
We will show that for such $c$ there exists an integer $m$ such that $\Supp(c^m(R))$ is not full.  From there, one can complete the desired sequence, because the $c$-orbit of $R$ is no larger than the least common multiple of the ranks of the components of~$\RST{c}$.

Indexing the components of $\RST{c}$ by indices $i$, we write $\beta_\aff^i$ for the root $\delta-\sum\beta$, where the sum is over those simple roots of the $i\th$ component that are also in $\RSfin$.
Thus $\beta_\aff^i$ is the unique simple root in the $i\th$ component that is not in $\RSfin$.
In each case, because $c$ cyclically permutes the simple roots of each component of $\RST{c}$ and since $\SuppT(R)$ is not component-full, by applying $c$ some number of times, we may as well assume that $\beta_\aff^1$ is not in $\SuppT(R)$.
(That is, the intersection of $R$ with the first component contains only roots in $\RSfin$.)
We now finish the proof in each case.
In each case, the simple generators of $W$ are numbered as in \cref{tab:type-by-type} and $c$ is $s_1\cdots s_n$.
We index components of $\RST{c}$ compatibly with the numbering of the roots $\beta_j$ in \cref{tab:type-by-type}:
The first component contains $\beta_1$, the second component contains $\beta_j$ for the smallest $j$ such that $\beta_j$ is not in the first component, etc.

In the proof, we use explicit simple-root coordinates of $\delta$ in types $A$, $B$, $D$, and~$E_6$.
These can be found, for example, in \cite[Table Aff 1]{Kac90}.

\noindent\textbf{Case $A_{n-1}^{(1)}$:}  
If $k\in\set{1,n-1}$, then there is nothing to prove.
Otherwise, if $\beta_\aff^2$ is not in $\SuppT(R)$, then $\alpha_n\not\in\Supp(R)$.
If $\beta_\aff^2\in\SuppT(R)$, then since $\SuppT(R)$ is not component-full, there exists $j\in\set{k,\ldots,n-2}$ such that $\beta_j\not\in\SuppT(R)$.
In this case, since $\beta_\aff^2=\delta-\sum_{j=k}^{n-2}\beta_j=\alpha_n+\sum_{i=1}^k\alpha_i$, we see that $\alpha_{j+1}\not\in\Supp(R)$.

\noindent\textbf{Case $B_{n-1}^{(1)}$:}  
If $\beta_\aff^2$ is not in $\SuppT(R)$, then $\alpha_n\not\in\Supp(R)$.
If $\beta_\aff^2\in\SuppT(R)$, then since $\SuppT(R)$ is not component-full, $\beta_{n-2}$ is not in $\SuppT(R)$.
Thus since $\beta_\aff^2=\delta-\beta_{n-2}=\alpha_n+\sum_{i=1}^{n-2}\alpha_i$, we see that $\alpha_{n-1}\not\in\Supp(R)$.

\noindent\textbf{Case $C_{n-1}^{(1)}$:} 
In this case, $\RST{c}$ has only one component. 
Since $\beta_\aff^1\not\in\SuppT(R)$, we have $\SuppT(R)\subseteq \RSfin$ and thus $\alpha_\aff=\alpha_n\not\in\Supp(R)$.

\noindent\textbf{Case $D_{n-1}^{(1)}$:}  
There are four cases depending on whether $\beta_\aff^2$ and $\beta_\aff^3$ are in $\SuppT(R)$.
If $\beta_\aff^2\not\in\SuppT(R)$ and $\beta_\aff^3\not\in\SuppT(R)$, then $\SuppT(R)\subseteq\RSfin$, so $\alpha_\aff=\alpha_n\not\in\Supp(R)$.
Now suppose $\beta_\aff^2\in\SuppT(R)$, so that $\beta_{n-3}\not\in\SuppT(R)$.
If furthermore $\beta_\aff^3\not\in\SuppT(R)$, then since $\beta_\aff^2=\delta-\beta_{n-3}=\alpha_n+\sum_{i=2}^{n-2}\alpha_i$, we have $\alpha_1\not\in\Supp(R)$.
If, on the other hand, $\beta_\aff^3\in\SuppT(R)$, then $\beta_{n-2}\not\in\SuppT(R)$.
Since $\beta_\aff^3=\delta-\beta_{n-2}=\alpha_1+\alpha_n+\sum_{i=3}^{n-2}\alpha_i$, we have $\alpha_{n-1}\not\in\Supp(R)$.
Finally, if $\beta_\aff^2\not\in\SuppT(R)$ and $\beta_\aff^3\in\SuppT(R)$, then $\beta_{n-2}\not\in\SuppT(R)$, so $\alpha_2\not\in\Supp(R)$.

\noindent\textbf{Case $E_6^{(1)}$:}  
Since there are $3$ simple roots in component $1$ of $\RST{c}$, permuted cyclically by $c$ and $2$ simple roots in component $3$, permuted cyclically by $c$, and since $gcd(2,3)=1$, we can (by applying $c$ repeatedly) assume that neither $\beta_\aff^1$ nor $\beta_\aff^3$ is in $\SuppT(R)$.
If furthermore $\beta_\aff^2\not\in\SuppT(R)$, then $\alpha_\aff=\alpha_7\not\in\Supp(R)$.
If $\beta_\aff^2\in\SuppT(R)$, then either $\beta_3$ or $\beta_4$ is not in $\SuppT(R)$.
If $\beta_4\not\in\SuppT(R)$, then since $\beta_\aff^2=\delta-(\beta_3+\beta_4)=\sum_{i\neq 3}\alpha_i$, we see that $\alpha_3\not\in\Supp(R)$.
If $\beta_3\not\in\SuppT(R)$, then we consider the collection $c^2R$.
None of the roots $\beta_2$, $\beta_\aff^2$, and $\beta_\aff^3$ is in $\SuppT(c^2R)$.
Since $\beta_\aff^1=\delta-(\beta_1+\beta_2)=\sum_{i=2}^7\alpha_i$, we see that $\alpha_1\not\in\Supp(R)$.

\noindent\textbf{Case $E_7^{(1)}$:}  
The ranks of the components of $\RST{c}$ are $4$, $3$, and $2$.
Since $\gcd(4,3,2)=1$, we can apply $c$ until $\set{\beta_\aff^i:i=1,2,3}\cap\SuppT(R)=\emptyset$.
Then $\alpha_\aff=\alpha_8\not\in\Supp(R)$.

\noindent\textbf{Case $E_8^{(1)}$:}
We proceed as in Case $E_7^{(1)}$.  
The ranks are $5$, $3$, and $2$.

\noindent\textbf{Case $F_4^{(1)}$:} 
We proceed as in the previous two cases.  
The ranks are $3$ and $2$.

\noindent\textbf{Case $G_2^{(1)}$:}  
In this case, $\RST{c}$ has only component, so we proceed as in Case~$C_{n-1}^{(1)}$. 
\end{proof}

\begin{proposition}\label{prop in the tubes}
For every affine root system $\RS$, Coxeter element $c$ and collection $R$ of roots in $\APTre{c}$ such that $\SuppT(R)$ is not component-full, there exists a root $\alpha\in\APre{c}\setminus\APTre{c}$ that is $c$-compatible with every root in $R$.
\end{proposition}
\begin{proof}
Let $s_{i_1},\ldots,s_{i_k}$ be a sequence as in \cref{lemma in the tubes} and let $\alpha_j$ be a simple root not in $\Supp(s_{i_k}\cdots s_{i_1}R)$.
Each $\beta\in R$ is positive and a simple induction using \cref{tUp init fin} implies that $s_{i_j}\cdots s_{i_1}(\beta)$ is positive for $j=1,\ldots,k$.
Thus $\sigma_{i_k}\cdots\sigma_{i_1}(\beta)=s_{i_k}\cdots s_{i_1}(\beta)$.
Now \eqref{compat base} implies that $-\alpha_j$ is $(s_{i_m}\cdots s_{i_1}cs_{i_1}\cdots s_{i_m})$\nobreakdash-compatible with $\sigma_{i_k}\cdots \sigma_{i_1}\beta$.
By \cref{sigma pos real in tubes} and \eqref{compat sigma}, the desired root is $\alpha=\sigma_{i_1}\cdots\sigma_{i_k}(-\alpha_j)$.
\end{proof}

\begin{lemma}\label{compat not full}
If $R$ is a set of pairwise $c$-compatible roots in $\APTre{c}$, then $\SuppT(R)$ is not component-full.
\end{lemma}
\begin{proof}
We need to show that there is a simple root of each component of $\RST{c}$ that is not in $\SuppT(R)$.
Fix a component $\Upsilon'$ of $\RST{c}$ and, if $\SuppT(R)$ includes any simple roots in $\Upsilon'$, choose $\beta\in R$ in $\Upsilon'$ so that $\SuppT(\beta)$ is maximal, under containment, among tube supports of roots in $R$.
By \cref{pos real in tubes}, $\SuppT(\beta)$ is not component-full.
Thus $\SuppT(\beta)$ is properly contained in the set of simple roots for $\Upsilon'$.
By \eqref{compat U} and since $\beta$ was chosen to have maximal tube support, any root $\beta'\in R\setminus\set{\beta}$ has either $\SuppT(\beta')\subsetneq\SuppT(\beta)$ or $\adj\beta{\beta'}=0$.
In particular, the simple root(s) of $\Upsilon'$ that are adjacent to $\beta$ are not in $\SuppT(R)$, and the lemma is proved.
\end{proof}

\begin{proof}[Proof of \cref{clus def OK}]
Let $R$ be a set of pairwise $c$-compatible roots in $\APre{c}$, all $c$-compatible with $\delta$.
\cref{compat not full} says that $R$ satisfies the conditions on $R$ in \cref{prop in the tubes}, which guarantees the existence a root $\alpha$ in the $\tau_c$-orbit of a negative simple root that is $c$-compatible with every root in $R$.
In particular, $R$ is not a maximal set of pairwise $c$-compatible roots.
\end{proof}

We next discuss compatibility among positive roots in $\APTre{c}$. 

\begin{definition}\label{nested spaced}
Two roots $\alpha,\beta\in\APTre{c}$ are \newword{nested} if $\SuppT(\alpha)\subseteq\SuppT(\beta)$ or $\SuppT(\beta)\subseteq\SuppT(\alpha)$.
They are \newword{spaced} if $\SuppT(c^{-1}\alpha)\cup\SuppT(\alpha)\cup\SuppT(c\alpha)$ is disjoint from $\SuppT(\beta)$.
\end{definition}

Each notion in \cref{nested spaced} is symmetric.
The roots $\alpha$ and $\beta$ are spaced if and only if $\SuppT(\alpha)$ is disjoint from $\SuppT(\beta)$ and $\adj\alpha\beta=0$.
In particular, roots from distinct components of $\RST{c}$ are spaced.
The next proposition follows immediately from \eqref{compat U}.

\begin{proposition}\label{compatible in tubes}
Two distinct roots $\alpha$ and $\beta$ in $\APTre{c}$ are $c$-compatible if and only if they are nested or spaced.
\end{proposition}

\cref{compatible in tubes} allows us to connect the notion of compatibility among roots in $\APTre{c}$ to some well-known combinatorics.
Supports of roots in $\APTre{c}$ correspond to \newword{tubes}, in the sense of graph associahedra~\cite[Definition~2.2]{CarrDevadoss}, in the Dynkin diagram of $\RST{c}$. 
Compatibility of these roots corresponds to compatibility of tubes in the same sense.  
Thus the restriction of the $c$-cluster complex to roots in $\APTre{c}$ is isomorphic to the boundary complex of the \newword{simplicial graph associahedron} for the Dynkin diagram of $\RST{c}$, the polytope dual to the (simple) graph associahedron of \cite[Definition~2.4]{CarrDevadoss}.
Each component of the Dynkin diagram of $\RST{c}$ is a cycle or consists of two vertices connected by an edge.
The simplicial graph associahedron of a $k$-cycle is a $(k-1)$-dimensional \newword{simplicial cyclohedron} (so its boundary complex is a $(k-2)$-dimensional simplicial complex).
The simplicial graph associahedron associated to two vertices connected by an edge is a line segment (so its boundary complex is two isolated vertices, i.e.\ a $0$-dimensional sphere), and we extend terminology slightly to call this a $1$-dimensional simplicial cyclohedron.
We summarize in the following proposition.
\begin{proposition}\label{cyclo}
In affine type, the restriction of the $c$-cluster complex to roots in $\APTre{c}$ is $(n-3)$-dimensional and isomorphic to a join of boundary complexes of simplicial cyclohedra.
There is a $(k-1)$-dimensional cyclohedron (and thus a $(k-2)$-dimensional boundary complex) for each rank-$k$ component of~$\RST{c}$.
\end{proposition}

We conclude the section by cataloging the crucial facts about real and imaginary $c$-clusters.
The finite-type versions of these facts are in \cite{FoZe03,MRZ}, or can be proved as outlined here for affine type.
We use the notation $Q$ \nomenclature[q]{$Q$}{root lattice for $\RS$}
for the root lattice of $\RS$ and $Q_c$ for the sublattice $Q\cap \eigenspace{c}$, \nomenclature[qc]{$Q_c$}{$Q\cap\eigenspace{c}$, the root lattice for $\RST{c}$}
which should be thought of as the root lattice of $\RST{c}$.
In the following proposition, we interpret indices in the expressions $\sigma_k\cdots\sigma_1$ and $s_k\cdots s_1$ modulo $n$ and, if $k<0$ we interpret the expression $\sigma_k\cdots\sigma_1$ to mean $\sigma_{k+1}\sigma_{k+2}\cdots\sigma_{-1}\sigma_0$ and similarly for $s_k\cdots s_1$.  
\begin{proposition}\label{cluster properties}
Let $\RS$ be a root system of affine type, let $c=s_1\cdots s_n$ be a Coxeter element, and let $C$ be a $c$-cluster.
If $C$ is a real $c$-cluster, then 
\begin{enumerate}[\qquad \upshape (1)]
\item \label{Q n}
$|C|=n$,
\item \label{Q basis}
$C$ is a $\integers$-basis for the root lattice $Q$ of $\RS$, 
\item \label{at least 2}
$C$ contains at least $2$ roots in the $\tau_c$-orbits of negative simple roots, and 
\item \label{sigmas}
There exists an integer $k$ and an index $i\in\set{1,\ldots,n}$ such that
\[\sigma_k\cdots\sigma_1(C)=s_k\cdots s_1(C)=\set{-\alpha_i}\cup C'\]
for a $c'$-cluster $C'$ of roots in $\RS'_{c'}$, where $\RS'$ is the parabolic root subsystem spanned by $\Simples\setminus\set{\alpha_i}$ and $c'$ is the restriction of $s_k\cdots s_1cs_1\cdots s_k$ to the corresponding parabolic subgroup of $W$.
\end{enumerate}
If $C$ is an imaginary $c$-cluster, then 
\begin{enumerate}[\qquad \upshape (1)]\setcounter{enumi}{4}
\item \label{Qc n-1}
$|C|=n-1$,
\item \label{Qc basis}
$C$ is a $\integers$-basis for the lattice $Q_c$,
\item \label{none}
$C$ contains no roots in the $\tau_c$-orbits of negative simple roots, and
\item \label{in the cone}
$C\setminus\set{\delta}$ consists of roots in $\APTre{c}$.
\end{enumerate}
\end{proposition}
\begin{proof}
First, assume $C$ is an imaginary $c$-cluster.
\cref{delta c compat} implies \eqref{none} and~\eqref{in the cone}.
Thus \cref{cyclo} implies \eqref{Qc n-1}.
Furthermore, one easily convinces oneself that for each root $\beta$ in $C$, there exists a unique $\beta_i\in\SuppT(\beta)$ such that $\beta_i$ is not in the tube support of any root $\beta'\in C$ with $\SuppT(\beta')\subsetneq\SuppT(\beta)$.
Thus the $\SimplesT{c}$-coordinates of roots in $C$ form a unitriangular matrix with respect to a linear extension of the containment order on tube supports of roots in $C$.
If $\RST{c}$ is irreducible, then $\SimplesT{c}$ is a basis for $Q_c$, so it follows that $C$ is also a basis.
In general,~$C$ contains a basis for each component of $\RST{c}$, with these bases overlapping exactly in the root $\delta$, and \eqref{Qc basis} follows.

Now, assume $C$ is a real $c$-cluster.
\cref{clus def OK,delta c compat} together imply that there exists an integer $\ell$ such that $\tau_c^\ell(C)$ contains a negative simple root.
That is, using the convention described before the statement of the proposition, $\sigma_{-n\ell}\cdots\sigma_1(C)$ contains a negative simple root.
If $\ell\le0$, choose the smallest $k$ with $0\le k\le-n\ell$ such that $\sigma_k\cdots\sigma_1(C)$ contains a negative simple root $-\alpha_i$.
If $\ell>0$, then choose the largest $k$ with $-n\ell\le k\le0$ such that $\sigma_k\cdots\sigma_1(C)$ contains a negative simple root $-\alpha_i$.
Because of this choice of $k$, in either case, we know that each application of $\sigma_j$ in the expression $\sigma_k\cdots\sigma_1(C)$ acts as $s_j$ on $C$, and thus $\sigma_k\cdots\sigma_1(C)=s_k\cdots s_1(C)$.
\cref{prop:inductive_cluster}\eqref{cluster sigma} implies that $\sigma_k\cdots\sigma_1(C)$ is a $s_k\cdots s_1cs_1\cdots s_k$-cluster.
Now \eqref{sigmas} follows by \cref{prop:inductive_cluster}\eqref{cluster induct}, with $C'=(\sigma_k\cdots\sigma_1(C))\setminus\set{-\alpha_i}$.

In particular, $s_k\cdots s_1(C)=\set{-\alpha_i}\cup C'$.
Since necessarily $\RS'$ is of finite type, $C'$ is a real $c'$-cluster.
By the finite version of this result (or by induction), $C'$ is a $\integers$-basis for the root lattice of $\RS'$, and therefore $s_k\cdots s_1(C)$ is a $\integers$-basis for $Q$.
Since elements of $W$ preserve the lattice, $C$ itself is also a $\integers$-basis for $Q$.
This yields \eqref{Q basis}, and \eqref{Q n} follows immediately.
By \cref{cyclo}, the largest possible size of a set of pairwise $c$-compatible roots in $\APTre{c}$ is $n-2$, and \eqref{at least 2} follows by \cref{tau prop}\eqref{tauc fin}.
\end{proof}

\section{Cluster expansions and the cluster fan}\label{fan sec}  
In this section, we show that the $c$-compatibility relation defines a simplicial fan, or equivalently that each vector in $V$ has a unique $c$-cluster expansion.  
Finite-type versions of these results are in \cite{FoZe03,MRZ,Ste13}.

\begin{definition}\label{clus exp def}  
Given a vector in $v\in V$, a \newword{$c$-cluster expansion} of $v$ is an expression $v=\sum_{\alpha \in \AP{c}}m_\alpha \alpha$ where the $m_\alpha$ are nonnegative real numbers having $m_\alpha m_\beta=0$ whenever $\alpha$ and $\beta$ are distinct and not $c$-compatible.
In light of \cref{cluster properties}, at most $n$ of the $m_\alpha$ can be nonzero.
The \newword{support} of a $c$-cluster expansion $v=\sum_{\alpha \in \AP{c}}m_\alpha \alpha$ is the set $\set{\alpha\in\AP{c}:m_\alpha\neq0}$.
\end{definition}

\begin{theorem}\label{clus exp thm}
For any root system $\RS$ of affine type and any Coxeter element $c$, every vector in $V$ admits a unique $c$-cluster expansion.  
\end{theorem}
We emphasize that a vector may have a $c$-cluster expansion is terms of a real $c$-cluster or an imaginary $c$-cluster (or both, if the support of the expansion is contained in the intersection of the two clusters).

Before proving \cref{clus exp thm}, we rephrase it geometrically.
A \newword{simplicial cone} is the nonnegative linear span of a linearly independent set $X$ of vectors and a \newword{face} of the cone is the nonnegative linear span of a subset of $X$. 
A \newword{simplicial fan} is a collection of simplicial cones that is closed under passing to faces, and such that the intersection of any two cones in the collection is a face of each.
A simplicial fan is \newword{complete} if the union of its cones is the entire ambient vector space.

\begin{definition}\label{Fanc def}
Given a set $C$ of pairwise $c$-compatible roots in $\AP{c}$, write $\Cone(C)$\nomenclature[cone]{$\Cone(C)$}{$\reals_{\ge 0}\cdot C$} for the nonnegative linear span of $C$.
Write $\Fan_c(\RS)$\nomenclature[fanc]{$\Fan_c(\RS)$}{finite or affine cluster fan} for the collection of cones $\Cone(C)$ for all sets $C$ of pairwise $c$-compatible roots in $\AP{c}$.
Write $\Fan_c^\re(\RS)$\nomenclature[fancre]{$\Fan_c^\re(\RS)$}{real affine cluster fan} for the subfan of $\Fan_c(\RS)$ consisting of cones spanned by real roots.
We call $\Fan_c(\RS)$ the \newword{affine cluster fan} or \newword{affine generalized associahedron fan}.
\end{definition}

\cref{cluster properties}\eqref{Q basis} and \cref{cluster properties}\eqref{Qc basis} imply that the cones in $\Fan_c(\RS)$ are simplicial.
The content of \cref{clus exp thm} is that every point in $V$ is in some cone of $\Fan_c(\RS)$ and that any two cones of $\Fan_c(\RS)$ intersect in such a way that, for any $v$ in their intersection, the $c$-cluster expansions of $v$ coming from the two cones agree.
This is precisely the assertion that the two cones intersect along a mutual face.
Thus \cref{clus exp thm} has the following rephrasing.

\begin{theorem}\label{fan thm}
For any root system $\RS$ of affine type and any Coxeter element $c$,
$\Fan_c(\RS)$ is a complete, simplicial fan. 
\end{theorem}

\begin{example}\label{fanex}
An example of $\Fan_c(\RS)$ is shown in \cref{fig:fan}, for $\RS$ of type $D_3^{(2)}$ and $c=s_1s_2s_3$.
(Compare~\cite[Example~1.3]{afforb}.)
The intersection of $\Fan_c(\RS)$ with a unit sphere is a collection of points, geodesic arcs, and spherical triangles.  
The figure shows that collection, stereographically projected to the plane.
The cone spanned by $-\Simples$ appears as the smallest triangle on the vertical line of symmetry of the picture.
The colored points are (the directions of) some of the roots in $\AP{c}$, with a different color for each $\tau_c$-orbit. 
\end{example}
\begin{figure}

  \caption{$\Fan_c(\RS)$ for $\RS$ and $c$ as in \cref{fanex}}
\label{fig:fan}
\end{figure}

We now prepare to prove \cref{clus exp thm}.
Let $\Delta_c$\nomenclature[zzdc]{$\Delta_c$}{$\reals_{\ge 0}\cdot\RST{c}$} be the nonnegative linear span of the positive roots in $\RST{c}$.
This is a closed polyhedral cone whose extreme rays are the simple roots $\SimplesT{c}$ of $\RST{c}$.
Thus the cross section of $\Delta_c$ is a product of simplices, with one $k$-vertex simplex for each rank-$k$ component of $\RST{c}$.
The imaginary root~$\delta$ is in the relative interior of $\Delta_c$, which we thus call the \newword{imaginary cone}.
(This is not the same as the cone spanned by positive imaginary roots, which is the ray spanned by~$\delta$.  Typically $\Delta_c$ contains infinitely many positive real roots.)
We write $\mathring\Delta_c$\nomenclature[zzdczzzz]{$\mathring\Delta_c$}{relative interior of $\Delta_c$} for the relative interior of $\Delta_c$.
As an immediate consequence of \cref{tUp simples}, we have the following proposition and corollary.

\begin{proposition}\label{scs Delta}
If $s$ is initial or final in $c$, then $\Delta_{scs}=s\Delta_c$.
\end{proposition}
\begin{corollary}\label{c Delta}
$c\,\Delta_c=\Delta_c$.
\end{corollary}

We separate out one case of \cref{clus exp thm} as the following proposition.

\begin{proposition}\label{clus exp imaginary}
Every vector $v\in\Delta_c$ has a $c$-cluster expansion.
This expansion is the unique $c$-cluster expansion for $v$ that is supported on roots in $\APT{c}$.
\end{proposition}
\begin{proof}
Such a vector $v$ can be written as a finite positive linear combination of simple roots of $\RST{c}$.
The simple roots in each component of $\RST{c}$ sum to $\delta$.
Thus we can write $v$ as a linear combination, with nonnegative coefficients, of $\delta$ and elements of $\SimplesT{c}$, with at least one zero coefficient in each component of $\RST{c}$.
This expression is unique because the linear spans of any two components of $\RST{c}$ only intersect in the line spanned by $\delta$. 
For the same reason, the proposition reduces to proving the following claim:
A vector $v'$ in the nonnegative linear span of a proper parabolic root subsystem of a component $\Upsilon'$ 
of $\RST{c}$ has a unique cluster expansion supported on positive real roots in that component.
(One may be tempted to directly deduce the claim from the finite case, in the parabolic subsystem.
This does not work, because the notions of proper parabolics for $\RST{c}$ and $\RS$ do not coincide.)

Given such a $v'$, we can write $v'=\sum_{\beta\in N'}x_\beta\beta$ for a unique $N'\subsetneq\SimplesT{c}\cap\Upsilon'$ with $x_\beta>0$ for all $\beta\in N'$.
Let $x=\min\set{x_\beta:\beta\in N'}$.
By hypothesis, $N'$ generates a root subsystem whose components are all of finite type $A$.
Let $\beta\in N'$ have $x_\beta=x$ and let $\alpha$ have maximal tube support among roots in this subsystem with $\beta\in\SuppT(\alpha)$ such that $v''=v'-x\alpha$ is in the nonnegative linear span of~$N'$. 
Write $v''=\sum_{\beta\in N''}y_\beta\beta$ with $N''\subsetneq N'$ and $y_\beta>0$ for all $\beta\in N''$.
By the minimality of $x$ and maximality of $\alpha$, $\alpha$ is either spaced or nested with all the roots in the subsystem generated by $N''$. 
The existence of a $c$-cluster expansion for $v'$ then follows by induction on $|N'|$.

Let $\ell$ be the minimum integer such that there exists a vector $v'$ with $|N'|=\ell$, having two distinct $c$-cluster expansions.
By minimality of $\ell$, the set $N'$ (as in the previous paragraph) induces a connected proper subgraph of the diagram of $\Upsilon'$, and this subgraph is necessarily a path.
This path is the tube support of the root $\alpha$.
In light of \cref{compatible in tubes}, every $c$-cluster expansion of $x'$ must include $\alpha$ with coefficient $x$.
Therefore the vector $v''=v'-x\alpha$ has two $c$-cluster expansions of length $\ell-1$, contradicting our assumption.
We have proved the claim.
\end{proof}

Besides proving part of \cref{clus exp thm}, \cref{clus exp imaginary} is helpful in proving the following key lemma.

\begin{lemma}\label{rotate to nonpos}  
Given a nonzero $v\in V\setminus\mathring\Delta_c$, there exists a sequence of simple reflections $a_1,\ldots,a_k$ such that 
\begin{enumerate}[\rm(i)]
\item $a_i$ is initial or final in the Coxeter element $a_{i-1}\cdots a_1ca_1\cdots a_{i-1}$ for $i=1\ldots k$, and 
\item there is at least one nonpositive coefficient in the expansion of $a_ka_{k-1}\cdots a_1(v)$ in the basis of simple roots of $\RS$.
\end{enumerate}
The sequence can be chosen so that, in condition (i), either $a_i$ is initial for all $i$ or $a_i$ is final for all $i$.  
If $v\in \eigenspace{c}\setminus\mathring\Delta_c$, then the sequence can be chosen with $k\le M_\RS$.
\end{lemma}
\begin{proof}
First, suppose $v\in \eigenspace{c}$.
For large enough positive $x\in\reals$, the vector $v+x\delta$ is in $\Delta_c$, and thus by \cref{clus exp imaginary}, the vector $v+x\delta$ has a cluster expansion consisting of roots in $\APT{c}$.
We can choose $x$ so that $\delta$ is not in the support of the expansion (or in other words, so that $v+x\delta$ is on the boundary of~$\Delta_c$).
We have $x\ge0$, because by hypothesis $v\not\in\mathring\Delta_c$.
Write $\sum x_i\beta_i$ for the cluster expansion of $v+x\delta$, so that $v=-x\delta+\sum x_i\beta_i$.
In light of \cref{compat not full}, we can apply \cref{lemma in the tubes} to find a sequence $a_1,\ldots,a_k$ satisfying the first condition of the lemma (choosing ``initial'' for each $i$) and a simple root $\alpha_j$ with ${[\sum x_ia_ka_{k-1}\cdots a_1\beta_i:\alpha_j]=0}$. 
Since $\delta$ has full support in the basis of simple roots, we see that ${[-x\delta+\sum x_ia_ka_{k-1}\cdots a_1\beta_i:\alpha_j]\le0}$, and since ${a_ka_{k-1}\cdots a_1\delta=\delta}$, the inequality says $[a_ka_{k-1}\cdots a_1v:\alpha_j]\le0$.

Now suppose $v\not\in \eigenspace{c}$ and write $v=a\gamma_c+b\delta+w$ with $w\in \eigenspace{c}\cap V_\fin$ and ${a\neq 0}$.
\cref{eigen}\eqref{eigen U} says that $c$ has finite order on $\eigenspace{c}$.
Thus there exists a positive integer $\ell$ such that $c^\ell w=w$. 
Since $c\gamma_c=\gamma_c+\delta$, we compute $c^{m\ell}v$ to be $a\gamma_c+(b+m\ell a)\delta+w$ for any $m\in\integers$.
We can choose $m$ so that $c^mv$ has at least one negative simple-root coordinate.
If $m=0$, then we are done.
If $m<0$, then the sequence $s_{-mn},s_{-mn+1},\ldots,s_1$, with indices interpreted modulo $n$, has the desired properties (choosing ``initial'' for each $i$).
If $m>0$, then the sequence $s_1,s_2,\ldots,s_{mn}$, again with indices modulo $n$, is the desired sequence (with ``final'' for all $i$).
\end{proof}

\begin{proof}[Proof of \cref{clus exp thm}]
A natural way to prove this theorem is to state it more broadly for finite and affine type together and argue by induction on rank.
However, since when $\RS$ is affine, every proper parabolic root subsystem of $\RS$ is finite, and since the finite case is known (\cite[Theorem~3.11]{FoZe03} and \cite[Proposition~3.6]{MRZ}), we appeal to the finite case rather than to induction.

Let $v$ be a vector in $V$.
First, suppose $v$ is zero.
Then $v$ has a cluster expansion with no terms.
If some other cluster expansion exists, then some negative simple root $-\alpha_i$ must occur in the expansion with positive coefficient.
Therefore, by \cref{compat base}, the other roots in the support of the second cluster expansion are in the complementary parabolic subspace, and this is a contradiction.
We assume from now on that $v$ is nonzero.

Suppose now that $v$ is a vector in $V\setminus\mathring\Delta_c$.
Choose a sequence $a_1,\ldots,a_k$ to satisfy conditions (i) and (ii) of \cref{rotate to nonpos}, and furthermore choose the sequence to minimize $k$.
Thus $\tilde v=a_ka_{k-1}\cdots a_1(v)$ has a nonpositive coefficient in its simple-root expansion.
Write $\tilde c$ for the Coxeter element $a_ka_{k-1}\cdots a_1ca_1\cdots a_{k-1}a_k$.

Write $\tilde v_+$ for the vector $\sum_{i=0}^n[\tilde v:\alpha_i]_+\alpha_i$.
By construction $\tilde v_+$ is supported on a proper parabolic root subsystem.
Therefore it has a unique $c'$-cluster expansion involving only positive roots in the subsystem.
(Here $c'$ is the restriction of $\tilde c$ to the parabolic subgroup.)
Property \eqref{compat base} and the fact that the only negative roots in $\AP{c}$ are negative simple roots then imply that $\tilde v$ has a unique $\tilde c$-cluster expansion and that this expansion writes $\tilde v$ as a nonnegative combination of positive roots from this parabolic and negative simple roots in its complement.

If $k=0$, then $\tilde v=v$ and we are done.
If $k>0$, then $v$ has strictly positive simple-root coordinates.
Because we chose $k$ to be minimal, $a_k\tilde v$ also has strictly positive simple-root coordinates.
But $a_k$ equals some $s_j$, so $a_k\tilde v=s_j\tilde v$ differs from $\tilde v$ by a multiple of the simple root $\alpha_j$, and therefore $\tilde v$ has strictly positive simple-root coordinates except for the coordinate of $\alpha_j$.
Thus the unique $\tilde c$-cluster expansion of~$\tilde v$ is a nonnegative combination of positive roots whose $j$-th coordinate equals~$0$ and the root~$-\alpha_j$.
On all of these roots, the map $\sigma_j$ agrees with the map $s_j$.
Property \eqref{compat sigma} thus implies that an $s_j\tilde cs_j$-cluster expansion of $s_j\tilde v$ is obtained by applying $s_j$ to each root in the $\tilde c$-cluster expansion of $\tilde v$.
This $s_j\tilde cs_j$-cluster expansion of $s_j\tilde v$ is supported on positive roots, because otherwise property \eqref{compat base} implies a contradiction to the fact that $s_j\tilde v$ has strictly positive simple-root coordinates.
We claim that this $s_j\tilde cs_j$-cluster expansion of $s_j\tilde v$ is unique. 
Any other $s_j\tilde cs_j$-cluster expansion of $s_j\tilde v$ is supported on positive roots for the same reason, so $\sigma_j$ agrees with $s_j$ on the other expansion as well.
Thus property \eqref{compat sigma} implies that by applying $s_j$ to each root appearing in the other $s_j\tilde cs_j$-cluster expansion of $s_j\tilde v$, we obtain another $\tilde c$-cluster expansion of~$\tilde v$, but we have already established that $\tilde v$ has a unique $\tilde c$-cluster expansion.
This contradiction implies the uniqueness claim.

We now apply repeatedly the same argument to obtain a unique $c$-cluster expansion of $v$. 
The minimality of the sequence $a_1,\ldots,a_k$ guarantees that, at each step, we obtain an expansion supported only on positive roots so that the relevant simple reflection $s_i$ coincides with $\sigma_i$ on its support.

Finally, suppose $v\in\mathring\Delta_c$.  
\cref{clus exp imaginary} says that $v$ has a $c$-cluster expansion supported on roots in $\APT{c}$, and that this is the unique $c$-cluster expansion of $v$ supported on such roots.
It remains to show that $v$ has no $c$-cluster expansion whose support contains a negative simple root or a root outside of $\eigenspace{c}$.
But if $v$ has a $c$-cluster expansion containing such a root, \cref{tau prop} implies that we can apply $\tau_c^m$, for some integer $m$, to one of the roots in the support of the $c$-cluster expansion to obtain a negative simple root.
Thus there is a sequence $a_1,\ldots,a_k$ satisfying condition (i) of \cref{rotate to nonpos} such that the corresponding sequence of $\sigma_i$ takes one of the roots in the support of the $c$-cluster expansion to a negative simple.
Choosing this sequence to minimize $k$, each of the $\sigma_i$ acts as the corresponding $s_i$ on each root, so $a_k\cdots a_1$ applied to the root yields a negative simple.
Furthermore, the linear map $a_k\cdots a_1$ applied to all of the roots in the $c$-cluster expansion of $v$ yields an $a_k\cdots a_1ca_1\cdots a_k$-cluster expansion of $a_k\cdots a_1v$.
By Property \eqref{compat base}, one of the simple-root coordinates of $a_k\cdots a_1v$ is strictly negative.
This contradicts \cref{scs Delta}, because the latter implies that $a_k\cdots a_1v$ is in $\mathring\Delta_c$.  
\end{proof}

If $s$ is initial or final in $c$, then we extend the map $\sigma_s$ linearly on each cone of $\Fan_c(\RS)$ to define a map (which we also call $\sigma_s$) on all of $V$.
We extend~$\tau_c$ in the same way.
The following proposition is immediate by \cref{sigma prop} and \eqref{compat sigma}.
\begin{proposition}\label{fan isom}\quad
\begin{enumerate}[\qquad \upshape (1)]
\item \label{sigma fan}
If $s$ is initial or final in $c$ then $\sigma_s$ induces an isomorphism from $\Fan_c(\RS)$ to $\Fan_{scs}(\RS)$, which restricts to an isomorphism $\Fan^\re_c(\RS)\to\Fan^\re_{scs}(\RS)$.
\item \label{tau fan}
The map $\tau_c$ induces an automorphism of $\Fan_c(\RS)$, which restricts to an automorphism of $\Fan^\re_c(\RS)$.
\end{enumerate}
\end{proposition}

\cref{prop:scaling} implies that the set $\AP{c}$ and the compatibility relation are preserved by rescaling $\RS$, so the following proposition is immediate.

\begin{proposition}\label{Fanc scaling}
If $\RS$ and $\RS'$ are finite or affine root systems related by rescaling, then $\Fan_c(\RS)$ coincides with $\Fan_c(\RS')$.
\end{proposition}

\cref{rotate to nonpos} implies a description of $\Delta_c$ in terms of inequalities.  

\begin{proposition}\label{Delta ineq}
The cone $\Delta_c$ is the set of points $v$ in $\eigenspace{c}$ satisfying the inequality $\br{s_{i_1}\cdots s_{i_k}\rho_{i_k},v}\ge0$
for all sequences $s_{i_1},\ldots,s_{i_k}$ of simple reflections with $k\le M_\RS$ such that $s_{i_j}$ is initial in the Coxeter element $s_{i_{j-1}}\cdots s_{i_1}cs_{i_1}\cdots s_{i_{j-1}}$ for all $i=1\ldots k$.
\end{proposition}
\begin{proof}
\cref{scs Delta} and the fact that $\Delta_c$ is contained in the cone of vectors with nonnegative simple-root coordinates for all $c$ imply that every $v\in\Delta_c$ satisfies the inequalities $\br{\rho_{i_k},s_{i_k}\cdots s_{i_1}v}\ge0$.
These are equivalent to ${\br{s_{i_1}\cdots s_{i_k}\rho_{i_k},v}\ge0}$.

If $v$ is a nonzero vector in $V\setminus\mathring\Delta_c$, then by \cref{rotate to nonpos} there is a sequence $s_{i_1},\ldots,s_{i_k}$, with $k\le M_\RS$, satisfying the ``initial or final'' condition of the proposition such that the inequality $\br{\rho_j,s_{i_k}\cdots s_{i_1}v}>0$ fails for some $j$.
Choosing $k$ as small as possible, because $\alpha_{i_k}$ is the unique positive root that becomes negative under the action of $s_{i_k}$, we can take $j=i_k$.
Thus the relative interior of $\Delta_c$ is separated from other points in $\eigenspace{c}$ by inequalities as described in the proposition.
We conclude that $\Delta_c$ is defined by the given list of inequalities for $k\le M_\RS$.
\end{proof}

The \newword{link} of a ray $r$ in a simpicial fan is the subfan consisting of cones $C$ such that $r\not\subseteq C$ but the nonnegative linear span of $r\cup C$ is a cone of the fan.
The \newword{star} of $r$ is the subfan consisting of cones containing~$r$.
The following proposition is a restatement of \cref{cyclo,delta c compat,Delta ineq}.
\begin{proposition}\label{delta star}
Suppose $\RS$ is of affine type and $c$ is a Coxeter element of $W$.
\begin{enumerate}
\item
The link of the ray spanned by $\delta$ is isomorphic, as a simplicial complex, to a join of boundary complexes of simplicial cyclohedra.
There is a $(k-1)$-dimensional cyclohedron (and thus a $(k-2)$-dimensional boundary complex) for each rank-$k$ component of~$\RST{c}$.
\item
The union of the cones in the star of the ray spanned by $\delta$ is the cone $\Delta_c$.  
It is a closed polyhedral cone whose extreme rays are spanned by the simple roots of $\RST{c}$.
Its defining inequalities are given by \cref{Delta ineq}.
\end{enumerate}
\end{proposition}

\section{Exchangeable roots}\label{exch sec}
The definition of $\Fan_c(\RS)$ depends only on when the compatibility degree is zero.
In this section, we use more information about the compatibility degree (along with some extra information for cones contained in $\Delta_c$) to describe when roots in $c$-clusters can be exchanged, in the sense of the following definition.

\begin{definition}\label{def exchangeable}
Two distinct roots $\alpha$ and $\beta$ in $\AP{c}$ are said to be \newword{$c$-exchangeable} if there exist $c$-clusters $C$ and $C'$ with $\alpha\in C$, with $\beta\in C'$, and with $C\setminus\{\alpha\}=C'\setminus\{\beta\}$.
Two distinct real roots in $\AP{c}$ are said to be \newword{$c$-real-exchangeable} if there exist \emph{real} clusters $C$ and $C'$ with these properties.
\end{definition}

If $\alpha,\beta\in\AP{c}$ are $c$-exchangeable, then they are not $c$-compatible.
(If so, then $C\cupdot\set{\beta}$ is pairwise $c$-compatible for $C$ as in \cref{def exchangeable}.
But $C$ is a $c$-cluster.)

\begin{theorem}\label{exchangeable}   
Suppose $\RS$ is of affine type and $c$ is any Coxeter element.
\begin{enumerate}[\qquad \upshape (1)]
\item \label{exch 1}
Real roots $\alpha$ and $\beta$ in $\APre{c}$ are $c$-exchangeable if and only if   
\[{\cm\alpha\beta c=1=\cm\beta\alpha c}.\]
\item \label{exch Delta}
Real roots $\alpha$ and $\beta$ in $\APre{c}$ are $c$-real-exchangeable if and only if they are $c$-exchangeable and $\alpha+\beta\not\in\mathring\Delta_c$.
\item \label{exch component-full}  
  Real roots $\alpha,\beta\in\APre{c}$ that are $c$-exchangeable fail to be $c$-real-exchangeable if and only if $\alpha,\beta\in\APTre{c}$ and $\SuppT(\alpha,\beta)$ is component-full.
\item\label{exch delta}
The imaginary root $\delta$ is not $c$-exchangeable with any other root.  
\end{enumerate}
\end{theorem}

\begin{remark}\label{exch rem}
In \cite{FoZe03a}, exchangeability is \emph{defined} by $\cm\alpha\beta c=1=\cm\beta\alpha c$ and then shown to be equivalent (in finite type) to the condition in \cref{def exchangeable}.  
We prefer the term ``exchangeable'' for the condition in \cref{def exchangeable}.  
\end{remark}

Despite \cref{exchangeable}\eqref{exch delta}, there are \emph{pairs} of real roots that are in some sense $c$-exchangeable with~$\delta$.  

\begin{definition}\label{def pair exchangeable}
The pair $\set{\alpha,\beta}$ is \newword{$c$-exchangeable with $\delta$} if there exists a (necessarily imaginary) $c$-cluster $C$ with $\delta\in C$ and a (necessarily real) $c$-cluster $C'$ with $\alpha,\beta\in C'$, such that $C\setminus\set{\delta}=C'\setminus\set{\alpha,\beta}$.
\end{definition}

\begin{theorem}\label{pair exchangeable}
Suppose $\RS$ is of affine type, but not of type $A^{(2)}_{2k}$ and suppose $c$ is any Coxeter element.
For each $\alpha\in\AP{c}$, the following are equivalent.
\begin{enumerate}
\item $\cm\alpha\delta c=1=\cm\delta\alpha c$.
\item There exists $\beta\in\AP{c}$ such that the pair $\set{\alpha,\beta}$ is $c$-exchangeable with $\delta$.
\end{enumerate}
When these equivalent conditions hold, $\alpha$ is in the $\tau_c$-orbit of a negative simple root $-\alpha_j$ such that there exists a diagram automorphism taking $\alpha_j$ to $\alpha_\aff$. 
\end{theorem}

\begin{remark}\label{which pairs}
\cref{pair exchangeable} stops short of describing which pairs $\set{\alpha,\beta}$ are $c$-exchangeable with $\delta$.
The simplest characterization one might propose is that $\set{\alpha,\beta}$ are $c$-exchangeable with $\delta$ if and only if $\alpha$ and $\beta$ are $c$-compatible and both satisfy the conditions of \cref{pair exchangeable}.
However, this is already false in type $A_2^{(1)}$, as we now explain.
Recall that $A_2^{(1)}$ has rank $3$, take $c$ as in \cref{tab:type-by-type}, and recall that each $\alpha_i$ equals $\alpha_i\ck$ and that $\delta=\delta\ck=\alpha_1+\alpha_2+\alpha_3$.
Setting $\alpha=-\alpha_1$ and $\beta=-\alpha_2$, indeed $\alpha$ and $\beta$ are $c$-compatible, $\cm\alpha\delta c=1=\cm\delta\alpha c$, and $\cm\beta\delta c=1=\cm\delta\beta c$.
Diagram automorphisms act transitively on the simple roots.
However, there are only two real $c$-clusters that contain $\alpha$ and $\beta$, namely $\set{\alpha,\beta,-\alpha_3}$ and $\set{\alpha,\beta,\alpha_3}$, and we calculate $\cm{-\alpha_3}\delta c=1$ and $\cm{\alpha_3}\delta c=\cm{\tau_c\alpha_3}{\tau_c\delta}c=\cm{-\alpha_3}\delta c=1$.
\end{remark}

\begin{example}\label{fails A22k}
We give two simple examples to illustrate why \cref{pair exchangeable} must exclude type $A^{(2)}_{2k}$.
First, type $A^{(2)}_2$ has rank $2$, so any real $c$-cluster is a $c$-exchangeable pair with $\delta=\alpha_1+2\alpha_2$ (taking $\aff=2$).
Thus $\set{-\alpha_1,-\alpha_2}$ is an exchangeable pair, but $\cm{-\alpha_2}{\delta}c=2$ for either choice of $c$.
Second, consider $A^{(2)}_4$ with $\aff=3$, so that $\delta=\alpha_1+2\alpha_2+2\alpha_3$, and take $c=s_1s_2s_3$.
Then $\Fan_c(\RS)$ is the same as the fan shown in \cref{fig:fan}. 
The smallest triangle shown has vertices $-\alpha_1$ (red), $-\alpha_2$ (green), and $-\alpha_3$ (blue), and is adjacent to a triangle with $-\alpha_1$, $\alpha_2$ (cyan), and $-\alpha_3$.
Since $\alpha_2$ is compatible with $\delta$, the pair $\set{-\alpha_1,-\alpha_3}$ is $c$-exchangeable with $\delta$, but $\cm{-\alpha_3}\delta c=2$.
\end{example}

We now prepare to prove \cref{exchangeable,pair exchangeable}.
The following proposition is immediate by \cref{prop:inductive_cluster}.

\begin{proposition}\label{prop:inductive_exch}
Suppose $\RS$ is of affine type and let $\alpha,\beta\in\AP{c}$.
\begin{enumerate}[\quad \upshape (1)]
\item \label{exch sigma}
If $s$ is an initial or final reflection in $c$ then $\alpha$ and $\beta$ are $c$-exchangeable if and only if $\sigma(\alpha)$ and $\sigma(\beta)$ are $scs$-exchangeable.
\item \label{exch tau}
$\alpha$ and $\beta$ are $c$-exchangeable if and only if $\tau_c(\alpha)$ and $\tau_c(\beta)$ are $c$-exchangeable.
\end{enumerate}
The same assertions hold for real $c$-exchangeability.
Similarly, a pair $\set{\alpha,\beta}$ is $c$-exchangeable with $\delta$ if and only if $\set{\tau_c\alpha,\tau_c\beta}$ is $c$-exchangeable with $\delta$ and if and only if $\set{\sigma_s\alpha,\sigma_s\beta}$ is $scs$-exchangeable with $\delta$.
\end{proposition}

Parts of the proofs in this section are deferred to \cref{type exch}, where they are checked using the classifications of affine root systems.
One can reduce the number of types that must be checked by performing a rescaling.
\cref{Fanc scaling} says that, if $\RS'$ is a rescaling of $\RS$, then $\Fan_c(\RS)=\Fan_c(\RS')$.
In particular, the notions of $c$-exchangeability with respect to $\RS$ and $\RS'$ coincide.
The condition $\cm\alpha\beta c=1=\cm\beta\alpha c$ is also, in most cases, robust under rescaling.

\begin{proposition}\label{exchangeable rescale}
  Suppose $\RS$ and $\RS'$ are affine root systems related by a rescaling $\beta\mapsto\beta'=\lambda_\beta\beta$.
  Then for real roots $\alpha$ and $\beta$ in $\APre{c}$:
  \begin{enumerate}[\quad \upshape (1)]
    \item
      the conditions $\cm\alpha\beta c=1=\cm\beta\alpha c$ and $\cm{\alpha'}{\beta'}c=1=\cm{\beta'}{\alpha'}c$ are equivalent;
    \item
      the conditions $\cm\alpha\delta c=1=\cm\delta\alpha c$ and $\cm{\alpha'}{\delta'}c=1=\cm{\delta'}{\alpha'}c$ are equivalent except possibly when $\RS$ or $\RS'$ is of type $A^{(2)}_{2k}$.
  \end{enumerate}
\end{proposition}
\begin{proof}
If $\cm\alpha\beta c=1=\cm\beta\alpha c$, then $\frac{\lambda_\alpha}{\lambda_\beta}\cm{\alpha'}{\beta'}c=1=\frac{\lambda_\beta}{\lambda_\alpha}\cm{\beta'}{\alpha'}c$  by \cref{prop:scaling}.
Thus $\frac{\lambda_\beta}{\lambda_\alpha}=\cm{\alpha'}{\beta'}c$ and $\frac{\lambda_\alpha}{\lambda_\beta}=\cm{\beta'}{\alpha'}c$.
In particular, both $\frac{\lambda_\beta}{\lambda_\alpha}$ and $\frac{\lambda_\alpha}{\lambda_\beta}$ are positive integers.
Therefore $\lambda_\alpha=\lambda_\beta$ and $\cm{\alpha'}{\beta'}c=1=\cm{\beta'}{\alpha'}c$.
The symmetric argument proves the converse.

For the second assertion, by \eqref{compat delta U}, we can assume that $\alpha\not\in\APT{c}$.
\cref{tau prop} says that $\alpha$ is in the $\tau_c$-orbit of a negative simple root.
Thus by \eqref{compat tau}, we may as well take $\alpha$ to be $-\alpha_i$ for some $i$.
In this case $\cm\alpha\delta c=[\delta:\alpha_i]$ by \eqref{compat base} and $\cm\delta\alpha c=[\delta\ck:\alpha_i\ck]$ by \eqref{compat cobase}.
If neither $\RS$ nor $\RS'$ is of type $A^{(2)}_{2k}$, then inspection of the classification of root system of affine type (see \cite[Tables Aff 1--3]{Kac90}), reveals that the property that $[\delta:\alpha_i]=1=[\delta\ck:\alpha_i\ck]$ is preserved by rescaling.  
\end{proof}

We now begin the proof of \cref{exchangeable}\eqref{exch 1}.
We prove all of one direction of the assertion and reduce the other direction to a proposition that we will prove type-by-type in \cref{type exch}.

\begin{proof}[Proof of \cref{exchangeable}\eqref{exch 1}]
As in the proof \cref{clus exp thm}, one could prove this for finite and affine type together by induction on rank, but instead, we use the finite case, which follows from \cite[Proposition~3.5]{FoZe03a} and \cite[Proposition~4.10]{MRZ}.

Suppose $\alpha,\beta$ in $\APre{c}$ are $c$-exchangeable.
If they are $c$-real-exchangeable, then let $C$ and $C'$ be real $c$-clusters with $\alpha\in C$, with $\beta\in C'$, and with $C\setminus\{\alpha\}=C'\setminus\{\beta\}$.
\cref{cluster properties}\eqref{at least 2} says that $C$ contains at least two roots in the $\tau_c$-orbits of negative simple roots.
Thus $C\setminus\set{\alpha}$ contains at least one root in the $\tau_c$-orbit of a negative simple root.
By \cref{prop:inductive_exch}\eqref{exch tau} and \eqref{compat tau}, we may as well assume that $C\setminus\set{\alpha}$ contains a negative simple root $-\alpha_i$.
Then \cref{prop:inductive_cluster} implies that $C\setminus\set{-\alpha_i}$ and $C'\setminus\set{-\alpha_i}$ are $c'$-clusters in a proper parabolic root subsystem $\RS'$ of $\RS$, where $c'$ is the restriction of $c$.
The finite case result then says that $\cm\alpha\beta{c'}=1=\cm\beta\alpha{c'}$.
\cref{compat restrict} implies that $\cm\alpha\beta c=1=\cm\beta\alpha c$.

If $\alpha$ and $\beta$ are \emph{not} $c$-real-exchangeable, then there exist $c$-clusters $C$ and $C'$, not both real, with $\alpha\in C$, with $\beta\in C'$, and with $C\setminus\{\alpha\}=C'\setminus\{\beta\}$.
But then $C$ and $C'$ have the same cardinality, so \cref{cluster properties} implies that both are imaginary and contained in $\APT{c}$.
In particular $\cm\alpha\beta c= \cmcirc\alpha\beta c$ and  $\cm\beta\alpha c=\cmcirc\beta\alpha c$.
Since $\alpha$ and $\beta$ are $c$-exchangeable, they are not $c$-compatible, so by symmetry, to show that $\cm\alpha\beta c=1=\cm\beta\alpha c$, it suffices to rule out $\cmcirc\alpha\beta c=2$.

Suppose $\cmcirc\alpha\beta c=2$.
Then $\alpha$ and $\beta$ are overlapping or adjacent on both sides and $\alpha$ has two distinct adjacent roots.
Let $\ell$ be the number of simple roots of $\RST{c}$ in the overlap of $\SuppT(\alpha)$ and $\SuppT(\beta)$ on one side and let $m$ be the number of simple roots in the overlap on the other side.
Let $p$ be the number of simple roots (in the component) not contained in or adjacent to $\SuppT(\alpha)$ and let $q$ be the number of simple roots in that are not contained in or adjacent to $\SuppT(\beta)$.
Using \cref{compatible in tubes}, we see that the maximum size of a set of pairwise $c$-compatible real roots in the component, $c$-compatible with both $\alpha$ and $\beta$, is $\ell+m+p+q$, which is $4$ less than the size of the component if $\beta$ has two adjacent roots, or $3$ less if $\beta$ has one adjacent root.
The maximum size of a set of pairwise $c$-compatible roots in each other component is $1$ less than the size of the component.
Since by \cref{Up details}, the rank of $\RST{c}$, minus the number of components, is $n-2$, the maximum size of a set of pairwise $c$-compatible real roots in $\RST{c}$, $c$-compatible with both $\alpha$ and $\beta$, is $n-4$.  
Allowing $\delta$ in the set, we conclude that $C\setminus\set{\alpha}$ can contain at most $n-3$ roots.
This contradicts \cref{cluster properties}\eqref{Qc n-1}, so $\cmcirc\alpha\beta c\neq2$.

We have proved one direction of \cref{exchangeable}\eqref{exch 1}.
To prove the remaining direction of \cref{exchangeable}\eqref{exch 1}, suppose ${\cm\alpha\beta c=1=\cm\beta\alpha c}$.
If there exists a root $\gamma$ in the $\tau_c$-orbit of some negative simple root $-\alpha_i$ that is $c$-compatible with both $\alpha$ and $\beta$, then by \cref{prop:inductive_exch}\eqref{exch tau} and \eqref{compat tau}, we may as well assume that some $-\alpha_i$ is $c$-compatible with both $\alpha$ and $\beta$.
By \eqref{compat base}, $\alpha$ and $\beta$ are in a proper standard parabolic root subsystem of $\RS$ and \cref{compat restrict} says that their compatibility degree in the root subsystem is $1$ in both directions.
By the result for finite type, $\alpha$ and $\beta$ are exchangeable in the root subsystem.
Adjoining~$-\alpha_i$ to the two relevant clusters, we see that $\alpha$ and $\beta$ are $c$-exchangeable in $\AP{c}$ as well.

Suppose that $\alpha$ and $\beta$ are both in $\APTre{c}$.
Then there are two possibilities allowed by the equality $\cm\alpha\beta c=1=\cm\beta\alpha c$.
One possibility is that $\alpha$ and $\beta$ are overlapping or adjacent on only one side.
But then $\SuppT(\alpha,\beta)$ is not component-full, so \cref{prop in the tubes} implies that there is a root in the $\tau_c$-orbit of a negative simple root that is $c$-compatible with both $\alpha$ and $\beta$, so we have already covered this possibility.
The other possibility, pictured in \cref{funny exch fig} with the same drawing conventions as in \cref{tubes table}, is that $\SuppT(\alpha)$ and $\SuppT(\beta)$ each omit a single simple root of the component of $\RST{c}$ that contains them both.
\begin{figure}
\includegraphics{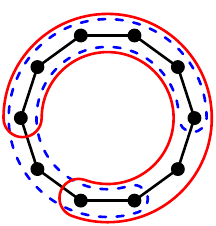}
\caption{One of the cases discussed in the proof of \cref{exchangeable}.}
\label{funny exch fig}
\end{figure}
In this case, let $p$ and $q$ be the numbers of simple roots in the two components of $\SuppT(\alpha)\cap\SuppT(\beta)$.
Using \cref{compatible in tubes}, we can construct a set of $p+q$ pairwise $c$-compatible real roots in the component, each $c$-compatible with both $\alpha$ and $\beta$.
This is $2$ less than the size of the component.
In each other component of size $k$, we can find a set $k-1$ of pairwise $c$-compatible roots.
Since the rank of $\RST{c}$ minus the number of components is $n-2$, we have constructed a set $C$ of $n-3$ pairwise $c$-compatible real roots in $\RST{c}$, each $c$-compatible with both $\alpha$ and $\beta$.  
By \cref{cluster properties,compat delta U}, the sets $C\cup\set{\alpha,\delta}$ and $C\cup\set{\beta,\delta}$ witness that $\alpha$ and $\beta$ are $c$-exchangeable.

The remaining case is where one or both of $\alpha$ and $\beta$ are not in $\APTre{c}$, or equivalently by \cref{tau prop}, one or both is in the $\tau_c$-orbit of a negative simple root.
By \cref{aff conj}, \cref{prop:inductive_exch}\eqref{exch sigma}, \cref{sigma pos real in tubes}, and \eqref{compat sigma}, we can choose one Coxeter element $c$ in each conjugacy class for which to check this remaining case.
Without loss of generality (by switching the names $\alpha$ and $\beta$ for the roots), $\alpha$ is in the $\tau_c$-orbit of a negative simple root, and by \cref{prop:inductive_exch}\eqref{exch tau} and \eqref{compat tau}, we may as well assume that $\alpha$ is a negative simple root $-\alpha_j$.
If $K(\gamma_c,\beta)>0$, then since $\eigenspace{c}$ is defined by the equation $K(\gamma_c,\,\cdot\,)=0$ and since $[\beta:\alpha_j]=1$, \cref{tau prop}\eqref{tauc fin} says that $\beta$ is in an infinite $\tau_c$-orbit, so \cref{tau prop}\eqref{tidily} says that $\tau_c^m(\beta)$ is a negative simple for some $m<0$.
\cref{tau prop}\eqref{tidily} also says that $K(\gamma_c,\tau_c^m(-\alpha_j))<0$.
Thus, up to applying $\tau_c$ several times and then switching $\alpha$ and $\beta$ and changing $j$, we can assume that $\alpha=-\alpha_j$ and $K(\gamma_c,\beta)\le0$.
Since we have already handled the case where some root in the $\tau_c$-orbit of some negative simple root is $c$-compatible with both $\alpha$ and $\beta$, by \eqref{compat base} and \eqref{compat tau} we may as well assume that $\Supp(\tau_c^m(\alpha),\tau_c^m(\beta))$ is full for every integer $m$.
In particular, $\beta$ has full support.
By \cref{compat commute}, we can rewrite the hypothesis $\cm\alpha\beta c=1=\cm\beta\alpha c$ as $\cm\alpha\beta c=1=\cm{\alpha\ck}{\beta\ck}c$, and since $\alpha=-\alpha_j$, this is $[\beta:\alpha_j]=[\beta\ck:\alpha\ck_j]=1$.
We complete the proof of \cref{exchangeable}\eqref{exch 1} by verifying that this final case cannot occur.

\cref{exchangeable rescale} allows us to restrict our attention to the standard affine root systems.
Thus the following proposition completes the proof.
\end{proof}

\begin{proposition}\label{can be completed}
Suppose $\RS$ is of standard affine type.
In each conjugacy class of Coxeter element, there exists $c$ such that the following assertion holds:
For each $j\in\set{1,\ldots,n}$, there does not exist $\beta\in\APre{c}$ with full support such that
\begin{enumerate}
\item $[\beta:\alpha_j]=[\beta\ck:\alpha\ck_j]=1$,
\item $\Supp(\tau_c^m(-\alpha_j),\tau_c^m(\beta))$ is full for all integers $m$.
\item $K(\gamma_c,\beta)\le0$.
\end{enumerate}
\end{proposition}

\cref{can be completed} is proved in \cref{type exch}.
We proceed to prove the remainder of \cref{exchangeable}.
We will use \cref{exchangeable}\eqref{exch 1} to prove the remaining parts, but when we prove \cref{can be completed} (and thus complete the proof of \cref{exchangeable}\eqref{exch 1}), we do not use \cref{exchangeable}.

\begin{proof}[Proof of \cref{exchangeable}(\ref{exch Delta}--\ref{exch component-full})]
Suppose $\alpha$ and $\beta$ are $c$-exchangeable roots but are not $c$-real-exchangeable.
In the proof of \cref{exchangeable}\eqref{exch 1}, we showed that in this case $\alpha,\beta\in\APTre{c}$.
We want to show that $\SuppT(\alpha,\beta)$ is component-full and that $\alpha+\beta$ is in $\mathring\Delta_c$.
If $\SuppT(\alpha,\beta)$ is not component-full, then $R=\set{\alpha,\beta}$ satisfies the hypotheses of \cref{prop in the tubes}, so there exists a root $\tau_c^k(-\alpha_j)$ for some $j\in\set{1,\ldots,n}$ and $k\in\integers$ that is $c$-compatible with both $\alpha$ and $\beta$.
Thus by \eqref{compat tau}, $\tau_c^{-k}\alpha$ and $\tau_c^{-k}\beta$ are both $c$-compatible with $-\alpha_j$, and thus by \eqref{compat base}, both are in the proper parabolic root subsystem deleting $\alpha_j$.
Writing $c'$ for the restriction of $c$ to the parabolic subgroup, \cref{exchangeable}\eqref{exch 1} and \cref{compat restrict} imply that $\cm{\tau_c^{-k}\alpha}{\tau_c^{-k}\beta}{c'}=1=\cm{\tau_c^{-k}\beta}{\tau_c^{-k}\alpha}{c'}$.
The finite-type result implies that $\tau_c^{-k}\alpha$ and $\tau_c^{-k}\beta$ are $c'$-exchangeable in the parabolic subgroup.
Inserting $-\alpha_j$ into the $c'$-clusters that witness the $c'$-exchangeability, we see that $\tau_c^{-k}\alpha$ and $\tau_c^{-k}\beta$ are $c'$-real-exchangeable.
Now \cref{prop:inductive_exch}\eqref{exch tau} implies that $\alpha$ and $\beta$ are $c$-real-exchangeable.
This contradiction shows that $\SuppT(\alpha,\beta)$ is component-full.
Therefore $\alpha+\beta-\delta$ is a nonnegative linear combination of positive roots of $\RST{c}$.
Since $\delta$ is in $\mathring\Delta_c$ and the positive roots of $\RST{c}$ are in $\Delta_c$, we conclude that $\alpha+\beta$ is in $\mathring\Delta_c$.

We now continue to assume that $\alpha$ and $\beta$ are $c$-exchangeable but now assume that they \emph{are} $c$-real-exchangeable.
To complete the proof, we need to show that $\alpha+\beta\not\in\mathring\Delta_c$ and to show that if $\alpha,\beta\in\APTre{c}$, then $\SuppT(\alpha,\beta)$ is not component-full.
Let $C$ and $C'$ be real $c$-clusters with $\alpha\in C$, with $\beta\in C'$, and with $C\setminus\{\alpha\}=C'\setminus\{\beta\}$.
\cref{cluster properties}\eqref{at least 2} says that $C$ contains at least two roots in the $\tau_c$-orbits of negative simple roots.
Thus $C\setminus\set{\alpha}$ contains at least one root in the $\tau_c$-orbit of a negative simple root.
Therefore, Properties~\eqref{compat base} and~\eqref{compat tau} imply that $\Supp(\tau_c^k(\alpha),\tau_c^k(\beta))$ is not full for some integer $k$.

Suppose for the sake of contradiction that $\alpha,\beta\in\APTre{c}$ and that $\SuppT(\alpha,\beta)$ is component-full.
Then by \cref{c on tUp components}, $\SuppT(\tau_c^k(\alpha),\tau_c^k(\beta))$ is also component-full.
But then $\tau_c^k(\alpha)+\tau_c^k(\beta)$ is $\delta$ plus a nonnegative combination of positive roots, and since $\Supp(\delta)$ is full, we conclude that $\Supp(\tau_c^k(\alpha),\tau_c^k(\beta))$ is full, and this is our contradiction.

It remains to show that  $\alpha+\beta\not\in\mathring\Delta_c$.
The fact that $\tau_c^k(C\setminus\set{\alpha})$ contains a negative simple root for some $k$ means that, reading indices modulo $n$, the set $\sigma_1\cdots\sigma_{nk}(C\setminus\set{\alpha})$, if $k>0$, or $\sigma_{-nk+1}\cdots\sigma_0(C\setminus\set{\alpha})$, if $k<0$, contains a negative simple root.  
Choose $\ell$ with the smallest absolute value such that $\sigma_1\cdots\sigma_\ell(C\cup\set{\beta})$, if $k\ge0$, or $\sigma_{-\ell}\cdots\sigma_0(C\cup\set{\beta})$, if $k<0$, contains a negative simple root, and write $\Sigma$ for $\sigma_1\cdots\sigma_\ell$ or $\sigma_{-\ell}\cdots\sigma_0$.
Since we choose $\ell$ with smallest absolute value, since $\sigma_i$ acts as $s_i$ on positive roots, and since $\pm\alpha_i$ are the only roots whose sign changes under the action of $s_i$, we see that $(\Sigma(C\cup\set{\beta}))\cap(-\Simples)=\set{-\alpha_\ell}$.
If neither $\Sigma(\alpha)$ nor $\Sigma(\beta)$ is $-\alpha_\ell$, then $\Sigma(\alpha)$ and $\Sigma(\beta)$ are both in the standard parabolic subgroup of $\RS$ that omits the simple root $\alpha_\ell$.
If $\Sigma(\alpha)$ or $\Sigma(\beta)$ is $-\alpha_\ell$, then, since we are assuming that $\alpha$ and $\beta$ are $c$-exchangeable and have proved that $\cm\alpha\beta c=1=\cm\beta\alpha c$, the other of $\Sigma(\alpha)$ or $\Sigma(\beta)$ has $\alpha_\ell$-coordinate $1$ in its simple-root expansion by \eqref{compat base}.
In either case, $\Sigma(\alpha)+\Sigma(\beta)$ is in the span of $\Simples\setminus\set{\alpha_\ell}$.
Write $\tilde c$ for $s_1\cdots s_\ell cs_\ell\cdots s_1$, if $\ell\ge0$, or for $s_{-\ell}\cdots s_0cs_0\cdots s_\ell$, if $\ell<0$.
Since all vectors in $\mathring\Delta_{\tilde c}$ have strictly positive simple-root coordinates, we see that $\Sigma(\alpha)+\Sigma(\beta)$ is not in $\mathring\Delta_{\tilde c}$.
Since each $\sigma_i$ that makes up $\Sigma$ acts as $s_i$, \cref{scs Delta} implies that $\alpha+\beta$ is not in $\mathring\Delta_c$.
This completes the proof of \cref{exchangeable}(\ref{exch Delta}--\ref{exch component-full}).
\end{proof}

\begin{proof}[Proof of \cref{exchangeable}\eqref{exch delta}]
Suppose $C$ is a $c$-cluster with $\delta\in C$.
Then by definition $C$ is an imaginary cluster, so \cref{cluster properties} says that $|C|=n-1$.
If $\beta$ is any real root in $\APre{c}$, then $(C\setminus\set{\delta})\cup\set{\beta}$ has $\le n-1$ roots, and thus is not a real cluster by \cref{cluster properties}.
It is also not an imaginary cluster, since it does not contain~$\delta$.
We see that $\beta$ and $\delta$ are not exchageable.
\end{proof}

Finally, we reduce \cref{pair exchangeable} to a proposition that will be proved in \cref{type exch}.

\begin{proof}[Proof of \cref{pair exchangeable}]
Suppose $\cm\alpha\delta c=1=\cm\delta\alpha c$.
Since in particular $\alpha$ and $\delta$ are not $c$-compatible, by \cref{delta c compat}, $\alpha$ is in the $\tau_c$-orbit of a negative simple root.
As in the proof of \cref{exchangeable}\eqref{exch 1}, we can assume that $\alpha=-\alpha_j$ for some $j$.
As before, the assumption that $\cm\alpha\delta c=1=\cm\delta\alpha c$ means that $[\delta:\alpha_j]=1=[\delta\ck:\alpha_j\ck]$.
Inspection of the classification (for example in \cite[Tables Aff 1--3]{Kac90}) shows that $[\delta:\alpha_j]=1=[\delta\ck:\alpha_j\ck]$ happens if and only if $\alpha_j=\alpha_\aff$ or there is some diagram automorphism taking $\alpha_j$ to $\alpha_\aff$.  
Since we will be proving this direction of the theorem for all possible $c$, we may as well take $\alpha_j=\alpha_\aff$.

As before, let $\beta_1,\ldots,\beta_{n-2}$ be the simple roots of $\RSTfin{c}$.
Then the set $\TravReg{c}$ is a collection of pairwise $c$-compatible roots in $\APTre{c}$.
(Within each component they are all nested.)
In particular, each of these is $c$-compatible with $\alpha=-\alpha_\aff$.
By \cref{clus def OK} and \cref{cluster properties}\eqref{Q n},  $\TravReg{c}\cup\set\alpha$ is contained in some real cluster $C'=\TravReg{c}\cup\set{\alpha,\beta}$ for some $\beta$.
Writing $C$ for the cluster $\TravReg{c}\cup\set{\delta}$, we have $C\setminus\set{\delta}=C'\setminus\set{\alpha,\beta}$.
That is, $\set{\alpha,\beta}$ is $c$-exchangeable with~$\delta$.

Conversely, given $\alpha\in\AP{c}$, suppose there exists $\beta\in\AP{c}$ such that the pair $\set{\alpha,\beta}$ is $c$-exchangeable with $\delta$.
Specifically, suppose $C$ is an imaginary $c$-cluster and $C'$ is a $c$-cluster with $\alpha,\beta\in C'$, such that $C\setminus\set{\delta}=C'\setminus\set{\alpha,\beta}$.
We want to prove that $\cm\alpha\delta c=1=\cm\delta\alpha c$.
By \cref{prop:scaling,exchangeable rescale}, it is enough to check the theorem only for the standard affine root systems.
By \cref{prop:inductive_exch} and \eqref{compat sigma}, it is enough to check only one choice of $c$ in each conjugacy class.

Because $C\setminus\set{\delta}=C'\setminus\set{\alpha,\beta}$, in particular, $\alpha\not\in C$.
Since $C$ is a cluster (a \emph{maximal} set of pairwise $c$-compatible roots in $\AP{c}$), and $\alpha$ is compatible with any root in $C\setminus\set{\delta}$, we deduce that the root $\alpha$ is not $c$-compatible with $\delta$.
In particular \cref{delta c compat} says that $\alpha$ is in the $\tau_c$-orbit of a negative simple root $-\alpha_j$.
By \eqref{compat tau} and the fact that $\tau_c$ fixes $\delta$, we may as well assume that $\alpha=-\alpha_j$.
By \eqref{compat base} and \eqref{compat cobase}, the assertion that $\cm{-\alpha_j}\delta c=1=\cm\delta{-\alpha_j}c$ is equivalent to the assertion that $[\delta:\alpha_j]=1=[\delta\ck:\alpha_j\ck]$.  
As mentioned above, the latter assertion is equivalent to the assertion that $\alpha_j=\alpha_\aff$ or there is some diagram automorphism taking $\alpha_j$ to $\alpha_\aff$.
The roots $C\setminus\set\delta$ are all in $\APTre{c}$ by \cref{delta c compat}.
Also, \cref{cluster properties} implies that $C\setminus\set\delta$ contains exactly $n-2$ roots, and these roots are linearly independent.
In particular, $|\SuppT(C\setminus\set\delta)|\ge n-2$.
As mentioned above in the proof of most of \cref{exchangeable}\eqref{exch 1}, the rank of $\RST{c}$, minus the number of components, is $n-2$, so \cref{compat not full} says that $\SuppT(C\setminus\set\delta)$ consists of all but one simple root in each component of $\RST{c}$.
Each root in $C\setminus\set\delta$ is $c$-compatible with $-\alpha_j$, so $\alpha_j\not\in\Supp(C\setminus\set\delta)$, which further implies that $\alpha_j\not\in\Supp(\SuppT(C\setminus\set\delta))$.
In \cref{type exch}, we complete the proof by proving the following proposition.
\end{proof}

\begin{proposition}\label{must be affine}
Suppose $\RS$ is a standard affine root system.
Then there exists a Coxeter element $c$ such that the following assertion holds:
If $R$ is a set of simple roots of $\RST{c}$ consisting of all but one simple root in each component of $\RST{c}$, and if $-\alpha_j$ is a simple root of $\RS$ with $\alpha_j\not\in\Supp(R)$, then $\alpha_j=\alpha_\aff$ or there is some diagram automorphism taking $\alpha_j$ to $\alpha_\aff$.
\end{proposition}

\section{Type-by-type arguments for exchangeability}\label{type exch}
In this section, we prove \cref{can be completed,must be affine} by a type-by-type check in the classification of affine root systems.
Neither proof relies on results of \cref{exch sec}.  

\begin{proof}[Proof of \cref{can be completed}]  
We argue type-by-type, and we take $c=s_1\cdots s_n$ according to the labeling in \cref{tab:type-by-type}.
Exceptional types are checked by computer.

Recall that $\alpha_\aff=\alpha_n$ is every case.
Since $\RS$ is of standard affine type, every root $\beta$ of $\RS$ is of the form $\beta'+k\delta$ for some $\beta'\in\RSfin$.
In every case, we assume conditions (1) and (3).
To eliminate cases, we use the observation that no root in $\APre{c}$ is in $\eigenspace{c}$ and has $\SuppT(\beta)$ component-full.
Thus for each possible $\beta$, we either prove that $\beta\in \eigenspace{c}$ and $\SuppT(\beta)$ is component-full or we prove that (2) fails.
In the simply laced cases, we scale $\RS$ so that roots and co-roots coincide.
We rewrite condition~(3) as $\br{\phi_c,\beta}\le0$ for $\phi_c$ as in \cref{phic sign}.
We have $\beta\in \eigenspace{c}$ if and only if $\br{\phi_c,\beta}=0$.
Throughout, we use the fact that $\delta$ is fixed by $c$.
We also use the fact that if each of the roots $\beta,c(\beta),\ldots,c^p(\beta)$ is positive for some $p$, then $c^p(\beta)=\tau_c^p(\beta)$.

\noindent\textbf{Case $A_{n-1}^{(1)}$:}  
In this case, \cref{phic sign} implies that, up to positive scaling, $\phi_c$ is $\rho_n-\rho_1$.
Thus we are looking for $\beta$ with $[\beta:\alpha_1]\ge[\beta:\alpha_n]$.
A root $\beta$ with full support, at least one simple-root coordinate $=1$ and $[\beta:\alpha_1]\ge[\beta:\alpha_n]$ is of the form $\beta'+\delta$ for $\beta'\in\RSpos_\fin$ or $-\beta'+2\delta$ for $\beta'\in\RSpos_{\set{2,\ldots,k}}$ or $\beta'\in\RSpos_{\set{k+1,\ldots,n-1}}$.
But in the latter case, $\beta\in \eigenspace{c}$ and $\SuppT(\beta)$ is component-full.
Similarly, if $\beta=\beta'+\delta$ for $\beta'\in\RSpos_\fin$ such that $[\beta':\alpha_1]=0$, again $\beta\in \eigenspace{c}$ and $\SuppT(\beta)$ is component-full.
Thus we can assume that $\beta=\beta'+\delta$ for $\beta'\in\RSpos_\fin$ and $[\beta':\alpha_1]=1$.

We write $\beta'=\alpha_1+\cdots+\alpha_\ell+\alpha_{k+1}+\cdots+\alpha_{k+m}$, and without loss of generality, we can take $m<\ell$.
(Otherwise, rename the roots $\alpha_2,\ldots,\alpha_k$ as $\alpha_{k+1},\ldots,\alpha_{n-1}$ and vice versa, replacing $k$ by $n-k$.)
We compute $c^{-m-1}(\beta')=-\alpha_{\ell-m}-\cdots-\alpha_k-\alpha_n$.
Since $\delta$ is fixed by $c$, we see that $[c^{-m-1}(\beta):\alpha_{\ell-m}]=0$.
Each of the roots $c^{-i}(\beta)$ for $i=0,\ldots,m+1$ is positive, so $\tau_c^{-m-1}(\beta)=c^{-m-1}(\beta)$ and thus $[\tau_c^{-m-1}(\beta):\alpha_{\ell-m}]=0$.

Since $[\beta:\alpha_j]=1$, we have $j\in\set{\ell+1,\ldots,k}\cup\set{k+m+1,\ldots,n}$.
If $j$ is in $\set{\ell+1,\ldots,k}$, then $\tau_c^{-m-1}(-\alpha_j)=\alpha_{j-m}+\cdots+\alpha_k+\alpha_{n-m}+\cdots+\alpha_n$.
If $j$ is in $\set{k+m+1,\ldots,n}$, then $\tau_c^{-m-1}(-\alpha_j)=\alpha_{k-m+1}+\cdots+\alpha_k+\alpha_{j-m}+\cdots+\alpha_n$.
Since $\ell<j<k+1$, in either case we have $[\tau_c^{-m-1}(-\alpha_j):\alpha_{\ell-m}]=0$, contradicting condition (2).

\noindent\textbf{Case $B_{n-1}^{(1)}$:}  
\cref{phic sign} implies that, up to positive scaling, $\phi_c$ is $\rho_n+\rho_{n-1}-2\rho_1$.  
Recalling that $\rho_i$ is in the basis dual to the simple \emph{co-roots} and writing $\rho\ck_i$ for elements of the basis dual to the simple roots, we have $\rho\ck_i=d_i^{-1}\rho_i$, where the $d_i$ are the symmetrizing constants described in \cref{sec:root_systems}.
We can take $d_1=\frac12$ and $d_i=1$ for $i\neq1$, so up to positive scaling $\phi_c$ is $\rho\ck_n+\rho\ck_{n-1}-\rho\ck_1$.
Thus we require $[\beta:\alpha_1]\ge[\beta:\alpha_{n-1}]+[\beta:\alpha_n]$.
We have $\delta=2\sum_{i=1}^{n-2}\alpha_i+\alpha_{n-1}+\alpha_n$.
The positive roots in $\RSfin$ are exactly the roots $\sum_{i=\ell}^m\alpha_i$ for $2\le\ell\le m\le n-1$ or $2\sum_{i=1}^\ell\alpha_i+\sum_{i=\ell+1}^m\alpha_i$ for $0\le\ell<m\le n-1$.

A positive root $\beta$ with full support, at least one simple-root coordinate $=1$ and $[\beta:\alpha_1]\ge[\beta:\alpha_{n-1}]+[\beta:\alpha_n]$ is of one of the following forms:
\begin{enumerate}[\qquad(a)]
\item $\beta'+\delta$ for $\beta'\in\RSpos_\fin$ with $[\beta':\alpha_1]\ge[\beta':\alpha_{n-1}]$.
\item $-\beta'+\delta$ for $\beta'=\sum_{i=\ell}^m\alpha_i$ for $2\le\ell\le m\le n-2$.
\item $-\beta'+2\delta$ for $\beta'\in\RSpos_\fin$ with $[\beta':\alpha_{n-1}]=1$ and $[\beta':\alpha_1]\le 1$.
\end{enumerate}

If $\beta$ is of form (a) and $\beta'\in\RSpos_{\set{2,\ldots,n-2}}$, then $\beta$ is in $\eigenspace{c}$ and has component-full tube support, so we can assume that either $[\beta':\alpha_1]$ or $[\beta':\alpha_{n-1}]$ is positive.
But since $[\beta':\alpha_1]\ge[\beta':\alpha_{n-1}]$, we have $[\beta':\alpha_1]\ge1$.
Thus $\beta'$ is of the form $2\sum_{i=1}^\ell\alpha_i+\sum_{i=\ell+1}^m\alpha_i$ for $0\le\ell<m\le n-1$.
If $\beta'$ is $\sum_{i=1}^{n-1}\alpha_i$ (that is, if $\ell=0$ and $m=n-1$), then $\beta'=\beta_{n-2}$ in \cref{tab:type-by-type}, so again $\beta$ is in $\eigenspace{c}$ and has component-full tube support, and we rule out his case.
If $\ell=0$ and $m<n-1$, then $c^{-m}(\beta')=-\sum_{i=1}^n\alpha_i$, and thus  $c^{-m}(\beta)=c^{-m}(\beta')+\delta=\sum_{i=1}^{n-2}\alpha_i$,
Since all these roots are positive, we have $\tau_c^{-m}(\beta)=\sum_{i=1}^{n-2}\alpha_i$.
We have $j=n-1$ or $j=n$.
In either case, we calculate $\tau_c^{-m}(\alpha_j)$ and find that either its $\alpha_{n-1}$- or $\alpha_n$-coefficient is zero.
Similarly, if $\ell>0$, we compute $c^{-\ell}(\beta')=-\sum_{i=m-\ell+1}^n\alpha_i$ and see that $\tau_c^{-\ell}=c^{-\ell}(\beta)$ has both its $\alpha_{n-1}$-coefficient and its $\alpha_n$-coefficient zero, while $\tau_c^{-\ell}(\alpha_j)$ has either its $\alpha_{n-1}$- or $\alpha_n$-coefficient zero.
In either case, we have found a contradiction to (2).

If $\beta$ is of form (b), then $\beta$ is in $\eigenspace{c}$ and has component-full tube support.

Finally, if $\beta$ is of form (c), then $\beta'=\sum_{i=\ell}^{n-1}$ for some $1\le\ell\le n-1$.
If $\ell=1$, then $\beta$ is in $\eigenspace{c}$ and has component-full tube support, so we assume $\ell>1$.
We compute $c^{-\ell+1}(\beta')=\delta+\alpha_q$, where $q$ is $n$ if $\ell$ is even or $q$ is $n-1$ if $\ell$ is odd.
Therefore $c^{-\ell+1}(\beta)=\delta-\alpha_q$, and since all these roots are positive, we have $\tau_c^{-\ell+1}(\beta)=\delta-\alpha_q$, which has $\alpha_q$-coefficient zero.
In this case, $j=n-1$ and we calculate $\tau_c^{-\ell+1}(-\alpha_j)=\sum_{i=n-\ell+1}^{n_2}\alpha_i+\alpha_p$, where $p$ is $n-1$ if $\ell$ is even or $p$ is $n$ if $\ell$ is odd.
Both have $\alpha_q$-coefficient zero, so we have again contradicted (2).

\noindent\textbf{Case $C_{n-1}^{(1)}$:}  
\cref{phic sign} implies that, up to positive scaling, $\phi_c$ is $\rho_n-\rho_1$.
Since the symmetrizing constants $d_1$ and $d_n$ are the same, up to scaling this is $\rho\ck_n-\rho\ck_1$, so we require $[\beta:\alpha_1]\ge[\beta:\alpha_n]$.  
We have $\delta=\sum_{i=1}^n\alpha_i+\sum_{i=2}^{n-1}\alpha_i$.
The positive roots in $\RSfin$ are exactly the roots $\sum_{i=\ell}^m\alpha_i$ for $2\le\ell\le m\le n-1$ or $\sum_{i=1}^m\alpha_i+\sum_{i=2}^\ell\alpha_i$ for $1\le\ell\le m\le n-1$.

A positive root $\beta$ with full support, at least one simple-root coordinate $=1$ and $[\beta:\alpha_1]\ge[\beta:\alpha_n]$ is either $\beta'+\delta$ for $\beta'\in\RSpos_\fin$ or $-\beta'+\delta$ for $\beta'\in\RSpos_{\set{2,\ldots,n-1}}$.  
First, assume $\beta=\beta'+\delta$ for $\beta'\in\RSpos_\fin$.
If $\beta'\in\RS_{\set{2,\ldots,n-1}}$, then $\beta\in \eigenspace{c}$ and $\SuppT(\beta)$ is component-full, so we assume that $\beta'=\sum_{i=1}^m\alpha_i+\sum_{i=2}^\ell\alpha_i$ for $1\le\ell\le m\le n-1$.
We calculate $c^{-m}(\beta')=-\sum_{i=1}^n\alpha_i$, so that $c^{-m}(\beta)=\delta-\sum_{i=1}^n\alpha_i=\sum_{i=2}^{n-1}\alpha_i$.
These roots are all positive, so $\tau_c^{-m}(\beta)=c^{-m}(\beta)$, which has $\alpha_1$-coordinate zero.
In this case, $j=n$, so we compute $\tau_c^{-m}(-\alpha_j)=2\sum_{i=n-m+1}^{n-1}+\alpha_n$, which also has $\alpha_1$-coordinate zero since $m\le n-1$.
We have a contradiction to (2) in this case.

Next, assume $\beta=-\beta'+\delta$ for $\beta'\in\RSpos_{\set{2,\ldots,n-1}}$.
Specifically, $\beta'=\sum_{i=\ell}^m\alpha_i$ for $2\le\ell\le m\le n-1$.
We compute $c^{-\ell+1}(\beta')=\sum_{i=1}^n\alpha_i+\sum_{i=2}^{m-\ell+1}\alpha_i$, so $\tau_c^{-\ell+1}(\beta)=c^{-\ell+1}(\beta)=\sum_{i=m-\ell+2}^{n-1}\alpha_i$, which has $\alpha_1$- and $\alpha_n$-coordinates zero.
In this case, either $j\in\set{1,n}$ or $j\in\set{\ell,\ell+1,\ldots,m}$.
If $j=n$, then $\tau_c^{-\ell+1}(-\alpha_j)=2\sum_{i=n-\ell+2}^{n-1}+\alpha_n$, which also has $\alpha_1$-coordinate zero.
If $j\in\set{\ell,\ldots,m}$, then $\tau_c^{-1}(-\alpha_j)=\sum_{i=j}^n\alpha_i$, so $\tau_c^{-\ell+1}(-\alpha_j)=\sum_{i=j-\ell+2}^n\alpha_i+\sum_{i=n-\ell+2}^{n-1}\alpha_i$, which also has $\alpha_1$-coordinate zero.
We have contradicted (2) except when $j=1$.
In that final case, we compute $\tau_c^{n-\ell}(\beta)=\sum_{i=m-\ell+2}^{n-1}\alpha_i$ and $\tau_c^{n-\ell}(-\alpha_1)=\alpha_1+2\sum_{i=2}^{n-\ell}\alpha_i$, and we have contradicted (2) in this final case.

\noindent\textbf{Case $D_{n-1}^{(1)}$:}  
By \cref{phic sign}, up to positive scaling $\phi_c$ is $\rho_n+\rho_{n-1}-\rho_2-\rho_1$, so we want $\beta$ with $[\beta:\alpha_1]+[\beta:\alpha_2]\ge[\beta:\alpha_{n-1}]+[\beta:\alpha_n]$.
We have $\delta=\sum_{i=1}^n\alpha_i+\sum_{i=3}^{n-2}\alpha_i$.
Each positive root in $\RSfin$ is either a sum of adjacent roots, with coefficients $1$, along a path in the diagram for $\RSfin$ or is of the form $\sum_{i=1}^{m}\alpha_i+\sum_{i=3}^\ell\alpha_i$ for some $\ell$ and $m$ with $2\le\ell<m\le n-1$.  

A positive root $\beta$ with full support, at least one simple-root coordinate $=1$ and $[\beta:\alpha_1]+[\beta:\alpha_2]\ge[\beta:\alpha_{n-1}]+[\beta:\alpha_n]$ is of one of the following three forms:

\begin{enumerate}[\qquad(a)]
\item $\beta'+\delta$ for $\beta'\in\RSpos_\fin$ with $[\beta':\alpha_1]+[\beta':\alpha_2]\ge[\beta':\alpha_{n-1}]$.
\item $-\beta'+\delta$ for $\beta'\in\RSpos_{\set{3,\ldots,n-2}}$.
\item $-\beta'+2\delta$ for $\beta'\in\RSpos_\fin$ with $[\beta':\alpha_{n-1}]=1$ and $[\beta':\alpha_1]+[\beta':\alpha_2]\le 1$.
\end{enumerate}

If $\beta$ is of form (a) and $\beta'\in\RS_{\set{3,\ldots,n-2}}$, then $\beta\in \eigenspace{c}$ and $\SuppT(\beta)$ is component-full, so we can rule out this case.
Thus at least one of $[\beta':\alpha_1]$, $[\beta':\alpha_2]$ and $[\beta':\alpha_{n-1}]$ is $1$, but since $[\beta':\alpha_1]+[\beta':\alpha_2]\ge[\beta':\alpha_{n-1}]$, we see that at least one of $[\beta':\alpha_1]$ and $[\beta':\alpha_2]$ is $1$.
By symmetry, we can assume $[\beta':\alpha_1]=1$.
If also $[\beta':\alpha_{n-1}]=1$ but $[\beta':\alpha_2]=0$, then again $\beta\in \eigenspace{c}$ and $\SuppT(\beta)$ is component-full.
Thus we need to consider two cases:  
$\beta'=\alpha_1+\sum_{i=3}^\ell\alpha_i$ for $2\le\ell\le n-2$ or $[\beta':\alpha_1]=[\beta':\alpha_2]=1$.

First, take $\beta'=\alpha_1+\sum_{i=3}^\ell\alpha_i$ for $2\le\ell\le n-2$.
In this case, $j\in\set{2,n-1,n}$.
If $\ell=2$, so that $\beta'=\alpha_1$, then $c^{-1}(\beta')=-\alpha_1-\sum_{i=3}^n\alpha_i$, so $[c^{-1}(\beta):\alpha_1]=0$.
Since $\beta$ and $c^{-1}(\beta)$ are positive, $\tau_c^{-1}(\beta)=c(\beta)$, so $[\tau_c^{-1}(\beta):\alpha_1]=0$.
But also $[\tau_c(-\alpha_j):\alpha_1]=0$ for $j\in\set{2,n-1,n-2}$, so condition (2) fails.
Set $p=1$ if $\ell$ is even or $p=2$ if $\ell$ is odd and define $q$ such that $\set{p,q}=\set{1,2}$.
If $3\le\ell\le n-2$, then we calculate $\tau_c^{-\ell+1}(\beta)=\alpha_q+\sum_{i=3}^{n-2}\alpha_i$.
Meanwhile, for $j\in\set{n-1,n}$, $\tau_c^{-\ell+1}(-\alpha_j)$ has positive coordinates only at indices $\ge n-\ell+1\ge3$.
Thus $\Supp(\tau_c^{-\ell+1}(\alpha_j),\tau_c^{-\ell+1}(\beta))$ does not contain $\alpha_p$, so condition (2) fails for $j\in\set{n-1,n}$.
Also, $\tau_c^{-1}(-\alpha_2)=\sum_{i=2}^n\alpha_i$, so $\tau_c^{-\ell+1}(-\alpha_2)$ has $\alpha_p$-coordinate zero, and condition (2) fails for $j=2$ as well.

Next suppose $[\beta':\alpha_1]=[\beta':\alpha_2]=1$.
If $[\beta':\alpha_{n-1}]=1$, then $j=n$ and $\beta'=\sum_{i=1}^{n-1}\alpha_i+\sum_{i=3}^\ell\alpha_i$ for some $\ell$ with $2\le\ell\le n-2$.  
Now set $p=n-1$ if $\ell$ is even or $p=n$ if $\ell$ is odd and define $q$ such that $\set{p,q}=\set{n-1,n}$.
We compute $\tau_c^{-\ell+1}(\beta)=\delta-\alpha_p$ and $\tau_c^{-\ell+1}(-\alpha_j)=\sum_{n-\ell+1}^{n-2}\alpha_i+\alpha_q$, and we have again contradicted condition (2).
If $[\beta':\alpha_{n-1}]=0$, then $j\in\set{n-1,n}$.
We compute $\tau_c^{-\ell+1}(\beta)=c^{-\ell+1}(\beta)=\delta-c^{-\ell+1}(\beta')$, which has $\alpha_{n-1}$- and $\alpha_n$-coordinates zero.
We already saw that $[\tau_c^{-\ell+1}(-\alpha_n):\alpha_p]=0$, and by symmetry, we conclude that $[\tau_c^{-\ell+1}(-\alpha_{n-1}):\alpha_q]=0$.
Thus whether $j$ is $n$ or $n-1$, we have once gain contradicted condition (2).

If $\beta$ is of form (b), then either $j\in\set{1,2,n-1,n}$ or $j\in\Supp(\beta')$.
Write $\beta'=\sum_{i=\ell}^m\alpha_i$ with $3\le\ell\le j\le m\le n-2$.
We compute that $[\tau_c^{n-m-1}(\beta):\alpha_{n-1}]=[\tau_c^{n-m-1}(\beta):\alpha_n]=0$.
If $j\in\set{1,2,n-1,n}$, then up to the symmetry of the diagram and replacing $c$ by $c^{-1}$, we can take $j=1$.
(We have $[\beta:\alpha_1]+[\beta:\alpha_2]=[\beta:\alpha_{n-1}]+[\beta:\alpha_n]$ in this case, so symmetry is not broken by the inequality in condition (3).)
We compute that $\tau_c^{n-m-1}(-\alpha_j)$ is a sum of roots with indices $\le n-m\le n-3$.
If $j\in\Supp(\beta')$, we compute that $\tau_c^{n-m-1}(-\alpha_j)$ is a sum of roots with indices $\le j+n-m-2\le n-2$.
In either case, we have contradicted condition (2).

If $\beta$ is of form (c) and $[\beta':\alpha_1]+[\beta':\alpha_2]=1$, then $\beta'$ is either $\beta_{n-3}$ or $\beta_{n-2}$ in the notation of \cref{tab:type-by-type}.
In either case, $\beta$ is in $\eigenspace{c}$ and $\SuppT(\beta)$ is component-full.
Thus we can take $\beta'=\sum_{i=\ell}^{n-1}\alpha_i$ for $3\le\ell\le n-1$.
In this case, $j$ must be $n-1$.
We compute $\tau_c^{\ell+2}(\beta)=\delta-\alpha_q$, so that in particular $[\tau_c^{-\ell+2}(\beta):\alpha_q]=0$.
Recalling that $j=n-1$, we calculate $\tau_c^{-\ell+2}(-\alpha_j)=\sum_{i=n-\ell+2}^{n-2}\alpha_i+\alpha_p$, contradicting condition~(2).
\end{proof}

\begin{proof}[Proof of \cref{must be affine}]
We check the proposition type-by-type using the choice of $c$ and the determination of the simple roots of $\RSTfin{c}$ given in \cref{tab:type-by-type}.
We continue the notation from the proof of \cref{lemma in the tubes} and index the components of $\RST{c}$ by indices $i$, writing $\beta_\aff^i$ for the unique simple root in the $i\th$ component that is not in $\RSfin$.
Simple-root coordinates of the roots $\beta_\aff^i$ can be computed using simple-root coordinates of $\delta$ found, for example, in \cite[Table Aff 1]{Kac90}.

\noindent\textbf{Case $A_{n-1}^{(1)}$:}  
Diagram automorphisms act transitively on the simple roots.

\noindent\textbf{Case $B_{n-1}^{(1)}$:}  
We need to show that $j\in\set{n-1,n}$.  
If $\beta_{n-2}\in R$, then $j$ can't be in $\set{1,\ldots,n-1}$, so $j=n$.
Otherwise $\beta_\aff^2=\alpha_n+\sum_{i=1}^{n-2}\alpha_i$ is in $R$, so $j=n-1$.

\noindent\textbf{Case $C_{n-1}^{(1)}$:} 
We need $j\in\set{1,n}$.
If $j\in\set{2,\ldots,n-1}$, then $\beta_{j-1}\not\in R$, so $\beta_\aff^1=\sum_{i=1}^n\alpha_i\in R$, and this is a contradiction.

\noindent\textbf{Case $D_{n-1}^{(1)}$:}  
We need $j\in\set{1,2,n-1,n}$.
If $\beta_{n-3}\in R$, then $j\in\set{2,n}$.
Otherwise, $\beta_\aff^2=\delta-\beta_{n-3}=\alpha_n+\sum_{i=2}^{n-2}\in R$, so $j\in\set{1,n-1}$.

\noindent\textbf{Case $E_6^{(1)}$:}  
We need $j\in\set{1,3,7}$.
If $\beta_5\in R$, then $j\in\set{1,3,7}$.
Otherwise, $\beta_\aff^3=\delta-\beta_5\in R$, and there is no possible $j$.

\noindent\textbf{Case $E_7^{(1)}$:}  
We need $j\in\set{2,8}$.
If $\beta_6\in R$, then $j\in\set{2,8}$.
Otherwise, $\beta_\aff^3\in R$, and there is no possible $j$.

\noindent\textbf{Case $E_8^{(1)}$:}
We need $j=9$.
If $\beta_7\in R$, then $j=9$.
Otherwise $\beta_\aff^3\in R$, and there is no possible $j$.

\noindent\textbf{Case $F_4^{(1)}$:} 
We need $j=5$.
If $\beta_2\in R$, then $j=5$.
Otherwise, $\beta_\aff^1\in R$, and there is no possible $j$.

\noindent\textbf{Case $G_2^{(1)}$:}  
We need $j=3$, which is forced if $\beta_1\in R$.
Otherwise $\beta_\aff^1\in R$, and there is no possible $j$.  
\end{proof}

\section{Connections with Cluster algebras: \texorpdfstring{$\g$-Vectors and $\d$-vectors}{g-Vectors and d-Vectors}}\label{g d sec}
In this section, we connect the real $c$-cluster fan $\Fan_c^\re(\RS)$ to the $\g$-vector fan and the $\d$-vector fan of the corresponding cluster algebra by proving \cref{nu thm,denom thm}.
We also discuss evidence for \cref{c p conj}.

\subsection{Cluster algebras notation and conventions}  
We begin by reviewing notation and establishing conventions.
We assume the basic definitions (from \cite{Fomin07}) of exchange matrices, mutations, cluster variables, clusters of cluster variables, seeds and the \newword{$\d$-vector} or denominator vector and \newword{$\g$-vector} of a cluster variable. 

We are interested in acyclic exchange matrices whose underlying Cartan matrix~$A$ is of affine type.
In this case, the exchange matrix $B$ is of affine type in the sense discussed in the introduction.
Let $\RS$ be the root system defined by the Cartan matrix $A$, and continue the notation of the rest of the paper for root systems.
Since~$B$ is acyclic, we can associate to it the Coxeter element $c$ obtained as the product of the simple reflections $S$ ordered so that $s_i$ precedes $s_j$ whenever $b_{ij}>0$.
We assume that $B$ is indexed so that~$b_{ij}\ge0$ whenever $i<j$; with this convention $c$ can be written $s_1\cdots s_n$ as in the rest of the paper.  
The notation $\A_\Sigma$ \nomenclature[azzs]{$\A_\Sigma$}{cluster algebra associated to a seed $\Sigma$} stands for the cluster algebra determined by a seed $\Sigma$ and $\A_\bullet(B)$ \nomenclature[azzzz]{$\A_\bullet(B)$}{principal-coefficients cluster algebra determined by $B$} stands for the principal-coefficients cluster algebra determined by $B$.

We write $\d(x)$ or $\d_\Sigma(x)$ \nomenclature[dzzs]{$\d_\Sigma(x)$}{denominator vector of $x$ with respect to a seed $\Sigma$}
for the denominator vector of a cluster variable $x$ with respect to the seed $\Sigma$ and similarly $\g(x)$ or $\g_\Sigma(x)$ \nomenclature[gzzs]{$\g_\Sigma(x)$}{$\g$-vector of $x$ with respect to a seed $\Sigma$}
for $\g$-vectors.  
In \cref{denom thm}, $\d$\nobreakdash-vectors are written as vectors in the root lattice:
Specifically, if $x$ has denominator $x_1^{e_1}\cdots x_n^{e_n}$, then $\d(x)$ is the vector $\sum_{i=1}^ne_i\alpha_i$.
Similarly, in \cref{nu thm}, $\g$-vectors are written as vectors in the weight lattice, and the realization of the $\g$-vector as an integer vector is obtained by taking fundamental-weight coordinates.

For each cluster in $\A_\bullet(B)$, the nonnegative linear span of the $\g$-vectors in the cluster is a full-dimensional simplicial cone, and these cones, together, form the $\g$\nobreakdash-vector fan.
The map from cluster variables to $\g$-vectors (or to rays in the $\g$\nobreakdash-vector fan) is a bijection, and the map from clusters to maximal cones in the $\g$-vector fan is also a bijection.
Indeed, the simplicial complex underlying the $\g$\nobreakdash-vector fan is isomorphic to the cluster complex.
These facts about $\g$-vector fans have been proved in various special cases before being proved in general in \cite{GHKK}.
In particular, they were proved in \cite{camb_fan,framework,afframe} in finite and affine type using the combinatorics of root systems and Coxeter groups.  

We define a piecewise linear map $\nu_c:V\to V^*$.
A similar \emph{linear} map, also called $\nu_c$, was defined in \cite[Section~5.3]{framework}, but a piecewise-linear version is more useful here.
The two maps agree on the nonnegative span of the simple roots and were only applied to positive roots in \cite{framework}.
Suppose $\beta\in V$ and $I\subseteq\set{1,\ldots,n}$ is the set of indices $i$ such that $[\beta:\alpha_i]<0$.
We write $\beta_+$ for the vector $\beta-\sum_{i\in I}[\beta:\alpha_i]\alpha_i$ and define
\nomenclature[zznc]{$\nu_c$}{piecewise linear map taking $\Fan^\re_c(\RS)$ to the $\g$-vector fan}
\[\nu_c(\beta)=-\sum_{i\in I}[\beta:\alpha_i]\rho_i-\sum_{i\not\in I}E_c(\alpha_i\ck,\beta_+)\rho_i,\]
where $E_c$ is as defined in \cref{Ec def}.
The linear version of $\nu_c$ has an inverse defined in \cite[Section~3.3]{framework}.
Using the same construction on each orthant, an inverse to the piecewise-linear map $\nu_c$ is easily constructed, and we see that $\nu_c$ is a piecewise-linear homeomorphism from $V$ to $V^*$.

\subsection{Proofs of \cref{nu thm,denom thm}}
As a first step to proving \cref{nu thm,denom thm}, we recall from \cite{afframe} the construction of the $\g$-vector fan as the doubled Cambrian fan, defined in terms of sortable elements.
We will be as brief as possible, skipping much of the combinatorics and geometry of sortable elements and Cambrian fans.

Let $c$ be a Coxeter element of a Coxeter group $W$.
The \newword{$c$-sortable elements} of $W$ can be characterized by the following recursion, together with the base condition that the identity element is $c$-sortable in any Coxeter group and for any $c$.
Suppose $w\in W$ and $s$ is initial in $c$.
\begin{itemize}
\item If $s\not\le w$, then $w$ is $c$-sortable if and only if it is contained in the parabolic subgroup generated by $S\setminus\set{s}$ and is $sc$-sortable as an element of that subgroup.
\item If $s\le w$, then $w$ is $c$ sortable if and only if $sw$ is $scs$-sortable.
\end{itemize}
Here the relation $\le$ refers to the weak order on $W$, and the condition $s\le w$ is equivalent to the condition that $w$ admits a reduced word whose first letter is $s$.

We define recursively a map $\cl_c$ from $c$-sortable elements to roots.  \nomenclature[clc]{$\cl_c$}{map from $c$-sortable elements to $c$-clusters}
The recursion was originally \cite[Lemma~8.5]{sortable}, but can be taken as a definition of the map.
Let $w$ be $c$-sortable and suppose $s$ is initial in $c$.  
\begin{itemize}
\item If $w$ is the identity, then $\cl_c(w)=-\Simples$.
\item If $s\not\le w$ then $\cl_c(w)=\set{-\alpha_s}\cup\cl_{sc}(w)$.
\item If $s\le w$ then $\cl_c(w)=\sigma_s(\cl_{scs}(sw))$.
\end{itemize}

For each $c$-sortable element, we define $\Cone_c(v)$ to be the nonnegative linear span of $\nu_c(\cl_c(v))$.
This definition is equivalent to the definition in \cite[Section~5.2]{framework} in light of \cite[Theorem~5.35]{framework}.
The \newword{$c$-Cambrian fan} $\F_c$ \nomenclature[fc]{$\F_c$}{$c$-Cambrian fan}
is the collection of all cones $\Cone_c(v)$ for $c$-sortable elements $v$, and all faces of these cones.
The \newword{doubled $c$-Cambrian fan} $\DF_c$ \nomenclature[dfc]{$\DF_c$}{$\F_c\cup(-\F_{c^{-1}})$, the doubled $c$-Cambrian fan}
is the collection consisting of the cones in $\F_{c}$ and the negations of cones in the $c^{-1}$-Cambrian fan $\F_{c^{-1}}$.
That is, $\DF_c=\F_c\cup(-\F_{c^{-1}})$.

The following theorem is \cite[Corollary~1.3]{afframe}.

\begin{theorem}\label{DF g}
Suppose $B$ is an acyclic exchange matrix whose associated Cartan matrix is of affine type and whose associated Coxeter element is $c$.
Then the doubled $c$-Cambrian fan $\DF_c$ coincides with the $\g$-vector fan for the cluster algebra $\A_\bullet(B)$.
\end{theorem}

We also need the following weak version of \cite[Corollary~4.9]{afframe}, where $|\DF_c|$ denotes the union of the cones in $\DF_c$.

\begin{theorem}\label{DF comp}
  $V^*\setminus|\DF_c|$ is an $(n-1)$-dimensional relatively open cone.
\end{theorem}

The final ingredient needed for the proof of \cref{nu thm} is the following lemma.
\begin{lemma}\label{nu nu tau}
The map $\beta\mapsto\nu_{c^{-1}}^{-1}(-\nu_c(\beta))$, applied to $\AP{c}$, coincides with $\tau_c$.
\end{lemma}

\begin{proof}
We again think of $E_c$ as the $n\times n$ matrix whose $ij$-entry is $a_{ij}$ if $i>j$, is~$1$ if $i=j$, and is $0$ if $i<j$.
But this time, we think of that matrix as taking the simple root coordinates of a vector in $V$ to the fundamental weight coordinates of a vector in $V^*$.
In particular, $E_c$ is the matrix of $\nu_c$ when the latter is applied to positive roots.
If $\beta$ is neither of the form $-\alpha_j$ nor of the form $s_n\cdots s_{j+1}\alpha_j$, then \cref{tau prop}\eqref{tauc def} implies that $\tau_c$ acts as $c$ on $\beta$ and that both $\beta$ and $\tau_c\beta$ are positive roots.
Thus $\nu_c$ acts on $\beta$ by the matrix $E_c$ and $\nu_{c^{-1}}$ acts on $\tau_c\beta=c\beta$ by the matrix $E_{c^{-1}}$, and we see by \cref{Howlett's} that $-\nu_c\beta=\nu_{c^{-1}}\tau_c\beta$.

If $\beta=-\alpha_j$ for some $j$, then $-\nu_c(\beta)=-\rho_j$ and $\nu_{c^{-1}}(\tau_c\beta)=\nu_{c^{-1}}(s_1\cdots s_{j-1}\alpha_j)$.
A simple calculation (see for example \cite[Lemma~2.9]{cyclic}) shows that
\[s_1\cdots s_{j-1}\alpha_j=\sum_{1\le i_1<i_2<\cdots<i_k=j}(-a_{i_1i_2})(-a_{i_2i_3})\cdots(-a_{i_{k-1}i_k})\alpha_{i_1}, \]
with $k$ varying from $1$ to $j$.
Thus $\nu_{c^{-1}}(s_1\cdots s_{j-1}\alpha_j)$ is 
\[-\sum_{i=1}^nE_{c^{-1}}\biggl(\alpha_i\ck,\sum_{1\le i_1<i_2<\cdots<i_k=j}(-a_{i_1i_2})(-a_{i_2i_3})\cdots(-a_{i_{k-1}i_k})\alpha_{i_1}\biggr)\rho_i.\]
This sum has nonzero terms for $i\le i_1$.
Separating out the terms for $i=i_1$, the sum becomes 
\begin{multline*}
-\sum_{1\le i_1<i_2<\cdots<i_k=j}(-a_{i_1i_2})(-a_{i_2i_3})\cdots(-a_{i_{k-1}i_k})\rho_{i_1}\\
-\sum_{1\le i<i_1<i_2<\cdots<i_k=j}a_{ii_1}(-a_{i_1i_2})(-a_{i_2i_3})\cdots(-a_{i_{k-1}i_k})\rho_i.
\end{multline*}
The terms in the two sums cancel each other out, except for the $k=1$ term in the first sum, so the expression for $\nu_{c^{-1}}(s_1\cdots s_{j-1}\alpha_j)$ collapses to $-\rho_j$ as desired.

If $\beta=s_n\cdots s_{j+1}\alpha_j$ for some $j$, then $\nu_{c^{-1}}(\tau_c\beta)=\nu_{c^{-1}}(-\alpha_j)=\rho_j$.
Replacing $c^{-1}$ by~$c$ in the argument above, one sees that $-\nu_c(\beta)$ is $\rho_j$ as well.
\end{proof}

\begin{proof}[Proof of \cref{nu thm}]
In light of \cref{DF g}, to prove the first assertion of \cref{nu thm}, we show that $\nu_c$ induces an isomorphism from $\Fan^\re_c(\RS)$ to $\DF_c$.

We first show that for every $c$-sortable element $v$, the set $\cl_c(v)$ is a real $c$-cluster.
The simple argument is follows the first paragraph of the proof of the finite-type result \cite[Theorem~8.1]{sortable}, except that instead appealing to induction on rank, we appeal to the finite-type result.
We do, however, argue by induction on the length of~$v$.
Writing $c=s_1\cdots s_n$, if $s_1\not\le v$, then $v$ is $s_1c$-sortable in the parabolic subgroup generated by $\set{s_2,\ldots,s_n}$.
By the finite-type result, $\cl_{s_1c}(v)$ is an $s_1c$-cluster, so \cref{prop:inductive_cluster}\eqref{cluster induct} says that $\cl_c(v)$ is a $c$-cluster.
If $s_1\le v$, then the length of $s_1v$ is shorter than the length of $v$, so by induction $\cl_{s_1cs_1}(s_1v)$ is an $s_1cs_1$-cluster, and therefore $\cl_c(v)$ is a $c$-cluster by \cref{prop:inductive_cluster}\eqref{cluster sigma}.

Since $\Cone_c(v)$ is the nonnegative linear span of $\nu_c(\cl_c(v))$ for any $c$-sortable element $v$, we conclude that every cone in the $c$-Cambrian fan $\F_c$ is the image, under $\nu_c$, of a cone in $\Fan_c^\re(\RS)$.

We similarly want to show that every cone in $-\F_{c^{-1}}$ is the image, under $\nu_c$, of a cone in $\Fan_c^\re(\RS)$.
The cones in $-\F_{c^{-1}}$ are the nonnegative linear spans of $-\nu_{c^{-1}}(\cl_{c^{-1}}(v))$ for $c^{-1}$-sortable elements $v$.
Thus we want to show that the set $\nu_c^{-1}(-\nu_{c^{-1}}(\cl_{c^{-1}}(v)))$ is a $c$-cluster for all $c^{-1}$-sortable elements $v$.
By the argument above, with $c^{-1}$ replacing $c$, we know that $\cl_{c^{-1}}(v)$ is a $c^{-1}$-cluster.
Since $\nu_c^{-1}(-\nu_{c^{-1}}(\cl_{c^{-1}}(v)))=\tau_{c^{-1}}\cl_{c^{-1}}(v)$ by \cref{nu nu tau}, \cref{prop:inductive_cluster}\eqref{cluster tau} implies that $\nu_c^{-1}(-\nu_{c^{-1}}(\cl_{c^{-1}}(v)))$ is a $c^{-1}$-cluster as well, and thus a $c$-cluster by \cref{c cinv prop}.

We have showed that $\nu_c^{-1}$ maps every cone in $\DF_c$ to a cone in $\Fan_c^\re(\RS)$.
Since $\nu_c^{-1}:V^*\to V$ is a homeomorphism and since by \cref{DF comp} the complement of $\DF_c$ in $V^*$ is an $(n-1)$-dimensional cone, the complement of $\nu_c^{-1}\left(\DF_c\right)$ is $(n-1)$-dimensional. 
In particular, since the maximal cones of $\Fan_c^\re(\RS)$ are $n$-dimensional by \cref{cluster properties}\eqref{Q basis}, every maximal cone of $\Fan_c^\re(\RS)$ is in the image of $\DF_c$ under $\nu_c^{-1}$.
Thus because $\nu_c^{-1}$ is a homeomorphism and because both $\DF_c$ and $\Fan_c^\re(\RS)$ are fans, we see that $\nu_c^{-1}$ induces an isomorphism from $\DF_c$ to $\Fan_c^\re(\RS)$.
Equivalently, $\nu_c$ induces an isomorphism from $\Fan_c^\re(\RS)$ to $\DF_c$.

The second assertion of \cref{nu thm} now follows because, as mentioned above, the map $x\mapsto\g(x)$ is an isomorphism from the cluster complex to the simplicial complex underlying the $\g$-vector fan.
The latter isomorphism is proved in affine type using the combinatorics of root systems and Coxeter groups \cite{camb_fan,framework,afframe}, so this entire proof occurs in that combinatorial setting.
\end{proof}

To establish \cref{denom thm}, we need the following theorem, which was conjectured as \cite[Conjecture~3.21]{framework} and proved as \cite[Proposition~9]{Rupel}.
(In fact, \cite[Proposition~9]{Rupel} only establishes \cref{nu d g} for non-initial cluster variables, but extending the theorem to include initial cluster variables is easy.)

\begin{theorem}\label{nu d g}  
Suppose $B$ is an acyclic exchange matrix with associated Coxeter element $c$ and $x$ is a cluster variable in the principal-coefficients cluster algebra associated to $B$.
Then $\g(x)=\nu_c(\d(x))$.
\end{theorem}

Combining \cref{nu thm,nu d g} (and observing that every cone in $\Fan_c^\re(\RS)$ is contained in a domain of linearity of $\nu_c$), we obtain the special case of \cref{denom thm} where $\Sigma$ has principal coefficients.
The full statement of \cref{denom thm} then holds by the following well-known observation, which is an easy consequence of \cite[Theorem~3.7]{Fomin07}.
(In the lemma, for a (labeled) seed $\Sigma$, the notation $\Sigma_i$ denotes the $i\th$ entry in the cluster in $\Sigma$.)
\begin{lemma}\label{denom coeff}
Suppose $\Sigma$ is a seed (with no conditions on coefficients) and suppose $m$ is some sequence of mutations.
If $\Sigma'$ is any seed (in any cluster algebra) with the same exchange matrix as $\Sigma$, then $\d_{\Sigma}([m(\Sigma)]_i)=\d_{\Sigma'}([m(\Sigma')]_i))$.
\end{lemma}

We have completed the proofs of \cref{nu thm,denom thm}.  

\subsection{Evidence for \cref{c p conj}}
As preparation for discussing evidence, we prove and quote some preliminary results.

By analogy with source-sink moves on Coxeter elements, we define a \newword{source-sink mutation} of an acyclic seed to be a mutation in an index $k$ such that the entries in the $k$-th colum of the exchange matrix are either all nonnegative or all nonpositive.
Source-sink mutations do not change the Cartan matrix underlying~$B$.
The following result, which lets us apply source-sink mutations to $\Sigma$ in \cref{c p conj}, is \cite[Corollary~10]{Rupel}.
The theorem holds in general, not just in finite or affine type.

\begin{theorem}\label{init d}
Suppose $\Sigma$ is an acyclic seed and $x$ is a cluster variable in $\A_\Sigma$.
If $\Sigma'$ is obtained from $\Sigma$ by a source-sink mutation corresponding to the source-sink move $c\to scs$, then $\d_{\Sigma'}(x)=\sigma_s\d_\Sigma(x)$.
\end{theorem}

As a consequence of \cref{init d}, we have the following proposition.

\begin{proposition}\label{conj acyc}
\cref{c p conj} holds when $\Sigma'$ can be obtained from $\Sigma$ by a sequence of source-sink mutations.
\end{proposition}
\begin{proof}
When $\Sigma'=\Sigma$, this is just \cref{denom thm}.
Thus \cref{init d} and \cref{compat sigma} combine to prove the proposition.
\end{proof}

Each acyclic seed $\Sigma$ defines a Cartan matrix $A$ and a Coxeter element $c$.
By \cref{denom coeff}, \cref{c p conj} holds or fails simultaneously for all $\Sigma$ corresponding to the same $A$ and $c$.
In fact, \cref{c p conj} is preserved under conjugation of $c$:
\begin{proposition}\label{c p simp}
Fix a Cartan matrix $A$ of affine type and a Coxeter element~$c$.
Suppose \cref{c p conj} holds for a seed $\Sigma$ corresponding to $A$ and $c$.
Then \cref{c p conj} holds when $\Sigma$ corresponds to $A$ and a Coxeter element conjugate to~$c$.
\end{proposition}
\begin{proof}
If $\tilde c$ is a Coxeter element conjugate to $c$, then \cref{aff conj} implies that there is a sequence of source-sink mutations taking $\Sigma$ to a seed $\tilde\Sigma$ corresponding to $A$ and $\tilde c$.
\cref{init d} and \cref{compat sigma} show that \cref{c p conj} for $A$ and $c$ implies \cref{c p conj} for $A$ and $\tilde c$.
\end{proof}

Using the surfaces model, one can prove \cref{c p conj} in affine types A and~D.
\begin{theorem}\label{conj surf}
\cref{c p conj} holds when $B$ is acyclic of type $A_{n-1}^{(1)}$ or $D_{n-1}^{(1)}$.
\end{theorem}

The proof of Theorem~\ref{conj surf} proceeds by analyzing the intersection numbers \cite[Definition~8.4]{Fomin08} of tagged arcs on the annulus and the twice-punctured disk.
By \cref{c p simp}, we need consider only one triangulation of each surface.
We omit the details, but they can be found in early arXiv versions of this paper.

Finally, we offer additional computational evidence.
The following proposition simplifies the process.
Given $\Sigma'$ and $\beta$ as in \cref{c p conj}, write $\dist(\Sigma',\beta)$ for the smallest number of mutations needed to take $\Sigma'$ to a seed containing $x(\beta)$.

\begin{proposition}\label{conj check}
Fix a Cartan matrix $A$ of affine type, a Coxeter element~$c$, and $d\ge0$.
Suppose \cref{c p conj} has been verified when $\Sigma$ corresponds to $A$ and $c$, the cluster of $\Sigma'$ has nonempty intersection with the cluster of $\Sigma$, and ${\dist(\Sigma',\beta)\le d}$.
Then \cref{c p conj} holds when $\Sigma$ corresponds to $A$ and a Coxeter element conjugate to $c$, for arbitrary $\Sigma'$ with $\dist(\Sigma',\beta)\le d$.
\end{proposition}

The proof of \cref{conj check} uses methods similar to other proofs in this section, and we omit the details.
Using this proposition, we have checked \cref{c p conj} for all rank-3 and rank-4 affine types not covered by \cref{conj acyc} or \cref{conj surf}, whenever $\dist(\Sigma',\beta)\le 100$.

{\small  \label{index}

\printnomenclature[45 pt]
}

\subsection*{Acknowledgments}
We thank the anonymous referees for many helpful comments.  
We are also grateful to Giovanni Cerulli-Irelli and Sibylle Schroll for insights on quiver representations.

\bibliographystyle{plain}
\bibliography{bibliography}

\begin{thebibliography}{10}

\bibitem{BGP}
I.~N. Bernstein, I.~M. Gelfand, and V.~A. Ponomarev.
\newblock Coxeter functors, and {G}abriel's theorem.
\newblock {\em Uspehi Mat. Nauk}, 28(2(170)):19--33, 1973.

\bibitem{Bj-Br}
Anders Bj{\"o}rner and Francesco Brenti.
\newblock {\em Combinatorics of {C}oxeter groups}, volume 231 of {\em Graduate
  Texts in Mathematics}.
\newblock Springer, New York, 2005.

\bibitem{Buan06}
Aslak~Bakke Buan, Robert Marsh, Markus Reineke, Idun Reiten, and Gordana
  Todorov.
\newblock Tilting theory and cluster combinatorics.
\newblock {\em Adv. Math.}, 204(2):572--618, 2006.

\bibitem{Buan08b}
Aslak~Bakke Buan and Robert~J. Marsh.
\newblock Denominators in cluster algebras of affine type.
\newblock {\em J. Algebra}, 323(8):2083--2102, 2010.

\bibitem{Buan05c}
Aslak~Bakke Buan, Robert~J. Marsh, Idun Reiten, and Gordana Todorov.
\newblock Clusters and seeds in acyclic cluster algebras.
\newblock {\em Proc. Amer. Math. Soc.}, 135(10):3049--3060 (electronic), 2007.
\newblock With an appendix coauthored in addition by P. Caldero and B. Keller.

\bibitem{CalderoChapoton}
Philippe Caldero and Fr\'{e}d\'{e}ric Chapoton.
\newblock Cluster algebras as {H}all algebras of quiver representations.
\newblock {\em Comment. Math. Helv.}, 81(3):595--616, 2006.

\bibitem{Caldero05}
Philippe Caldero and Bernhard Keller.
\newblock From triangulated categories to cluster algebras. {II}.
\newblock {\em Ann. Sci. \'Ecole Norm. Sup. (4)}, 39(6):983--1009, 2006.

\bibitem{CalZel}
Philippe Caldero and Andrei Zelevinsky.
\newblock Laurent expansions in cluster algebras via quiver representations.
\newblock {\em Mosc. Math. J.}, 6(3):411--429, 587, 2006.

\bibitem{CarrDevadoss}
Michael~P. Carr and Satyan~L. Devadoss.
\newblock Coxeter complexes and graph-associahedra.
\newblock {\em Topology Appl.}, 153(12):2155--2168, 2006.

\bibitem{CP}
Cesar Ceballos and Vincent Pilaud.
\newblock Denominator vectors and compatibility degrees in cluster algebras of
  finite type.
\newblock {\em Trans. Amer. Math. Soc.}, 367(2):1421--1439, 2015.

\bibitem{associahedra}
Fr{\'e}d{\'e}ric Chapoton, Sergey Fomin, and Andrei Zelevinsky.
\newblock Polytopal realizations of generalized associahedra.
\newblock {\em Canad. Math. Bull.}, 45(4):537--566, 2002.
\newblock Dedicated to Robert V. Moody.

\bibitem{Deodhar}
Vinay~V. Deodhar.
\newblock A note on subgroups generated by reflections in {C}oxeter groups.
\newblock {\em Arch. Math. (Basel)}, 53(6):543--546, 1989.

\bibitem{Dlab76}
Vlastimil Dlab and Claus~Michael Ringel.
\newblock Indecomposable representations of graphs and algebras.
\newblock {\em Mem. Amer. Math. Soc.}, 6(173):v+57, 1976.

\bibitem{Dyer}
Matthew Dyer.
\newblock Reflection subgroups of {C}oxeter systems.
\newblock {\em J. Algebra}, 135(1):57--73, 1990.

\bibitem{FeShTu12}
Anna Felikson, Michael Shapiro, and Pavel Tumarkin.
\newblock Skew-symmetric cluster algebras of finite mutation type.
\newblock {\em J. Eur. Math. Soc. (JEMS)}, 14(4):1135--1180, 2012.

\bibitem{Fomin08}
Sergey Fomin, Michael Shapiro, and Dylan Thurston.
\newblock Cluster algebras and triangulated surfaces. {I}. {C}luster complexes.
\newblock {\em Acta Math.}, 201(1):83--146, 2008.

\bibitem{FoZe03a}
Sergey Fomin and Andrei Zelevinsky.
\newblock Cluster algebras. {II}. {F}inite type classification.
\newblock {\em Invent. Math.}, 154(1):63--121, 2003.

\bibitem{FoZe03}
Sergey Fomin and Andrei Zelevinsky.
\newblock {$Y$}-systems and generalized associahedra.
\newblock {\em Ann. of Math. (2)}, 158(3):977--1018, 2003.

\bibitem{Fomin07}
Sergey Fomin and Andrei Zelevinsky.
\newblock Cluster algebras. {IV}. {C}oefficients.
\newblock {\em Compos. Math.}, 143(1):112--164, 2007.

\bibitem{GHKK}
Mark Gross, Paul Hacking, Sean Keel, and Maxim Kontsevich.
\newblock Canonical bases for cluster algebras.
\newblock {\em J. Amer. Math. Soc.}, 31(2):497--608, 2018.

\bibitem{Howlett}
Robert~B. Howlett.
\newblock Coxeter groups and {$M$}-matrices.
\newblock {\em Bull. London Math. Soc.}, 14(2):137--141, 1982.

\bibitem{IPT}
Colin Ingalls, Charles Paquette, and Hugh Thomas.
\newblock Semi-stable subcategories for {E}uclidean quivers.
\newblock {\em Proc. Lond. Math. Soc. (3)}, 110(4):805--840, 2015.

\bibitem{Kac82}
V.~G. Kac.
\newblock Infinite root systems, representations of graphs and invariant
  theory. {II}.
\newblock {\em J. Algebra}, 78(1):141--162, 1982.

\bibitem{Kac90}
Victor~G. Kac.
\newblock {\em Infinite-dimensional {L}ie algebras}.
\newblock Cambridge University Press, Cambridge, third edition, 1990.

\bibitem{Macdonald}
I.~G. Macdonald.
\newblock Affine root systems and {D}edekind's {$\eta $}-function.
\newblock {\em Invent. Math.}, 15:91--143, 1972.

\bibitem{MRZ}
Robert Marsh, Markus Reineke, and Andrei Zelevinsky.
\newblock Generalized associahedra via quiver representations.
\newblock {\em Trans. Amer. Math. Soc.}, 355(10):4171--4186, 2003.

\bibitem{McSul}
Jon McCammond and Robert Sulway.
\newblock Artin groups of {E}uclidean type.
\newblock {\em Invent. Math.}, 210(1):231--282, 2017.

\bibitem{MusPro}
Gregg Musiker and James Propp.
\newblock Combinatorial interpretations for rank-two cluster algebras of affine
  type.
\newblock {\em Electron. J. Combin.}, 14(1):Research Paper 15, 23, 2007.

\bibitem{sortable}
Nathan Reading.
\newblock Clusters, {C}oxeter-sortable elements and noncrossing partitions.
\newblock {\em Trans. Amer. Math. Soc.}, 359(12):5931--5958, 2007.

\bibitem{universal}
Nathan Reading.
\newblock Universal geometric cluster algebras.
\newblock {\em Math. Z.}, 277(1-2):499--547, 2014.

\bibitem{camb_fan}
Nathan Reading and David~E. Speyer.
\newblock Cambrian fans.
\newblock {\em J. Eur. Math. Soc. (JEMS)}, 11(2):407--447, 2009.

\bibitem{cyclic}
Nathan Reading and David~E. Speyer.
\newblock Sortable elements for quivers with cycles.
\newblock {\em Electron. J. Combin.}, 17(1):Research Paper 90, 19, 2010.

\bibitem{typefree}
Nathan Reading and David~E. Speyer.
\newblock Sortable elements in infinite {C}oxeter groups.
\newblock {\em Trans. Amer. Math. Soc.}, 363(2):699--761, 2011.

\bibitem{framework}
Nathan Reading and David~E. Speyer.
\newblock Combinatorial frameworks for cluster algebras.
\newblock {\em Int. Math. Res. Not. IMRN}, (1):109--173, 2016.

\bibitem{afframe}
Nathan Reading and David~E. Speyer.
\newblock Cambrian frameworks for cluster algebras of affine type.
\newblock {\em Trans. Amer. Math. Soc.}, 370(2):1429--1468, 2018.

\bibitem{afforb}
Nathan {Reading} and Salvatore {Stella}.
\newblock The action of a coxeter element on an affine root system.
\newblock Preprint (arXiv:1808.05090), 2017.

\bibitem{RupSteWil}
D.~{Rupel}, S.~{Stella}, and H.~{Williams}.
\newblock {Affine cluster monomials are generalized minors}.
\newblock Preprint (arXiv:1712.09143), 2017.

\bibitem{Rupel}
Dylan Rupel and Salvatore Stella.
\newblock {Some consequences of categorification}.
\newblock Preprint (arXiv:1712.08478), 2017.

\bibitem{Scherotzke}
Sarah Scherotzke.
\newblock Component clusters for acyclic quivers.
\newblock {\em Colloq. Math.}, 144(2):245--264, 2016.

\bibitem{Seven}
Ahmet~I. Seven.
\newblock Cluster algebras and semipositive symmetrizable matrices.
\newblock {\em Trans. Amer. Math. Soc.}, 363(5):2733--2762, 2011.

\bibitem{SZ}
Paul Sherman and Andrei Zelevinsky.
\newblock Positivity and canonical bases in rank 2 cluster algebras of finite
  and affine types.
\newblock {\em Mosc. Math. J.}, 4(4):947--974, 982, 2004.

\bibitem{Ste13}
Salvatore Stella.
\newblock Polyhedral models for generalized associahedra via {C}oxeter
  elements.
\newblock {\em J. Algebraic Combin.}, 38(1):121--158, 2013.

\bibitem{shih}
Shih-Wei Yang and Andrei Zelevinsky.
\newblock Cluster algebras of finite type via {C}oxeter elements and principal
  minors.
\newblock {\em Transform. Groups}, 13(3-4):855--895, 2008.

\bibitem{Zel}
Andrei Zelevinsky.
\newblock Semicanonical basis generators of the cluster algebra of type
  {$A^{(1)}_1$}.
\newblock {\em Electron. J. Combin.}, 14(1):Note 4, 5, 2007.

\end{thebibliography}

\end{document}